\documentclass[11pt]{article}
\usepackage[utf8]{inputenc}

\usepackage{general_manuscript}
\pgfplotstableread{data/diffusion_per_errors_relH1.dat}{\MsFEMtestsRel} % convergence results for various H in relative H^1 norm

\pgfplotstableread{data/diffusion_per_Gal_compare.dat}{\MsFEMcompare} % difference between Galerkin and G-ni or PG solutions in H^1 norm, not normalized

\pgfplotstableread{data/diffusion_locper_errors_relH1.dat}{\MsFEMtestsVarRel} % convergence results for various H in relative H^1 norm

\pgfplotstableread{data/diffusion_locper_Gal_compare.dat}{\MsFEMcompareVar} % difference between Galerkin and G-ni or PG solutions in H^1 norm, not normalized

\pgfplotstableread{data/diffusion_nonper_errors_relH1.dat}{\MsFEMtestsNPRel} % convergence results for various H in relative H^1 norm

\pgfplotstableread{data/diffusion_nonper_Gal_compare.dat}{\MsFEMcompareNP} % difference between Galerkin and G-ni or PG solutions in H^1 norm, not normalized

\title{Non-intrusive implementation of a wide variety of Multiscale Finite Element Methods}
\author[1,2]{R. A. Biezemans}
\author[1,2]{C. Le Bris}
\author[1,2]{F. Legoll}
\author[2,3]{A. Lozinski}
\affil[1]{\small École Nationale des Ponts et Chaussées, 6 et 8 avenue Blaise Pascal, 77455 Marne-La-Vallée Cedex 2, France}
\affil[2]{\small MATHERIALS project-team, Inria Paris, 2 rue Simone Iff, CS 42112, 75589 Paris Cedex 12, France}
\affil[3]{\small Université de Franche-Comté, CNRS, LmB, F-25000 Besançon, France}

\date{\today}

\begin{document}

\maketitle

\begin{abstract}
   Multiscale Finite Element Methods (MsFEMs) are now well-established finite element type approaches dedicated to multiscale problems. They first compute local, oscillatory, problem-dependent basis functions that generate a suitable discretization space, and next perform a Galerkin approximation of the problem on that space. We investigate here how these approaches can be implemented in a non-intrusive way, in order to facilitate their dissemination within industrial codes or non-academic environments. 
 	We develop an abstract framework that covers a wide variety of MsFEMs for linear second-order partial differential equations. Non-intrusive MsFEM approaches are developed within the full generality of this framework, which may moreover be beneficial to steering software development and improving the theoretical understanding and analysis of MsFEMs.
\end{abstract}

%%
%% Main body of the article
%%
\section{Introduction}
	%%
%%%%%%%%%%%%%%%%%%%%%%%%%%%%%%%%%%
%% Introduction to the non-intrusive MsFEM article with general MsFEM framework
%%%%%%%%%%%%%%%%%%%%%%%%%%%%%%%%%%
%%
\label{sec:introduction}
In this article, we consider highly oscillatory partial differential equations (PDEs) of the form
$
	\mathcal{L}^\varepsilon u^\varepsilon = f,
$
posed in a bounded domain $\Omega \subset \bbR^d$, where $\mathcal{L}^\varepsilon$ is a second-order linear differential operator with (possibly rough) coefficients that oscillate on a microscopic length scale of size $\varepsilon$ much smaller than the diameter of $\Omega$. (See Sec.~\ref{sec:msfem-diffusion} and~\ref{sec:gen-framework} for a complete description of the problems that we study.) We seek a numerical approximation of~$u^\varepsilon$ by applying a Galerkin approach. 

It is well-known that standard (say $\Pone$) finite element methods (FEMs) yield a poor approximation as a consequence of the highly oscillatory nature of the problem, unless a prohibitively expensive fine mesh is employed. An explicit example of this phenomenon is given, e.g., in~\cite[Example 1.1.]{altmann_numerical_2021}. Dedicated multiscale approaches have thus been introduced, which provide a reasonably accurate approximation of~$u^\varepsilon$ at a limited computational cost. Among the many multiscale approaches that have been proposed in the literature, we mention the {\em Heterogeneous Multiscale Method} (abbreviated HMM)~\cite{e_heterogeneous_2003}, the {\em Local Orthogonal Decomposition} (LOD)~\cite{malqvist_localization_2014} method, and the {\em Multiscale Finite Element Method} (MsFEM)~\cite{hou_multiscale_1997}, on which we focus here. We refer the reader to~\cite{abdulle_heterogeneous_2012,efendiev_multiscale_2009,le_bris_examples_2017,altmann_numerical_2021} for more comprehensive expositions of these multiscale methods.

The MsFEM is a finite element type approach that relies on the adaptation of the finite element space to the highly oscillatory differential operator, an idea that was first introduced in~\cite{babuska_generalized_1983}. The MsFEM was introduced in~\cite{hou_multiscale_1997}. It consists of two steps:
\begin{enumerate}
\item An ``offline'' stage, where highly oscillatory, problem-dependent basis functions are computed numerically as solutions to local problems (that mimic the reference problem on a subdomain). The local problems serve as a preprocessing step of the microstructure.
\item An ``online'' stage, where a Galerkin approximation of the reference problem, performed in the finite-dimensional space generated by the basis functions that are computed in the offline stage, is solved. This constitutes the global coupling between the local computations performed in the offline stage.
\end{enumerate}

The MsFEM approach is particularly interesting for multi-query contexts, where the PDE of interest is to be solved repeatedly for multiple right-hand sides $f$ (think e.g.~of optimization problems or time-dependent problems where the time discretization results in a PDE in space to advance from one time step to the next). In this case, the basis functions, which depend on~$\mathcal{L}^\varepsilon$, remain unchanged, so the offline stage is performed only once. The online stage solves a problem on a space of much lower dimension than a high fidelity space that fully resolves the microscale, resulting in a significant computational gain.

The specific choice of problem-dependent basis functions has led to various MsFEM variants in the literature (see, for instance, \cite{hou_multiscale_1997, efendiev_multiscale_2009, le_bris_msfem_2013, le_bris_numerical_2017}). Although their implementation is rightfully considered relatively easy~\cite{efendiev_multiscale_2009,nguyen_residual-driven_2019}, these methods dictate an intrusive workflow, i.e., substantial modifications have to be made to the code of some existing finite element software to implement the MsFEM. For instance, all automatic computations involving integrals of standard (e.g. polynomial) basis functions have to be replaced by the corresponding computations for specialized basis functions for each specific problem. We explore this further in Sec.~\ref{sec:msfem-diffusion-intrusive}.

The intrusive character of the MsFEM hinders the use of the method in many industrial contexts, because the time and tools required to adapt a legacy code that is currently in use may not be available. In this work, we propose a `non-intrusive' MsFEM strategy that uses an existing legacy code for single-scale problems (based on standard finite elements), without any modifications, to obtain an accurate resolution of oscillatory PDEs. 

The intrusiveness of the MsFEM is due to the following fact: the microstructure is preprocessed through the computation of specialized basis functions, thereby coupling the microstructure explicitly to the global numerical model. We circumvent this coupling with a novel formulation of the MsFEM basis functions. The resulting MsFEM strategy can be summarized as follows:
\begin{enumerate}
	\item In the offline stage, so-called numerical correctors are computed as the solution to problems mimicking the reference problem on a subdomain. The numerical correctors are then used to average the microstructure in the form of effective, piecewise constant coefficients, leading to an effective PDE.
	\item In the online stage, the effective PDE is solved by a standard FEM \emph{with the legacy code}.
	\item A post-processing stage is introduced to restore microscopic features in the macroscopic FEM result (obtained in the online stage) with the help of the numerical correctors computed in the offline stage.
\end{enumerate}

Our non-intrusive MsFEM approach was introduced on a prototypical example in~\cite{biezemans_non-intrusive_2023}. In this article, we extend the findings of~\cite{biezemans_non-intrusive_2023} to a large class of MsFEM variants that can be `related' to a $\Pone$ FEM, in a sense that will be made precise in Rem.~\ref{rem:p1}. This  requires the formulation of a general MsFEM framework, covering a generic definition of the MsFEM for the approximation of an abstract variational formulation of second-order linear PDEs. To the best of our knowledge, the question of how to make MsFEM approaches less intrusive has not been studied in the literature, except for the preliminary study in~\cite{biezemans_non-intrusive_2023}, and this work is a first step in that direction. We comment on the intrusiveness of other multiscale methods in Sec.~\ref{sec:noni-gen-other}. 

An overview of the contents of this article is as follows. In Sec.~\ref{sec:msfem-diffusion}, we recall the basic principles of the FEM and the MsFEM and we explain the intrusive character of the MsFEM. We also review the non-intrusive MsFEM approach that was proposed in~\cite{biezemans_non-intrusive_2023} for the simplest MsFEM variant on the example of a diffusion problem. We also highlight a link between the non-intrusive MsFEM approach and classical homogenization here. Then we summarize in Sec.~\ref{sec:gen-motivation} which properties of the MsFEM are essential for the non-intrusive workflow, and motivate the development of a general framework covering a wide variety of MsFEMs that we present in Sec.~\ref{sec:gen-framework}. We extend the non-intrusive MsFEM approach of~\cite{biezemans_non-intrusive_2023} to our general MsFEM framework in Sec.~\ref{sec:noni-gen}. The non-intrusive MsFEM approach of~\cite{biezemans_non-intrusive_2023} was found to be equivalent to a Petrov-Galerkin MsFEM (with $\Pone$ test functions). This is no longer true for all MsFEMs covered by our general framework, and we obtain two non-intrusive MsFEMs: the Petrov-Galerkin MsFEM, which is completely equivalent to its non-intrusive implementation, and an approximate version of the Galerkin MsFEM that can be implemented in a non-intrusive way. The three essential formulas for the formulation of the non-intrusive MsFEM are highlighted in special boxes, both for the diffusion problem in Sec.~\ref{sec:msfem-diffusion} and for the general framework in Sec.~\ref{sec:gen-framework} and~\ref{sec:noni-gen}. We then study the general MsFEM framework applied to diffusion problems in Sec.~\ref{sec:compare-gal-pg}, where we obtain a number of convergence results for the difference between the intrusive and non-intrusive MsFEM approaches. We conclude the article in Sec.~\ref{sec:gen-num} by a numerical comparison of the intrusive and non-intrusive MsFEM approaches for diffusion problems, in order to assess the efficiency of our approaches for cases that are not covered by the convergence results of Sec.~\ref{sec:compare-gal-pg}. Our results show that the Petrov-Galerkin MsFEM as well as the non-intrusive approximation of the Galerkin MsFEM are close to the original Galerkin MsFEM. Any possible additional error introduced by making the MsFEM non-intrusive is thus negligible.

\subsection*{Notation}
In this article, we shall adopt standard notation for Sobolev spaces. 
In particular, the dual space of~$H^1_0(\Omega)$ is denoted~$H^{-1}(\Omega)$. 
Further, for a given simplicial mesh~$\mesh$ of~$\Omega$, we use the notation~$H^1(\mesh)$ to denote the broken Sobolev space
\begin{equation*}
	H^1(\mesh)
	=
	\left\{
		u \in L^2(\Omega) \, \left\vert \,
		u\vert_K \in H^1(K) \text{ for all mesh elements } K \in \mesh \right.
	\right\}.
\end{equation*}
The standard norm for the space~$H^1(\Omega)$ is
$ \displaystyle
	\lVert u \rVert_{H^1(\Omega)}
	=
	\sqrt{
		\lVert u \rVert_{L^2(\Omega)}^2 + \lVert \nabla u \rVert_{L^2(\Omega)}^2
	}
$
and the corresponding broken norm is
$ \displaystyle
	\lVert u \rVert_{H^1(\mesh)}
	=
	\sqrt{
		\sum_{K\in\mesh} \lVert u \rVert_{H^1(K)}^2
	}.
$
The space of functions whose restriction to each element of~$\mesh$ is a polynomial of degree~$k$ is denoted~$\Pk(\mesh)$.

\section{The (intrusive) multiscale finite element method}
	\label{sec:msfem-diffusion}

%%
%%%%%%%%%%%%%%%%%%%%%%%%%%%%%%%%%%
%% A brief description of standard FEM, standard MsFEM, intrusiveness, and the non-intrusive MsFEM
%%%%%%%%%%%%%%%%%%%%%%%%%%%%%%%%%%
%%

\subsection{Discrete variational formulation}

Let $d \geq 1$ denote the space dimension of interest and let $\Omega \subset \bbR^d$ be a bounded polytope (e.g.\ a polygon in dimension $d=2$, a polyhedron in dimension $d=3$). 
%When~$\Omega$ is not a polyhedron, additional regularity assumptions on the boundary are needed to control the errors made by the finite element method when~$\Omega$ is approximated by a simplicial mesh, see~\cite{ciarlet_finite_1978}. 
Convexity of~$\Omega$ can be assumed for elliptic regularity results to hold, for which we refer to~\cite{grisvard_elliptic_1985}. This technical assumption is not necessary for the algorithmic aspects of the MsFEM that are the main focus of this article.

By way of example, we consider first the diffusion equation with homogeneous Dirichlet boundary conditions. In a second step, from Sec.~\ref{sec:gen-framework} onwards, we will also consider more general problems, and we will mention other types of boundary conditions in Sec.~\ref{sec:noni-gen-more}.
More precisely, we focus here on the boundary value problem
\begin{equation}
\left\{
\begin{IEEEeqnarraybox}[][c]{uts?s}
    \IEEEstrut
    $-{\operatorname{div}(A^\varepsilon \nabla u^\varepsilon)}$
    &$=$
    &$f$ 
    &in $\Omega$,
    \\
    $u^\varepsilon$
    &$=$
    &$0$
    &on $\partial \Omega$,
    \IEEEstrut
    \IEEEeqnarraynumspace %preserve the same alignment as obtained in an equation environment
\end{IEEEeqnarraybox}
\right. 
\label{eq:diffusion-pde}
\end{equation}
where the diffusion tensor $A^\varepsilon \in L^\infty(\Omega,\ \bbR^{d \times d})$ satisfies the uniform bounds
\begin{gather}
\begin{aligned}
	&\forall \, \xi\in\bbR^d,
    \quad
    m |\xi|^2 \leq \xi \cdot A^\varepsilon(x) \, \xi 
    \quad \text{a.e.~in } \Omega, \\
    \text{and}
    \qquad
    &\forall \, \xi,\eta \in \bbR^d,
    \quad 
    | \eta \cdot A^\varepsilon(x) \xi | \leq M \, |\xi| \, |\eta|
    \quad \text{a.e.~in } \Omega,
\end{aligned}
\label{ass:bounds}
\end{gather}
for some $M \geq m > 0$ independent of~$\varepsilon$. The right-hand side~$f$ does not vary on the microscopic scale~$\varepsilon$. We denote the diffusion tensor with a superscript $\varepsilon$ to keep in mind that $A^\varepsilon$ might be highly oscillatory on a typical length scale of size $\varepsilon$ much smaller than the diameter of $\Omega$ (assumed to be of order~1). 
No further structural assumptions on~$A^\varepsilon$ are made. In particular,~$A^\varepsilon$ need not be the rescaling of a fixed periodic matrix of the form~$A^\varepsilon(x) = A(x/\varepsilon)$. We will specialize to this periodic setting in Sec.~\ref{sec:homogenization-convergence} only to obtain convergence results, but this assumption is of no relevance for the practical implementation of the MsFEM. Let us also mention that none of the considerations in this article require symmetry of the diffusion tensor. Our development of non-intrusive MsFEMs also generalizes to linear \emph{systems} of PDEs. The analysis we provide is also expected to extend to e.g.~the system of linear elasticity up to some technicalities that we do not consider here.

For simplicity of exposition, we assume that $f\in L^2(\Omega)$ (rather than $f \in H^{-1}(\Omega)$, for which the problem~\eqref{eq:diffusion-pde} is in fact well-posed). We do so to avoid unnecessary technicalities. Our proposed non-intrusive MsFEM carries over to the more general case. For some convergence results, the condition $f \in L^2(\Omega)$ cannot be relaxed. In this case, this is also explicitly stated. 

Problem~\eqref{eq:diffusion-pde} admits a unique solution in the space $H^1_0(\Omega)$. This solution is also characterized by the variational formulation 
\begin{equation}
\left\{
\begin{IEEEeqnarraybox}[][c]{t?s}
    \IEEEstrut
    Find $u^\varepsilon \in H^1_0(\Omega)$ such that
    \\
    $a^{\varepsilon,\dif}(u^\varepsilon,\, v) = F(v)$
    &for all $v \in H^1_0(\Omega)$,
    \IEEEstrut
    \IEEEeqnarraynumspace %preserve the same alignment as obtained in an equation environment
\end{IEEEeqnarraybox}
\right.
\label{eq:diffusion-vf}
\end{equation}
where the bilinear form $a^{\varepsilon,\dif}$ and the linear form $F$ are defined, for any $u,\, v \in H^1_0(\Omega)$, by
\begin{equation}
    a^{\varepsilon, \dif}(u,\, v) = 
    \int_\Omega \nabla v \cdot A^\varepsilon \nabla u,
    \quad 
    F(v) = \int_\Omega f v.
\label{eq:diffusion-bilin-form}
\end{equation}
The coercivity hypothesis in~\eqref{ass:bounds} ensures that the bilinear form $a^{\varepsilon,\dif}$ is coercive on the space~$H^1_0(\Omega)$. Then the Lax-Milgram Theorem \cite[Theorem 5.8]{gilbarg_elliptic_2001} shows that~\eqref{eq:diffusion-vf} is indeed well-posed.

The numerical approximation of~\eqref{eq:diffusion-vf} with a finite element method starts by the introduction of a mesh~$\mesh$ for~$\Omega$. The subscript~$H$ denotes the typical size of the mesh elements. 
We assume~$\mesh$ to be a simplicial, conformal mesh. For some convergence results, we shall assume quasi-uniformity. These assumptions are standard in finite element analysis. We refer, e.g., to~\cite{quarteroni_numerical_2017, ciarlet_finite_1978, ern_theory_2004} for a general exposition and various examples. Again, these regularity properties of the mesh do not have any impact on the implementation of the MsFEM on a given mesh. The regularity plays a role only to obtain convergence results. 

A finite element method for~\eqref{eq:diffusion-pde} is obtained by restricting the equivalent formulation~\eqref{eq:diffusion-vf} to a finite-dimensional subspace of~$H^1_0(\Omega)$, typically consisting of functions that are piecewise polynomial on the mesh~$\mesh$. We suppose that we are in the regime where~$H$ is larger than or comparable to the microscale~$\varepsilon$. In this case, it is well known that a Galerkin approximation of~\eqref{eq:diffusion-vf} on, say, the standard (conforming) Lagrange $\Pone$ space on~$\mesh$ provides only a poor, not to say an incorrect approximation of $u^\varepsilon$. See~\cite[Example 1.1]{altmann_numerical_2021}, for instance, for an explicit example where the $\Pone$ approximation on a coarse mesh fails. At the same time, the use of a finite element method on a fine mesh of size $H \ll \varepsilon$ might be unfeasible from a computational point of view because of its prohibitive computational cost. To remedy this issue, we shall next introduce the \emph{multiscale finite element method} (MsFEM) \cite{hou_multiscale_1997, efendiev_multiscale_2009}.

\subsection{A simple multiscale finite element method}
\label{sec:msfem}
The MsFEM is a Galerkin approximation of~\eqref{eq:diffusion-vf} for which the approximation space is adapted in order to achieve satisfactory accuracy even on a coarse mesh. The correct choice of approximation space yields a numerical approximation that is much closer to $u^\varepsilon$ than a standard $\Pone$-approximation when~$\varepsilon$ is smaller than~$H$, and especially when~$\varepsilon$ becomes asymptotically small. To begin with, we introduce here the simplest variant of the MsFEM, which originally appeared in~\cite{hou_multiscale_1997}, before moving on to other MsFEM variants in Sec.~\ref{sec:gen-framework}.

Let $x_1,\dots,\, x_{N_0}$ be an enumeration of the interior vertices of the mesh $\mesh$, i.e., the vertices that do not lie on $\partial \Omega$. 
We denote by $\phiPone{i}$ the unique piecewise $\Pone$ function such that $\phiPone{i}(x_j) = \delta_{i,j}$ for all $1 \leq j \leq N_0$. (These are the basis functions for the standard $\Pone$ Lagrange finite element.) We define the \emph{multiscale basis functions} $\phiEps{i}$ (for $1 \leq i \leq N_0$) by 
\begin{equation}
    \forall \ K \in \mesh, 
    \quad 
    \left\{
    \begin{IEEEeqnarraybox}[][c]{uts?s}
        \IEEEstrut
        $-{\operatorname{div}(A^\varepsilon \nabla \phiEps{i})}$
        &$=$
        &$0$ 
        &in $K$,
        \\
        $\phiEps{i}$
        &$=$
        &$\phiPone{i}$
        &on $\partial K$.
        \IEEEstrut
        \IEEEeqnarraynumspace %preserve the same alignment as obtained in an equation environment
    \end{IEEEeqnarraybox}
    \right.
\label{eq:MsFEM-basis}
\end{equation}
All these problems, on each mesh element $K$, are again well-posed by coercivity of $A^\varepsilon$ and the Lax-Milgram Theorem. The functions $\phiEps{i}$ so defined belong to the global space $H^1_0(\Omega)$ because the local boundary conditions on $\partial K$ imply continuity across all mesh elements $K$. It is also immediately seen that $\phiEps{i}$ is supported by exactly the same mesh elements as $\phiPone{i}$.

\begin{remark}
    \label{rem:msfem-diff-basis-trivial}
        On each mesh element~$K$, problem~\eqref{eq:MsFEM-basis} defines at most~\mbox{$d+1$} non-trivial basis functions. Let~$i_1,\dots,i_{d+1}$ be the indices of the vertices of~$K$. It is easily inferred from~\eqref{eq:MsFEM-basis} that
        $\displaystyle
            \left. \phiEps{i_{d+1}} \right\vert_K 
            = 
            1 - \sum_{j=1}^d \left. \phiEps{i_j}\right\vert_K .
        $
    Thus, one only has to compute~$d$ basis functions by the resolution of the PDE~\eqref{eq:MsFEM-basis} on~$K$.
\end{remark}

The multiscale approximation space is defined as $V_{H,0}^\varepsilon = \operatorname{span}\{ \phiEps{i} \mid 1 \leq i \leq N_0 \}$. This is a finite-dimensional space of the same dimension as the one used for a $\Pone$ Lagrange finite element approximation on the mesh~$\mesh$. The MsFEM consists in computing the approximation $u^\varepsilon_H \in V_{H,0}^\varepsilon$ defined by the problem
\begin{equation}
     \forall \, v_H^\varepsilon \in V_{H,0}^\varepsilon,
     \qquad 
     a^{\varepsilon,\dif}{ \left(u^\varepsilon_H,\, v_H^\varepsilon \right)} = F \left(v_H^\varepsilon \right).
    \label{eq:diffusion-MsFEM}
\end{equation}
Since $V_{H,0}^\varepsilon$ is a subspace of $H^1_0(\Omega)$, the bilinear form~$a^{\varepsilon,\dif}$ is coercive on $V_{H,0}^\varepsilon$ and the discrete problem~\eqref{eq:diffusion-MsFEM} is again well-posed by virtue of the Lax-Milgram Theorem.

The computation of the multiscale basis functions~$\phiEps{i}$ is called the \emph{offline stage} of the MsFEM and only has to be carried out once if~\eqref{eq:diffusion-pde} has to be solved multiple times for different right-hand sides. Also note that all problems~\eqref{eq:MsFEM-basis} are independent of each other, and can thus be solved in parallel. Once all basis functions are known, one can compute the stiffness matrix of the MsFEM (see Sec.~\ref{sec:msfem-diffusion-intrusive} for more details), which is also part of the offline stage. In practice, the~$\phiEps{i}$ are approximated numerically on a fine mesh of~$K\in\mesh$ of mesh size $h \leq \varepsilon$ that resolves the oscillations of~$A^\varepsilon$. We omit these details here because they have no importance for the non-intrusive strategy that we shall propose later in this article.

The resolution of the global problem~\eqref{eq:diffusion-MsFEM}, each time the right-hand side~$F$ changes, is called the \emph{online stage}. The computational cost for this problem is the same as for a standard $\Pone$ approximation on the same mesh.
A further discussion of the practical implementation of the MsFEM is provided in Sec.~\ref{sec:msfem-diffusion-intrusive}. This discussion partially reproduces some elements of~\cite{biezemans_non-intrusive_2023}. We include it here to clarify and motivate the developments in the sequel.

\subsection{Intrusive workflow}
\label{sec:msfem-diffusion-intrusive}
The practical resolution of the global problem~\eqref{eq:diffusion-MsFEM} consists in the construction and resolution of the following linear system:
\begin{equation}
	\mathds{A}^\varepsilon U^\varepsilon = \mathds{F}^\varepsilon,
\label{eq:linear-system-msfem}
\end{equation}
with the stiffness matrix~$\mathds{A}^\varepsilon$ and the right-hand side~$\mathds{F}^\varepsilon$ of the linear system given by
\begin{equation}
    \forall \ 1 \leq i,\,j \leq N_0,
    \quad
    \mathds{A}^\varepsilon_{j,i} = a^{\varepsilon,\dif} \left( \phiEps{i},\, \phiEps{j} \right),
    \quad 
    \mathds{F}^\varepsilon_j = F \left( \phiEps{j} \right),
\label{eq:linear-system-msbasis}
\end{equation}
where we recall that~$N_0$ denotes the number of interior vertices of~$\mesh$.
The MsFEM approximation~$u^\varepsilon_H$ is given by
$\displaystyle
	u^\varepsilon_H 
	=
	\sum_{i=1}^{N_0} U_i^\varepsilon \phiEps{i}.
$
The MsFEM can then be written (as it is traditionally presented) as in Algorithm~\ref{alg:msfem-diff}. We use the notation $\displaystyle a^{\varepsilon,\dif}_K(u,v) = \int_K \nabla v \cdot A^\varepsilon \nabla u$ for all $u,v \in H^1(K)$ and we write $\displaystyle F_K(v) = \int_K fv$ for any $v \in L^2(K)$.

\begin{algorithm}[ht]
\caption{MsFEM approach for problem~\eqref{eq:diffusion-pde} (see comments in the text)}
\label{alg:msfem-diff}
\begin{algorithmic}[1]
        \State Construct a mesh $\mesh$ of $\Omega$, denote $N_0$ the number of internal vertices and $\mathcal{N}(n,K)$ the global index of the vertex of $K \in \mesh$ that has local index $1 \leq n \leq d+1$ in $K$
       \label{alg:msfem-diff-offl1}

    \medskip

	\State Set $\mathds{A}^\varepsilon \coloneqq 0$ and $\mathds{F}^\varepsilon \coloneqq 0$    
    
    \ForAll{$K \in \mesh$}
        \For{$1 \leq n \leq d+1$}
			\State Set $i \coloneqq \mathcal{N}(n,K)$
            \State Solve for $\left. \phiEps{i} \right\vert_K$ in~\eqref{eq:MsFEM-basis}
            \label{alg:msfem-diff-basis}
        \EndFor

	    \For{$1 \leq l \leq d+1$}
			\State Set $j \coloneqq \mathcal{N}(l,K)$
	    	\For{$1 \leq n \leq d+1$}
				\State Set $i \coloneqq \mathcal{N}(n,K)$ and $\mathds{A}^\varepsilon_{j,i} \pluseq a_K^{\varepsilon,\dif}(\phiEps{i},\, \phiEps{j})$
	            \label{alg:msfem-diff-stiffness}
			\EndFor
			\label{alg:msfem-diff-offl2}
			\State Set $\mathds{F}^\varepsilon_{j} \pluseq F_K(\phiEps{j})$
	        \label{alg:msfem-diff-onl1}
	    \EndFor
   	\EndFor

    \medskip 
    
    \State Solve the linear system $\mathds{A}^\varepsilon U^\varepsilon = \mathds{F}^\varepsilon$

    \medskip 
    
    \State Obtain the MsFEM approximation $u^\varepsilon_H = \sum\limits_{i=1}^{N_0} U^\varepsilon_i \phiEps{i}$
    \label{alg:msfem-diff-onl2}
\end{algorithmic}
\end{algorithm}

Lines~\ref{alg:msfem-diff-offl1}-\ref{alg:msfem-diff-offl2} of Algorithm~\ref{alg:msfem-diff} (resp.~\ref{alg:msfem-diff-onl1}-\ref{alg:msfem-diff-onl2}) constitute the offline (resp.~online) stage of the MsFEM. Note that the computation of the stiffness matrix~$\mathds{A}^\varepsilon$ in line~\ref{alg:msfem-diff-stiffness} only depends on the multiscale basis functions (and not on the right-hand side~$f$) and can therefore be carried out once and for all in the offline stage. Also note that, for an efficient computation of the~$\phiEps{i}$ in line~\ref{alg:msfem-diff-basis}, one should apply Rem.~\ref{rem:msfem-diff-basis-trivial}. Only the online stage is to be repeated when problem~\eqref{eq:diffusion-pde} is to be solved multiple times for various right-hand sides~$f$.

Implementing Algorithm~\ref{alg:msfem-diff} in an industrial code is challenging. Indeed, the practical implementation of any finite element method relies on (i) the construction of a mesh, (ii) the construction of the linear system associated to the discrete variational formulation and (iii) the resolution of the linear system. An efficient implementation of the second step heavily relies on the choice of the discretization space.

Regarding the construction of the linear system (performed in line~\ref{alg:msfem-diff-stiffness} of Algorithm~\ref{alg:msfem-diff}), it is by no means obvious to adapt existing finite element codes based on generic approximation spaces (for instance spaces of piecewise polynomial functions, such as the piecewise affine functions that we will introduce in Def.~\ref{def:underlying-space} below) to a different, problem-dependent choice of space such as $V_H^\varepsilon$. No analytic expressions for the basis functions~$\phiEps{i}$ are available (and thus a fine mesh should be used to approximate them), the computation of~$a_K^\varepsilon(\phiEps{i},\phiEps{j})$ should be performed by quadrature rules on the fine mesh because the integrands are highly oscillatory, one should have at hand the correspondence between element and vertex indices of the coarse mesh ($\mathcal{N}(n,K)$ in Algorithm~\ref{alg:msfem-diff}), the assembly of the global stiffness matrix $\{ \mathds{A}^\varepsilon_{j,i} \}_{1 \leq i,j \leq N_0}$ should be executed by a dedicated new piece of software, etc. To alleviate these obstacles, we introduce below a way of implementing the MsFEM that capitalizes on an existing code for solving~\eqref{eq:diffusion-pde} by a $\Pone$ approximation on $\mesh$ in the case of slowly varying diffusion coefficients. The three central identities for our approach that we aim to generalize to other MsFEMs in this article are framed in distinctive boxes.

\subsection{Effective problem on the macroscopic scale}
\label{sec:msfem-effective}
Let us consider the construction of the stiffness matrix of the MsFEM in more detail. The stiffness matrix defined in~\eqref{eq:linear-system-msbasis} requires the computation of the quantities
\begin{equation}
	\mathds{A}^\varepsilon_{j,i}
	=
	a^{\varepsilon,\dif}(\phiEps{i},\phiEps{j})
	=
	\sum_{K\in\mesh} \int_K \nabla \phiEps{j} \cdot A^\varepsilon \nabla \phiEps{i},
	\label{eq:stifness-msfem-msbasis}
\end{equation}
for all $1 \leq i,j \leq N_0$. 

Following~\cite{biezemans_non-intrusive_2023}, we rewrite the multiscale basis functions as
\begin{equation}
    \tcboxmath{
        \forall \, K \in \mesh, 
        \qquad
        \phiEps{i}
        =
        \phiPone{i} + \sum_{\alpha=1}^d 
            \left.\left(\partial_\alpha \phiPone{i}\right) \right\vert_K \VK{\alpha}
        \quad 
        \text{in } K,
    }
\label{eq:diffusion-MsFEM-Vxy-grad}
\end{equation}
for all~$1 \leq i \leq N_0$, where, for each mesh element~$K$, we define the numerical corrector $\VK{\alpha} \in H^1_0(\Omega)$ ($1\leq\alpha\leq d$) as the function supported by~$K$ that is the unique solution to the local problem
\begin{equation}
    \tcboxmath{
        \left\{
        \begin{IEEEeqnarraybox}[][c]{uts?s}
            \IEEEstrut
            $-{\operatorname{div}(A^\varepsilon \nabla \VK{\alpha})}$
            &$=$
            &$\operatorname{div}(A^\varepsilon e_ \alpha)$ 
            &in $K$,
            \\
            $\VK{\alpha}$
            &$=$
            &$0$
            &on $\partial K$.
            \IEEEstrut
            \IEEEeqnarraynumspace %preserve the same alignment as obtained in an equation environment
        \end{IEEEeqnarraybox}
        \right.
    }
    \label{eq:diffusion-MsFEM-correctors}
\end{equation}
Here, $e_\alpha$ denotes the $\alpha$-th canonical unit vector of~$\bbR^d$.
The expansion~\eqref{eq:diffusion-MsFEM-Vxy-grad} is obtained upon rewriting~\eqref{eq:MsFEM-basis} as a PDE for $\phiEps{i}-\phiPone{i}$, and then using linearity of the PDE and the fact that~$\nabla\phiPone{i}$ is constant in~$K$ to show that 
$ \displaystyle
    \sum_{\alpha=1}^d \left.\left(\partial_\alpha \phiPone{i}\right) \right\vert_K \VK{\alpha}
$ is indeed the unique solution to this PDE.

Inserting~\eqref{eq:diffusion-MsFEM-Vxy-grad} for the trial and test functions in~\eqref{eq:stifness-msfem-msbasis} and again exploiting the fact that all~$\phiPone{i}$ have piecewise constant gradients, we obtain
\begin{align*}
    \mathds{A}^\varepsilon_{j,i} 
    &= 
    \sum_{K \in \mesh}
        \sum_{\alpha,\beta=1}^d 
        \left.\left(\partial_\beta \phiPone{j}\right)\right\vert_K \left(
            \int_K \left( e_\beta + \nabla \VK{\beta} \right) \cdot A^\varepsilon \left( e_\alpha + \nabla \VK{\alpha} \right)
        \right)
        \left.\left(\partial_\alpha \phiPone{i}\right)\right\vert_K
    \\
	&=
	\sum_{K \in \mesh}
        \sum_{\alpha,\beta=1}^d 
        \left.\left(\partial_\beta \phiPone{j}\right)\right\vert_K
            \, a_K^{\varepsilon,\dif} \left(x^\alpha + \VK{\alpha}, x^\beta + \VK{\beta}\right) \,
        \left.\left(\partial_\alpha \phiPone{i}\right)\right\vert_K.
\end{align*}
Next we define the piecewise constant effective diffusion tensor $\overline{A} \in \mathbb{P}_0(\mesh,\,\bbR^{d\times d})$ by
\begin{equation}
    \left. \overline{A}_{\beta,\alpha} \right\vert_K 
    =
    \frac{1}{|K|} a^{\varepsilon,\dif}_K{ \left(x^\alpha + \VK{\alpha},\, x^\beta + \VK{\beta} \right)}
    \quad \text{for each } K \in \mesh \text{ and each } 1 \leq \alpha,\beta \leq d,
	\label{eq:diffusion-eff-msfem-gal}
\end{equation}
where $|K|$ denotes the measure of the mesh element $K$.
Then~\eqref{eq:stifness-msfem-msbasis} can be written as
\begin{equation}
    \tcboxmath{ 
        \mathds{A}^\varepsilon_{j,i} = \int_\Omega \nabla \phiPone{j} \cdot \overline{A} \, \nabla \phiPone{i}.
    }
    \label{eq:stiffness-msfem-p1basis}
\end{equation}

Motivated by~\eqref{eq:stiffness-msfem-p1basis}, we introduce the coarse-scale problem
\begin{equation}
\left\{
\begin{IEEEeqnarraybox}[][c]{uts?s}
    \IEEEstrut
    $-{\operatorname{div}\left(\,\overline{A}\, \nabla u\right)}$
    &$=$
    &$f$ 
    &in $\Omega$,
    \\
    $u$
    &$=$
    &$0$
    &on $\partial \Omega$,
    \IEEEstrut
    \IEEEeqnarraynumspace %preserve the same alignment as obtained in an equation environment
\end{IEEEeqnarraybox}
\right. 
\label{eq:diffusion-pde-effective}
\end{equation}
and its Galerkin discretization with $\Pone$ Lagrange elements: with $V_{H,0} = \operatorname{span} \left\{\phiPone{i} \mid 1 \leq i \leq N_0 \right\}$ (note that the definition of~$V_{H,0}$ will be generalized in Def.~\ref{def:DOF}), find $u_{H,0} \in V_{H,0}$ such that
\begin{equation}
  \forall \, v_H \in V_{H,0}, \quad \overline{a}^{\dif}(u_H,v_H) = F(v_H),
  \label{eq:diffusion-FEM-effective}
\end{equation}
where the linear form~$F$ is defined in~\eqref{eq:diffusion-vf} and the bilinear form~$\overline{a}^{\dif}$ is defined as
\begin{equation}
  \forall \, u, v \in H^1_0(\Omega), \quad \overline{a}^{\dif}(u,v) = \int_\Omega \nabla v \cdot \overline{A} \, \nabla u. 
  \label{eq:diffusion-form-effective}
\end{equation}
Problem~\eqref{eq:diffusion-FEM-effective} equivalently writes
\begin{equation}
  \mathds{A}^{\Pone} \, U^{\Pone} = \mathds{F}^{\Pone},
  \label{eq:linear-system-p1}
\end{equation}
with
\begin{equation}
  \forall \, 1 \leq i, j \leq N_0,
  \quad
  \mathds{A}^{\Pone}_{j,i} = \overline{a}^{\dif} \left(\phiPone{i},\phiPone{j} \right),
  \quad 
  \mathds{F}^{\Pone}_j = F \left( \phiPone{j} \right).  
  \label{eq:linear-system-effective}
\end{equation}
Comparing the expressions~\eqref{eq:stiffness-msfem-p1basis} and~\eqref{eq:linear-system-effective}, we deduce that~$\mathds{A}^\varepsilon = \mathds{A}^\Pone$. In other words:
\begin{lemma}
	\label{lem:stiffness-MsFEM-P1}
	The stiffness matrix of the MsFEM problem~\eqref{eq:diffusion-MsFEM} is identical to the stiffness matrix of the~$\Pone$ problem~\eqref{eq:diffusion-FEM-effective}.
\end{lemma}

This lemma immediately implies that the $\Pone$ problem~\eqref{eq:diffusion-FEM-effective} is well-posed, since the MsFEM~\eqref{eq:diffusion-MsFEM} itself is well-posed.

Let us point out that problems~\eqref{eq:diffusion-pde-effective} and~\eqref{eq:diffusion-FEM-effective} are defined entirely in terms of quantities that vary only on the macroscopic scale~$H$. The finite element problem~\eqref{eq:diffusion-FEM-effective} can thus be solved using a legacy code that is designed for standard FEMs.
Lemma~\ref{lem:stiffness-MsFEM-P1} then suggests including the~$\Pone$ approximation~\eqref{eq:diffusion-FEM-effective} of the effective, coarse-scale problem~\eqref{eq:diffusion-pde-effective} as an integral part of the MsFEM approach. We do so in Algorithm~\ref{alg:msfem-diff-noni} below.

The right-hand side vector~$\mathds{F}^\varepsilon$ in~\eqref{eq:linear-system-msbasis} is, in general, different from~$\mathds{F}^\Pone$ in~\eqref{eq:linear-system-effective}. Indeed, we integrate the product of $f$ with highly oscillatory basis functions in the former problem and with~$\Pone$ basis functions in the latter. The solutions~$U^\varepsilon$ and~$U^\Pone$ to~\eqref{eq:linear-system-msfem} and~\eqref{eq:linear-system-p1}, respectively, are thus different a priori.

\subsection{Non-intrusive workflow}
\label{sec:msfem-diffusion-nonin}
We propose the following non-intrusive MsFEM variant:
\begin{equation}
	\text{Set } u_H^\varepsilon = u_H + \sum_{K\in\mesh} (\partial_\alpha u_H)\vert_K \, \VK{\alpha} \in V_{H,0}^\varepsilon \text{ where } u_H \in V_{H,0} \text{ is the unique solution to~\eqref{eq:diffusion-FEM-effective}}.
\label{eq:diffusion-MsFEM-noni}
\end{equation}
The MsFEM approximation~$u_H^\varepsilon$ is well-defined, since we have seen above that problem~\eqref{eq:diffusion-FEM-effective} is well-posed.

Note that the symbol~$u^\varepsilon_H$ shall be used here and in the sequel for the solution to various MsFEMs variants to alleviate the notation. The exact MsFEM will be specified by the context. We will use distinct notation for different MsFEM variants when required for clarity.

The most efficient way to compute~$u^\varepsilon_H$ from~$u_H$ is not as stated here, however. The \emph{evaluation} of~$u_H(x)$ may require the determination of the degrees of freedom associated to the simplex~$K$ to which~$x$ belongs. This demands the use of the internal mechanisms of the legacy code that is used to compute~$u_H$. The use of the legacy code can be avoided by expanding~$u_H$ as follows. For any affine function~$\varphi$ on~$K$, we have
\begin{equation}
    \varphi(x) = \varphi \left(x_{c,K}\right) + \sum_{\alpha=1}^d \partial_\alpha \varphi \cdot \left( x^\alpha - x^\alpha_{c,K} \right)
    \quad 
    \text{on } K,
    \label{eq:P1-expansion}
\end{equation}
where $x^\alpha$ denotes the function that to a point~$x\in\Omega$ associates its $\alpha$-th coordinate, and $x_{c,K}=(x^1_{c,K},\dots,x^d_{c,K})$ is the centroid of~$K$.
If one uses the legacy code to store the values of~$u_H(x_{c,K})$ and~$\partial_\alpha u_H$ element  by element at the end of the online stage, then~$u^\varepsilon_H$ defined in~\eqref{eq:diffusion-MsFEM-noni} can be computed element by element according to
\begin{equation}
	\forall \, K \in \mesh,
	\quad 
	u^\varepsilon_H(x)
	= 
	u_H(x_{c,K}) + \sum\limits_{\alpha=1}^d \left.\left(\partial_\alpha u_H\right)\right\vert_K \left( x^\alpha - x^\alpha_{c,K} + \VK{\alpha}(x) \right)
	\quad
	\text{on } K,
\label{eq:msfem-diff-noni-post}
\end{equation}
without using the legacy code.

The above observations culminate in the computational approach presented in Algorithm~\ref{alg:msfem-diff-noni}. We can distinguish
\begin{enumerate}[label=(\arabic*)]
    \item the offline stage consisting of lines~\ref{alg:msfem-diff-noni-offl1}-\ref{alg:msfem-diff-noni-offl2},
    \item the online stage being executed entirely in line~\ref{alg:msfem-diff-noni-onl},
    \item a post-processing step in line \ref{alg:msfem-diff-noni-post}.
\end{enumerate}

\begin{algorithm}[ht]
\caption{Non-intrusive MsFEM approach for problem~\eqref{eq:diffusion-pde}}
\label{alg:msfem-diff-noni}
\begin{algorithmic}[1]
        \State Let $\mesh$ be the mesh used by the legacy code
        \label{alg:msfem-diff-noni-offl1}

    \medskip
    
    \ForAll{$K \in \mesh$}
   	\label{alg:msfem-diff-noni-local}
        \For{$1 \leq \alpha \leq d$}
            \State Solve for $\VK{\alpha}$ defined by~\eqref{eq:diffusion-MsFEM-correctors}
            \label{alg:msfem-diff-noni-offl-corr}
        \EndFor
		\State Compute $\overline{A}\vert_K$ defined by~\eqref{eq:diffusion-eff-msfem-gal}
		\label{alg:msfem-diff-noni-offl-tensors}
	\EndFor
	\label{alg:msfem-diff-noni-offl2}
	
	\medskip
    
    \State Use the legacy code to solve for~$u_H$ defined by~\eqref{eq:diffusion-FEM-effective} and to save $\left\{u_H(x_{c,K})\right\}_{K\in\mesh}$ and $\left\{ (\partial_\alpha u_H) \vert_K \right\}_{K\in\mesh, \, 1 \leq \alpha \leq d}$
    \label{alg:msfem-diff-noni-onl}

    \medskip 
    
    \State Obtain the MsFEM approximation~$u^\varepsilon_H$ by~\eqref{eq:msfem-diff-noni-post}
    \label{alg:msfem-diff-noni-post}
\end{algorithmic}
\end{algorithm}

The superiority of Algorithm~\ref{alg:msfem-diff-noni} over the classical MsFEM Algorithm~\ref{alg:msfem-diff} is that the global problem of the online stage can completely be constructed and solved by the use of a pre-existing $\Pone$ PDE solver. The only requirements for the legacy code are the functionality to provide piecewise constant diffusion coefficients to the solver and the existence of a procedure to store the value of the $\Pone$ solution and its gradient at the centroids of the mesh.
An additional advantage in the online stage is that the construction of the right-hand side~$\mathds{F}^\Pone$ (see~\eqref{eq:linear-system-effective}) for the global problem only requires a numerical quadrature on the coarse mesh and is therefore cheaper than the construction of~$\mathds{F}^\varepsilon$ (see~\eqref{eq:linear-system-msbasis}), involving the multiscale basis functions and requiring numerical quadratures at the microscale. 

The part of the offline stage that manipulates fine meshes (lines~\ref{alg:msfem-diff-noni-local}-\ref{alg:msfem-diff-noni-offl2}) and the post-processing step can be developed independently of the legacy code used in line~\ref{alg:msfem-diff-noni-onl}.
The requirement for these fine-scale solvers is that they have access to the coarse mesh $\mesh$ used by the global solver.
Note also that the local problem~\eqref{eq:diffusion-MsFEM-correctors} is only indexed by the coarse mesh element~$K$, in contrast to the local problem~\eqref{eq:MsFEM-basis} that is indexed both by the coarse mesh element~$K$ and the vertex index~$i$. For the latter problems, one has to know, for each element~$K$, the global index that corresponds to the vertices of~$K$, a piece of information that may be difficult to access in a legacy code. For the problems~\eqref{eq:diffusion-MsFEM-correctors}, this correspondence is not needed to compute~$\overline{A}$, nor for the computation of the fine-scale solution~$u^\varepsilon_H$ in~\eqref{eq:msfem-diff-noni-post}, both of which are entirely defined element-wise.

\begin{remark}[Quantities of interest]
    \label{rem:postprocessing-compute}
    In the post-processing step of Algorithm~\ref{alg:msfem-diff-noni}, it is easy to compute pointwise values of the approximation~$u^\varepsilon_H$ by~\eqref{eq:msfem-diff-noni-post} and to use these for further computational steps, such as the evaluation of the energy or other quantities of interest. This task can be carried out element wise, hence Eq.~\eqref{eq:msfem-diff-noni-post} can easily be used. (See also Rem.~\ref{rem:error-computation}.)
\end{remark}

\begin{remark}[Visualization]
    \label{rem:postprocessing-vis}
    We focus our attention here on the visualization of the MsFEM approximation~$u^\varepsilon_H$, which can be an important tool in engineering practices. Visualization requires the combination of information on neighbouring mesh elements, and this can in general not be carried out by the legacy code since it does not have access to the fine meshes used to compute the numerical correctors. Even if this were the case, the fine meshes may not yield a globally conformal mesh when combined. The question of a global visualization then becomes a complex one that requires innovations beyond the contributions of this article. Instead, we propose the following two-step visualization approach:
    \begin{itemize}
        \item One can visualize the coarse part~$u_H$ of the MsFEM approximation for a global view of the solution with the tools provided by the legacy code that is used to compute~$u_H$;
        \item The fine scale details of~$u^{\varepsilon}_H$ in regions of interest can be studied through zooms inside mesh elements, using the code that is used for computations at the microscale.
    \end{itemize}
\end{remark}

\subsection{Interpretation of the non-intrusive MsFEM}
We emphasized above that the right-hand sides of the linear system for the MsFEM in~\eqref{eq:linear-system-msbasis} and the linear system solved for the non-intrusive MsFEM in~\eqref{eq:linear-system-effective} are different in general. This motivates the comparison of the non-intrusive MsFEM approach~\eqref{eq:diffusion-MsFEM-noni} to the following Petrov-Galerkin MsFEM:
\begin{equation}
	\text{Find } u^\varepsilon_H \in V_{H,0}^\varepsilon \text{ such that } 
	a^{\varepsilon,\dif}{ \left(u_H^\varepsilon, \phiPone{j} \right)} 
	= 
	F \left( \phiPone{j} \right) \quad \text{for all } 1 \leq j \leq N_{0},
	\label{eq:diffusion-MsFEM-testP1}	
\end{equation}
based on the trial space~$V_{H,0}^\varepsilon$ and the test space~$V_{H,0}$ for both the bilinear and the linear form. The following result was shown in~\cite{biezemans_non-intrusive_2023}. 

\begin{lemma}
\label{lem:IBP-bubbles}
    The non-intrusive MsFEM variant~\eqref{eq:diffusion-MsFEM-noni} coincides with the Petrov-Galerkin MsFEM~\eqref{eq:diffusion-MsFEM-testP1}.
\end{lemma}

The non-intrusive MsFEM approach is generalized in Sec.~\ref{sec:noni-gen} after the development of a general framework to define a wide variety of MsFEMs in Sec.~\ref{sec:gen-framework}. Lemma~\ref{lem:IBP-bubbles} does not generalize to the full framework. We will see the conditions under which the non-intrusive approach leads to a Petrov-Galerkin MsFEM in Lemma~\ref{lem:IBP-bubbles-gen}.
\subsection{Relation to homogenization theory}
	\label{sec:homogenization}

%%
%%%%%%%%%%%%%%%%%%%%%%%%%%%%%%%%%%
%% A brief description of (periodic) homogenization, similarities with the non-intrusive (numerical) MsFEM approach
%%%%%%%%%%%%%%%%%%%%%%%%%%%%%%%%%%
%%

We highlight in this section the fact that many ingredients of our non-intrusive MsFEM approach are reminiscent of standard quantities of homogenization theory, or the theory of $H$-convergence, which studies the limit of a sequence of solutions~$u^\varepsilon$ to a PDE as~$\varepsilon$ tends to~$0$. This relation to $H$-convergence provides an interesting interpretation of the effective tensor~$\overline{A}$ introduced in~\eqref{eq:diffusion-eff-msfem-gal}.

Let us suppose in this section (and in this section only, except for Sec.~\ref{sec:homogenization-convergence}) that $A^\varepsilon(x) = A^\mathsf{per}(x/\varepsilon)$ for some bounded, $\bbZ^d$-periodic matrix~$A^{\mathsf{per}}$ satisfying the coercivity property in~\eqref{ass:bounds}. In this case, the sequence of matrices~$A^\varepsilon$ has a homogenized limit that is explicitly known. (An explicit characterization of the limit is not available for $H$-convergence in general.) We summarize the main results below. See, for instance,~\cite{bensoussan_asymptotic_1978, zhikov_homogenization_1994} or~\cite[Chapter 1]{allaire_shape_2002} for details on periodic homogenization.

Due to $H$-convergence of~$A^\varepsilon$, the functions $u^\varepsilon$, solution to~\eqref{eq:diffusion-pde}, converge to a limit function~$u^\star$ (weakly in~$H^1(\Omega)$, strongly in~$L^2(\Omega)$) as $\varepsilon \to 0$. The homogenized limit $u^\star$ is the solution to the homogenized equation~\eqref{eq:diffusion-hom-pde} below. 

Let $Q$ denote the unit cube of $\bbR^d$. 
We introduce the corrector functions $w_1,\dots,\,w_d \in H^1_{per}(Q)$ solution to
\begin{equation}
\left\{
\begin{IEEEeqnarraybox}[][c]{uts?s}
    \IEEEstrut
    $-{\operatorname{div}(A^\mathsf{per} \nabla w_\alpha)}$
    &$=$
    &$\operatorname{div}(A^\mathsf{per}e_\alpha)$
    &in $\bbR^d$,
    \\
    $w_\alpha$
    &is
    &$Q$-periodic,
    \IEEEstrut
    \IEEEeqnarraynumspace %preserve the same alignment as obtained in an equation environment
\end{IEEEeqnarraybox}
\right. 
\label{eq:diffusion-correctors}
\end{equation}
which uniquely defines $w_\alpha$ up to an irrelevant additive constant. The entries of the (constant) homogenized diffusion tensor~$A^\star \in \bbR^{d\times d}$ are given by
\begin{equation}
    A^\star_{\beta,\alpha} 
    =
    \int_Q (e_\beta + \nabla w_\beta) \cdot A^\mathsf{per}(e_\alpha + \nabla w_\alpha),
    \qquad
	1 \leq \alpha, \, \beta \leq d.
    \label{eq:diffusion-hom-coef}
\end{equation}
The homogenized limit $u^\star$ of $u^\varepsilon$ is the unique solution in $H^1_0(\Omega)$ to the boundary value problem
\begin{equation}
\left\{
\begin{IEEEeqnarraybox}[][c]{uts?s}
    \IEEEstrut
    $-{\operatorname{div}(A^\star \nabla u^\star)}$
    &$=$
    &$f$ 
    &in $\Omega$,
    \\
    $u^\star$
    &$=$
    &$0$
    &on $\partial \Omega$.
    \IEEEstrut
    \IEEEeqnarraynumspace %preserve the same alignment as obtained in an equation environment
\end{IEEEeqnarraybox}
\right. 
\label{eq:diffusion-hom-pde}
\end{equation}
The truncated reconstruction of $u^\star$ that is called the first-order two-scale expansion takes the form
\begin{equation}
    u^{\varepsilon,1}(x) =
    u^\star(x) +
    \varepsilon \sum_{\alpha=1}^d \partial_\alpha u^\star(x) \, w_\alpha \left(\frac{x}{\varepsilon}\right).
\label{eq:expansion-1-hom}
\end{equation}
Under suitable regularity assumptions, the difference $u^\varepsilon - u^{\varepsilon,1}$ converges to~$0$ strongly in~$H^1(\Omega)$ as $\varepsilon\to0$. This property will be used for the convergence results in Sec.~\ref{sec:homogenization-convergence}.

In the periodic setting, the expansion~\eqref{eq:expansion-1-hom} can be used to construct a numerical approximation of~$u^\varepsilon$, without the need of any computations at the fine scale. This approximation is presumably valid only in the regime of very small parameters~$\varepsilon$ and deteriorates if~$\varepsilon$ grows. Moreover, in more general settings, the corrector functions are not local nor explicit, for their definition involves a PDE posed on the whole domain~$\Omega$ and that depends on an effective tensor that is itself defined in terms of the corrector functions. Details can be found, e.g., in~\cite{allaire_shape_2002, murat_h-convergence_1997, zhikov_homogenization_1994}. This prevents the $H$-convergence theory from being directly applicable for the numerical approximation of~$u^\varepsilon$.

Numerical homogenization techniques, that draw their inspiration from the various elements above, offer an alternative for the approximation of~$u^\varepsilon$ that can be applied in much more general contexts. We can see the similarities between the corrector functions~$w_\alpha$ in~\eqref{eq:diffusion-correctors} and the numerical correctors~$\VK{\alpha}$ in~\eqref{eq:diffusion-MsFEM-correctors}. Note that the~$\VK{\alpha}$ solve problems similar to~\eqref{eq:diffusion-correctors}, but that they need to be solved at the microscale and on each mesh element~$K$. Similarly, we note the resemblance between the reconstruction~\eqref{eq:expansion-1-hom} and the definition of~$u_H^\varepsilon$ in~\eqref{eq:diffusion-MsFEM-noni}, and between the homogenized coefficient~$A^\star$ defined in~\eqref{eq:diffusion-hom-coef} and the effective macroscopic coefficient $\overline{A}$ from~\eqref{eq:diffusion-eff-msfem-gal}. However, contrary to~$A^\star$, the MsFEM quantity~$\overline{A}$ has to be computed on an element-by-element basis, and it is not necessarily constant throughout $\Omega$. 
Finally, the MsFEM analogue of the homogenized problem~\eqref{eq:diffusion-hom-pde} is the resolution of the effective macroscale problem~\eqref{eq:diffusion-pde-effective}.

\begin{example}
A very particular setting, although academic in nature and only useful for pedagogical purposes, actually leads to an MsFEM approximation that is exactly equivalent to a discretization of the periodic homogenization setting. Consider~\eqref{eq:diffusion-pde} in 2D posed on the unit square. Let us consider a mesh consisting of squares that are perfectly aligned with the periodicity of~$A^\varepsilon$. We solve the corrector problems~\eqref{eq:diffusion-MsFEM-correctors} on all square mesh elements with periodic boundary conditions and subsequently compute the effective diffusion tensor~$\overline{A}$ according to~\eqref{eq:diffusion-eff-msfem-gal}. 

In this case, the problems for the numerical correctors all reduce to~\eqref{eq:diffusion-correctors} and~$\overline{A}$ is constant and equal to the homogenized coefficient~$A^\star$ as defined by~\eqref{eq:diffusion-hom-coef}. A~$\mathbb{Q}_1$ discretization of the effective problem~\eqref{eq:diffusion-pde-effective} thus constitutes a non-intrusive MsFEM that is equivalent to the $\mathbb{Q}_1$ approximation of the homogenized equation~\eqref{eq:diffusion-hom-pde}.
\end{example}

\section{Why develop a general framework?}
	\label{sec:gen-motivation}

%%
%%%%%%%%%%%%%%%%%%%%%%%%%%%%%%%%%%
%% Introduction to MsFEM extensions OS and CR, essential elements of the non-intrusive approach, and motivations to build a general MsFEM framework
%%%%%%%%%%%%%%%%%%%%%%%%%%%%%%%%%%
%%

In the sequel we develop a general framework for a wide variety of MsFEMs in an abstract setting. We motivate here why this general framework for MsFEMs is useful.

\subsection{Local boundary conditions}
\label{sec:local-bc}

First, let us explain why various MsFEMs have been proposed in the literature. One reason is that different equations than~\eqref{eq:diffusion-pde} (e.g. advection-diffusion equations) give rise to different choices of the local problem~\eqref{eq:MsFEM-basis}, depending on which terms of the global PDE are included (see, for instance, \cite{le_bris_numerical_2017,biezemans_difficult_nodate}.) 

The other reason is that, even for the pure diffusion problem~\eqref{eq:diffusion-pde}, the choice of the basis functions defined in~\eqref{eq:MsFEM-basis} has an important drawback. The definition of the multiscale basis functions requires a choice of \emph{arbitrary} boundary conditions on the mesh element boundary~$\partial K$, since the exact boundary condition satisfied by~$u^\varepsilon$ is unknown. In~\eqref{eq:MsFEM-basis}, affine boundary conditions are imposed. In view of this choice, we shall refer to the MsFEM defined above as the `MsFEM-lin'. 

The MsFEM-lin cannot yield an accurate representation of $u^\varepsilon$ near $\partial K$ if $A^\varepsilon$ is highly oscillatory and the mesh $\mesh$ is coarse. Variations on the definition of the functions $\phiEps{i}$ have been proposed to improve the MsFEM. Here we summarize the ideas of \emph{oversampling} and of MsFEM \emph{\`{a} la Crouzeix-Raviart}, which together inspire the formulation of a general MsFEM framework in Sec.~\ref{sec:gen-framework}.

The oversampling variant of the MsFEM was introduced along with the variant based on~\eqref{eq:MsFEM-basis} at the time of its first appearance in~\cite{hou_multiscale_1997}. For this method, an oversampling domain~$S_K$ is associated to each mesh element~$K$ (details are provided in Sec.~\ref{sec:ospatch}). The problems~\eqref{eq:MsFEM-basis} are solved on the larger domain~$S_K$ rather than~$K$, so the inadequate boundary conditions are pushed away from the actual mesh elements. To construct the multiscale basis functions, the resulting functions on~$S_K$ are restricted to the actual mesh elements~$K$ and suitably combined around each vertex~$x_i$. The new multiscale basis functions oscillate on~$\partial K$ if the oversampling patch is taken large enough. We note that, in general, this strategy leads to discontinuous basis functions. Hence, the finite element space obtained is no longer conforming.

The MsFEM with Crouzeix-Raviart type boundary conditions for the local problems (which we shall abbreviate as `MsFEM-CR') was introduced in \cite{le_bris_msfem_2013}. 
It uses basis functions associated to the edges of the mesh (in contrast to the MsFEM-lin presented above, and its oversampling variant, where basis functions are associated to the vertices of the mesh). A typical basis function satisfies the following on~$\partial K$: the flux through each face of~$K$ is constant, and the constants are determined by the condition that the average of the basis function be~1 over one particular face and~0 over all other faces. 
Again, this is a way to avoid imposing any conditions on the trace of the basis function directly. The multiscale functions can thus be oscillatory on the faces of the mesh.
As is the case for oversampling methods, the resulting finite element space is nonconforming.

All of these variations, applied to any MsFEM for linear second-order PDEs, are covered by the general MsFEM framework that we develop in Sec.~\ref{sec:gen-framework}.

\subsection{The non-intrusive approach}

The intrusiveness of the specific MsFEM-lin variant introduced in Sec.~\ref{sec:msfem} is exemplary for all MsFEMs described in Sec.~\ref{sec:local-bc}. It turns out that the non-intrusive MsFEM approach introduced in~\cite{biezemans_non-intrusive_2023} and recalled in Sec.~\ref{sec:msfem-diffusion-nonin} can also be generalized to all these MsFEM variants. We summarize the key ingredients that allow for the formulation of the non-intrusive MsFEM approach of Algorithm~\ref{alg:msfem-diff-noni} (corresponding to the identities in boxes in Sec.~\ref{sec:msfem-effective}).

The non-intrusive MsFEM follows from the expansion~\eqref{eq:diffusion-MsFEM-Vxy-grad}, namely the expression of the multiscale basis function as a $\Pone$ basis function and a linear combination of numerical correctors that are fully localized. We note that
\begin{itemize}
    \item the full localization of the numerical correctors defined in~\eqref{eq:diffusion-MsFEM-correctors} allows the preprocessing of the microstructure independently of the global approximation indices related to the finite element method;
    \item the expansion~\eqref{eq:diffusion-MsFEM-Vxy-grad} follows from the fact that $\nabla \phiPone{i}$ is piecewise constant combined with linearity of the local problems~\eqref{eq:MsFEM-basis};
    \item the stiffness matrix can be formulated in terms of a piecewise constant effective diffusion tensor in~\eqref{eq:stiffness-msfem-p1basis} thanks to full localization of the corrector functions, the piecewise constant gradient of~$\phiPone{i}$ in the expansion~\eqref{eq:diffusion-MsFEM-Vxy-grad} and bilinearity of the global problem~\eqref{eq:diffusion-vf}.
\end{itemize}

These observations provide the main structure of the general framework. First, we choose an underlying, low-dimensional space of \emph{piecewise affine} functions to which the MsFEM is associated (Def.~\ref{def:underlying-space}). This will be the standard conforming Lagrange space of order 1 (for the MsFEM-lin), or the Crouzeix-Raviart space of order 1 (for the MsFEM-CR). Second we need to formulate the local problems for the numerical correctors (Def.~\ref{def:corr-dofe} and Def.~\ref{def:corr-dofc}). This involves the definition of oversampling patches (for MsFEMs with oversampling, Def.~\ref{def:ospatch}), and an extension of the notion of degrees of freedom to define the boundary conditions for the numerical correctors (Def.~\ref{def:DOF}, \ref{def:DOF-os-lin} and~\ref{def:DOF-os-cr}) on oversampling patches. It is then possible to define the multiscale basis functions as a generalization of~\eqref{eq:diffusion-MsFEM-Vxy-grad} (see Def.~\ref{def:msbasis-OS}) and finally to define the MsFEM for our general framework in Def.~\ref{def:gen-MsFEM}.

\begin{remark}
    \label{rem:p1}
    We note that our development of non-intrusive MsFEM approaches relies to a great extent on the fact that~\eqref{eq:diffusion-MsFEM-Vxy-grad}, and its generalization~\eqref{eq:MsFEM-Vxy} in the general framework developed below, provide a description of the multiscale basis functions in terms of $\Pone$ basis functions, without the need of higher-order functions, in a linear manner. Higher-order MsFEMs can be found in~\cite{allaire_multiscale_2005, hesthaven_high-order_2014} (see also~\cite{legoll_msfem_2022}). Possible analogues of~\eqref{eq:MsFEM-Vxy} for such MsFEMs and the subsequent techniques to design a non-intrusive MsFEM variant are more involved and may be the topic of future work. See~\cite{biezemans_difficult_nodate}.
\end{remark}

\subsection{Other motivations for the general framework}

Besides a unified formulation of our non-intrusive MsFEM approach, our general framework can also be beneficial to concrete code development for the MsFEM. Common features among various multiscale methods have previously been used to design flexible and efficient software for the implementation of such methods on the DUNE platform~\cite{bastian_generic_2008, bastian_generic_2008-1} within the \textsc{Exa-Dune} project~\cite{bastian_exa-dune_2014}. For example, the distribution of local problems over multiple processors and subsequent coupling in a global problem are handled by designated software components~\cite{bungartz_advances_2016}. Our work may contribute to the efficient implementation of all MsFEMs covered by our general framework in such a project and similar endeavours yet to come. 

When formulating the general framework, we also clarify a few practical matters that are often left pending in the various research articles we are aware of. In particular, we give a rigorous definition of the oversampling procedure near the boundary~$\partial \Omega$ of the global domain.

As we explore the general framework, we will also propose an MsFEM variant that has not yet appeared in the literature: the MsFEM-CR combined with the oversampling technique (see Example~\ref{ex:msfem-cr-basis}). We hope that our framework may also further the development of new MsFEM variants in an attempt to improve on the shortcomings of the methods known today.

Finally, the present study may also uncover a deeper understanding of MsFEMs by paving the way to a unified convergence analysis of different variants. This work is currently in preparation.
\section{Abstract definition of the MsFEM}
	\label{sec:gen-framework}

%%
%%%%%%%%%%%%%%%%%%%%%%%%%%%%%%%%%%
%% Definition of the general MsFEM framework including MsFEM-lin, MsFEM-CR, with and without oversampling
%%%%%%%%%%%%%%%%%%%%%%%%%%%%%%%%%%
%%

We develop here a general framework for multiscale finite element methods. The ultimate aim is to generalize the key identities of Sec.~\ref{sec:msfem-effective}. This is done in Def.~\ref{def:corr-dofe} and~\ref{def:corr-dofc} for the numerical correctors introduced in~\eqref{eq:diffusion-MsFEM-correctors}, and in Def.~\ref{def:msbasis-OS} for the expansion~\eqref{eq:diffusion-MsFEM-Vxy-grad} of the multiscale basis functions. This allows the reformulation of the linear system of the MsFEM as the linear system of an effective problem in~\eqref{eq:gen-FEM-effective} (for a Petrov-Galerkin MsFEM) and~\eqref{eq:gen-FEM-effective-gal} (for a Galerkin MsFEM) in Sec.~\ref{sec:noni-gen}. The other notions introduced in this section, although rather technical and abstract, are necessary tools to capture a wide variety of MsFEMs in our general framework.

\subsection{The continuous problem}
The abstract variational problem for our general MsFEM framework is as follows.  
Let $a^\varepsilon$ be a continuous bilinear form on $H^1(\Omega) \times H^1(\Omega)$.
We are interested in the solution to the problem
\begin{equation}
	\text{Find } u^\varepsilon \in H^1_0(\Omega) \text{ such that }
	a^\varepsilon(u^\varepsilon, v) = F(v)
	\text{ for any } v \in H^1_0(\Omega),
	\label{eq:gen-pb}
\end{equation}
where~$F$ is defined as in~\eqref{eq:diffusion-bilin-form} for any $f\in L^2(\Omega)$. To ensure well-posedness of~\eqref{eq:gen-pb}, we suppose that the bilinear form~$a^\varepsilon$ is coercive on~$H^1_0(\Omega)$. 
The bilinear form~$a^\varepsilon$ may contain coefficients that oscillate on a microscopically small scale.

The oversampling and Crouzeix-Raviart variants of the MsFEM introduced in Sec.~\ref{sec:local-bc} show that we need to accommodate for approximation spaces with discontinuities at the interfaces. This requires some additional assumptions on the formulation of the abstract problem. 
We suppose that the bilinear form~$a^\varepsilon$ is in fact defined on the broken Sobolev space $H^1(\mesh) \times H^1(\mesh)$. More precisely, we assume that we can represent it as $a^\varepsilon = \sum\limits_{K\in\mesh} a^\varepsilon_K$, where, for each $K \in \mesh$, $a^\varepsilon_K$ is a continuous bilinear form defined on $H^1(K) \times H^1(K)$.

To ensure well-posedness of MsFEMs, which may use nonconforming approximation spaces, coercivity on~$H^1_0(\Omega)$ may be insufficient. Therefore, we add the following coercivity hypothesis for the bilinear forms~$a_K^\varepsilon$:
\begin{equation}
\begin{IEEEeqnarraybox}[][c]{s}
	\IEEEstrut
	for all $K \in \mesh$, there exists $\alpha_K > 0$ such that
	\\
	\quad
	$
	\forall \, u \in H^1(K),
	\quad
	a_K^\varepsilon(u,u)
		\geq
	\alpha_K \, \lVert \nabla u \rVert_{L^2(K)}^2.
	$
	\IEEEstrut
	\IEEEeqnarraynumspace %preserve the same alignment as obtained in an equation environment
\end{IEEEeqnarraybox}
\label{eq:gen-pb-coer}
\end{equation}
In order to perform a convergence analysis, one also has to assume that the~$\alpha_K$ are bounded from below by some~$\tilde{\alpha}>0$ that does not depend on~$H$. We provide convergence results in Sec.~\ref{sec:compare-gal-pg} for the pure diffusion problem~\eqref{eq:diffusion-pde}, in which case we have $\alpha_K=m$ from~\eqref{ass:bounds}. 

As an example, the introductory problem~\eqref{eq:diffusion-pde} with the associated bilinear form~$a^{\varepsilon,\dif}$ is covered by this framework as is made explicit in Example~\ref{ex:diffusion-vf-gen} below. Other second-order PDEs that fit in our abstract variational formulation are given in Example~\ref{ex:gen-vfs1}. 

\begin{example}
\label{ex:diffusion-vf-gen}
The diffusion problem~\eqref{eq:diffusion-pde} is covered by the abstract variational formulation above. Indeed, we can set
\[
	a^\varepsilon = a^{\varepsilon,\dif},
\]
where~$a^{\varepsilon,\dif}$ is the bilinear form defined in~\eqref{eq:diffusion-bilin-form}. Further, we define 
\[
	a^\varepsilon_K(u,v)
	=
	a^{\varepsilon,\dif}_K(u,v)
	\coloneqq 
 	\int_K \nabla v \cdot A^\varepsilon \nabla u,
\]
for all~$u,v \in H^1(\mesh)$, so that we have indeed $a^{\varepsilon,\dif} = \sum\limits_{K\in\mesh} a^{\varepsilon,\dif}_K$. Clearly, each~$a^{\varepsilon,\dif}_K$ satisfies~\eqref{eq:gen-pb-coer} with $\alpha_K=m$ the coercivity constant from~\eqref{ass:bounds}.
\end{example}

\begin{example}
The reaction-advection-diffusion equation,
\begin{equation*}
 -{\operatorname{div}(A^\varepsilon \nabla u^\varepsilon)} + b\cdot\nabla u^\varepsilon + \sigma u^\varepsilon= f,
\end{equation*}
with a divergence-free advection field $b:\Omega \mapsto \bbR^d$ and a non-negative reaction coefficient $\sigma : \Omega \mapsto \bbR$,
can be modelled (under some regularity hypotheses that we do not state here) with the bilinear forms
\begin{equation*}
	a^{\varepsilon}_K(u,v)
	=
	\int_K \nabla v \cdot A^\varepsilon \nabla u
	+ v \, b \cdot \nabla u + \sigma u v.
\end{equation*}
However, these bilinear forms~$a_K^\varepsilon$ do not satisfy~\eqref{eq:gen-pb-coer} even though the bilinear form $a^\varepsilon$ is coercive on~$H^1_0(\Omega)$. To this end, a skew-symmetrized formulation of the transport term can be used. 
The skew-symmetrized formulation uses the bilinear form
\begin{equation}
	a^{\varepsilon}_K(u,v)
	=
	\int_K \nabla v \cdot A^\varepsilon \nabla u
	+ \frac{1}{2} v \, b \cdot \nabla u 
    - \frac{1}{2} u \, b \cdot \nabla v
    + \sigma u v,
    \label{eq:ex-skew-symmetric}
\end{equation}
which does satisfy~\eqref{eq:gen-pb-coer}. Assumption~\eqref{eq:gen-pb-coer} is used for proving well-posedness of the MsFEM in Lemma~\ref{lem:MsFEM-OS-well-posedness}, but note that both choices for~$a_K^{\varepsilon}$ mentioned here can be studied in practice. 
We refer e.g.~to~\cite{john_nonconforming_1997, le_bris_multiscale_2019,biezemans_difficult_nodate} for more details. 
Within the general MsFEM framework, $b$ and~$\sigma$ are allowed to be highly oscillatory, and this may impact the specific MsFEM strategy to be preferred.
\label{ex:gen-vfs1}
\end{example}

\subsection{Piecewise affine structure}
In Sec.~\ref{sec:msfem}, we have seen that the relation between multiscale basis functions and piecewise affine functions is essential for the development of our non-intrusive MsFEM. For the MsFEM definition in the general framework, we start by choosing such a structure in the following definition.

\begin{definition}
Let a mesh~$\mesh$ be given. The \textbf{underlying \ $\Pone$ space} for the MsFEM, denoted~$V_H$, is one of the following two spaces: the Lagrange approximation space
\begin{equation*}
    V_H^{L} = 
    \{ v \in \Pone(\mesh) \mid v \text{ is continuous on } \Omega \},
\end{equation*}
in which case we shall refer to the associated MsFEM as the MsFEM-lin, or the Crouzeix-Raviart approximation space
\begin{equation*}
    V_H^{CR} = 
    \left\{ v \in \Pone(\mesh) \mid \forall\ K \in \mesh, \ \forall e \in \mathcal{F}(K) \text{ such that } e \subset \Omega : \ \int_e \llbracket v \rrbracket = 0 \right\},
\end{equation*}
in which case the associated MsFEM shall be called the MsFEM-CR. We use the notation $\mathcal{F}(K)$ for the set of faces of~$K$ and $\llbracket v \rrbracket$ denotes the jump of~$v$ over the face~$e$. The space $V_H^L$ is a subspace of~$H^1(\Omega)$, but~$V_H^{CR}$ is not. Note that no restrictions apply on faces lying on~$\partial\Omega$. 
\label{def:underlying-space}
\end{definition}

We note that the underlying $\Pone$ space has the following property: if $v \in V_H$ is piecewise constant on the mesh~$\mesh$, then~$v$ is constant in~$\Omega$. Contrary to the space~$V_H^L$, functions in the Crouzeix-Raviart space~$V_H^{CR}$ are discontinuous in general. They are continuous, however, at the centroids of all faces of the mesh. 

For standard finite elements, the notion of degrees of freedom allows to characterize any finite element function. The idea of the MsFEM is to preserve this notion of degrees of freedom (in a suitable way made precise below) in the definition of a multiscale approximation space, while adapting the piecewise affine structure to the microstructure of the PDE.
We formalize this notion for the two underlying $\Pone$ spaces that we introduced in Def.~\ref{def:underlying-space}. The definition involves an arbitrary simplex $K$, which is typically an element of the mesh~$\mesh$, or an associated oversampling patch (for the oversampling technique of the MsFEM) that we shall define in Def.~\ref{def:ospatch}. The latter is not always a simplex, and we extend Def.~\ref{def:DOF} to such oversampling patches in Def.~\ref{def:DOF-os-lin} and~\ref{def:DOF-os-cr}.

\begin{definition}
    \label{def:DOF}
    A \textbf{degree of freedom operator} (\DOF~operator)~$\Gamma$ associates to any simplex $K \subset \bbR^d$ and $v \in \Pone(K)$ a vector $\Gamma(K,v) \in \bbR^{d+1}$, whose components are called the \textbf{degrees of freedom} of~$v$ on~$K$, in such a way that the application $v \mapsto \Gamma(K,v)$ is a linear bijection from~$\Pone(K)$ to~$\bbR^{d+1}$. More precisely, $\Gamma(K,\cdot)$ will denote in the sequel one of the following two operators:
    \begin{enumerate}
        \item (\DOF~operator for the MsFEM-lin.) Let $x_{0},\dots,x_{d}$ denote the vertices of~$K$. We set
        \begin{equation*}
            \forall \, v \in \Pone(K),
            \qquad
            \Gamma^L(K,v) = \left( v(x_0),\dots,v(x_d) \right).
        \end{equation*}
        For~$K \in \mesh$, the degree of freedom $[\Gamma^{L}(K,\cdot)]_j$ is said to be \textbf{associated to the boundary} if, for all $v \in \Pone(K)$, $[\Gamma^{L}(K,v)]_j = v(x)$ for a vertex~$x$ of the mesh that lies on~$\partial \Omega$.

        \item (\DOF~operator for the MsFEM-CR.) Let $e_0,\dots,e_d$ denote the faces of~$K$. We set
        \begin{equation*}
            \forall \, v \in \Pone(K),
            \qquad
            \Gamma^{CR}(K,v) = \left( \frac{1}{|e_0|} \int_{e_0} v,\dots,\frac{1}{|e_d|} \int_{e_d} v \right).
        \end{equation*}
        For~$K \in \mesh$, the degree of freedom $[\Gamma^{CR}(K,\cdot)]_j$ is said to be \textbf{associated to the boundary} if, for all $v \in \Pone(K)$, $\displaystyle [\Gamma^{CR}(K,v)]_j = \frac{1}{|e|} \int_{e} v$ for a face~$e$ of the mesh that lies on~$\partial \Omega$.
    \end{enumerate}
    
    The \ \textbf{$\Pone$ test space} is defined as
    \begin{equation*}
        V_{H,0} =
        \left\{
            v \in V_H \ \left\vert \ 
            \begin{aligned}
                &\forall \, K \in \mesh, \ \forall \, 1 \leq j \leq d+1, [\Gamma(K,v)]_j=0 \text{ if the degree} \\
                &\text{of freedom } [\Gamma(K,\cdot)]_j \text{ is associated to the boundary} \ 
            \end{aligned}
            \right.
        \right\}.
    \end{equation*}
\end{definition}

The $\Pone$ test space is used in practice to approximate the subspace~$H^1_0(\Omega)$ of~$H^1(\Omega)$. 
The degrees of freedom are defined element per element and are thus local. 
Global properties of the underlying $\Pone$ space~$V_H$ are most easily made explicit through the identification of a basis for~$V_H$.

% \begin{definition}
%     The functions~$\phiPone{1},\dots,\phiPone{N} \in V_H$ are called \textbf{$\Pone$ basis functions} for the underlying $\Pone$ space~$V_H$ with degree of freedom~$\Gamma$ if they form a basis of~$V_H$ and satisfy the following property:
%     \begin{align*}
%         &\text{For all } K\in\mesh, \text{ we have } \Gamma \left( K,\phiPone{i} \right) \in \{0,1\}^{d+1} \text{ and if, for some } K\in\mesh \text{ and } 1\leq i\leq N, \\
%         &\text{it holds that } \left[\Gamma \left( K,\phiPone{i} \right)\right]_k=1 \text{ for some } 1 \leq k \leq d+1, \text{ then }  \left[\Gamma \left( K,\phiPone{j} \right)\right]_k=0 \text{ for all } j \neq i.
%     \end{align*}
% \end{definition} 

\begin{definition}
\label{def:underlying-basis}
Let~$V_H$ be an underlying $\Pone$ space as in Def.~\ref{def:underlying-space}, and let~$\Gamma$ be the associated \DOF~operator. We shall denote by~$N$ the dimension of~$V_H$. The \textbf{$\Pone$ basis functions} $\phiPone{1},\dots,\phiPone{N}$ are defined as follows:
\begin{itemize}
    \item For the MsFEM-lin, let $x_1,\dots,x_N$ be an enumeration of the (internal and boundary) vertices of~$\mesh$. Then $\phiPone{i}$ is defined by $\phiPone{i}(x_j) = \delta_{i,j}$ for all $1 \leq i,j \leq N$.
    \item For the MsFEM-CR, let $e_1,\dots,e_N$ be an enumeration of the (internal and boundary) faces of~$\mesh$. Then $\phiPone{i}$ is defined by $\displaystyle \frac{1}{|e_j|} \int_{e_j} \phiPone{i} = \delta_{i,j}$ for all $1 \leq i,j \leq N$.
\end{itemize}
In both cases, these functions form a basis of the corresponding space~$V_H$ of Def.~\ref{def:underlying-space}.
\end{definition}

\subsection{Local problems}

\subsubsection{Oversampling patches}
\label{sec:ospatch}
To replace the (standard) underlying $\Pone$ space by a space of the same (low) dimension, adapted to the microstructure of~$a^\varepsilon$, we associate to each mesh element $K \in \mesh$ an oversampling patch. It serves to avoid imposing artificial, non-oscillatory boundary conditions on~$K$ directly when computing numerical correctors to process the microstructure.

\begin{definition}
    \label{def:ospatch}
    Let $K\in\mesh$ be any mesh element and let $S_K'$ be a simplex obtained from $K$ by homothety around the centroid of~$K$ with homothety ratio~$\rho\geq1$. The \textbf{oversampling patch}~$S_K$ is defined as~$S_K = S_K' \cap \Omega$.
\end{definition}

See Fig.~\ref{fig:os} for an illustration of the construction of oversampling patches in dimension~2. In this work, we allow for the trivial homothety ratio $\rho=1$. In this case, the patch~$S_K$ coincides with~$K$.

We will call an MsFEM \textbf{without oversampling} an MsFEM for which all oversampling patches satisfy~$S_K=K$. Otherwise, the MsFEM is called an MsFEM \textbf{with oversampling}. We speak simply of an MsFEM when there are no assumptions on the oversampling patches.

\begin{figure}[ht]
    \centering
    \begin{subfigure}{.4\textwidth}
            \centering
        %%
%%%%%%%%%%%%%%%%%%%%%%%%%%%%%%%%%%
%% An oversampling patch that does not touch the boundary of the domain
%%%%%%%%%%%%%%%%%%%%%%%%%%%%%%%%%%
%%

\begin{tikzpicture}

\tikzmath{ %adjustable parameters
    \xb = 0; \xt = 4.5; 
    \ncells =4;
    \gridsize = (\xt-\xb)/\ncells;
    \os = 2.35; %over-sampling ratio
} 

\coordinate (botLeft) at (\xb,\xb);
\coordinate (topRight) at (\xt,\xt);

\draw[step=\gridsize cm,gray,very thin] (botLeft) grid (topRight);
\draw[very thick] (botLeft) rectangle (topRight);
\node[anchor=south west] at (\xt,\xb){$\partial \Omega$} ;
    
\foreach \x in {0,...,\ncells} {
    \tikzmath{
        \xs = \xb + \x * \gridsize;
        \ys = \xb;
        \xe = \xt;
        \ye = \xt - \x * \gridsize;
    }
    \draw[gray, very thin] (\xs,\ys) -- (\xe,\ye);
}
\foreach \x in {-\ncells,...,-1} {
    \tikzmath{
        \xs = \xb;
        \ys = \xb - \x * \gridsize;
        \xe = \xt + \x * \gridsize;
        \ye = \xt;
    }
    \draw[gray, very thin] (\xs,\ys) -- (\xe,\ye);
}

\tikzmath{ %definition of the triangle to be oversampled
    \choosex = 2; \choosey = 2;
    \tax = \xb + \choosex * \gridsize;
    \tay = \xb + \choosey * \gridsize;
    \tbx = \xb + (\choosex+1) * \gridsize;
    \tby = \tay;
    \tcx = \tbx;
    \tcy = \xb + (\choosey+1) * \gridsize;
    % K =     c
    %        /| 
    %       / |
    %      /  |
    %     a___b,    a = (tax,tay) etc.
}

\tikzmath{ %definition of the oversampling patch
    \xbar = (\tax+\tbx+\tcx)/3;
    \ybar = (\tay+\tby+\tcy)/3;
    \tosax = \xbar + \os*(\tax-\xbar);
    \tosay = \ybar + \os*(\tay-\ybar);
    \tosbx = \xbar + \os*(\tbx-\xbar);
    \tosby = \ybar + \os*(\tby-\ybar);
    \toscx = \xbar + \os*(\tcx-\xbar);
    \toscy = \ybar + \os*(\tcy-\ybar);
}

\filldraw[fill=blue!40, draw=black,thick] (\tosax,\tosay) -- (\tosbx,\tosby) -- (\toscx,\toscy) -- cycle;

\filldraw[fill=white, draw=black,thick] (\tax,\tay) -- (\tbx,\tby) -- (\tcx,\tcy) -- cycle;

\node[anchor=south east] at (\tbx,\tby){\textcolor{blue!15!red!90!black}{$K$}} ;
\node[anchor=south east] at (\tosbx,\tosby){\textcolor{blue!15!red!90!black}{${S_K}$}} ;

\end{tikzpicture}
        \caption{}
        \label{fig:os-inside}
    \end{subfigure}
    \hspace{0.08\textwidth}
    \begin{subfigure}{.45\textwidth}
            \centering
        %%
%%%%%%%%%%%%%%%%%%%%%%%%%%%%%%%%%%
%% An oversampling patch that is not a simplex
%%%%%%%%%%%%%%%%%%%%%%%%%%%%%%%%%%
%%

\begin{tikzpicture}

\tikzmath{
% changeable parameters
    \xb = 0; \xt = 4.5; 
    \ncells = 3;
    \gridsize = (\xt-\xb)/\ncells;
    \os = 2.35; %oversampling ratio
} 

\coordinate (botLeft) at (\xb,\xb);
\coordinate (topRight) at (\xt,\xt);

\draw[step=\gridsize cm,gray,very thin] (botLeft) grid (topRight);

\foreach \x in {0,...,\ncells} {
    \tikzmath{
        \xs = \xb + \x * \gridsize;
        \ys = \xb;
        \xe = \xt;
        \ye = \xt - \x * \gridsize;
    }
    \draw[gray, very thin] (\xs,\ys) -- (\xe,\ye);
}
\foreach \x in {-\ncells,...,-1} {
    \tikzmath{
        \xs = \xb;
        \ys = \xb - \x * \gridsize;
        \xe = \xt + \x * \gridsize;
        \ye = \xt;
    }
    \draw[gray, very thin] (\xs,\ys) -- (\xe,\ye);
}

\tikzmath{ %definition of the triangle to be oversampled
    \choosex = \ncells-1; \choosey = \ncells-1;
    \tax = \xb + \choosex * \gridsize;
    \tay = \xb + \choosey * \gridsize;
    \tbx = \tax;
    \tby = \xb + (\choosey+1) * \gridsize;
    \tcx = \xb + (\choosex+1) * \gridsize;
    \tcy = \tby;
    % K = b___c
    %     |  /
    %     | /
    %     |/
    %     a,    a = (tax,tay) etc.
}

\tikzmath{ %definition of the oversampling domain before intersecting with \Omega
    \xbar = (\tax+\tbx+\tcx)/3;
    \ybar = (\tay+\tby+\tcy)/3;
    \tosax = \xbar + \os*(\tax-\xbar);
    \tosay = \ybar + \os*(\tay-\ybar);
    \tosbx = \xbar + \os*(\tbx-\xbar);
    \tosby = \ybar + \os*(\tby-\ybar);
    \toscx = \xbar + \os*(\tcx-\xbar);
    \toscy = \ybar + \os*(\tcy-\ybar);
}

\tikzmath{ %definition of the oversampling patch
    \tpax = \tosax;
    \tpay = \tosay;
    \tpbx = \tosbx;
    \tpby = \xt;
    \tpcx = \xt;
    \tpcy = \xt;
    \tpdx = \xt;
    \tpdy = \tpay - \tpax + \xt;
}

\filldraw[fill=blue!40, opacity=0.5, thin, draw=black,thick] (\tosax,\tosay) -- (\tosbx,\tosby) -- (\toscx,\toscy) -- cycle;
\filldraw[fill=blue!40, draw=black,thick] (\tpax,\tpay) -- (\tpbx,\tpby) -- (\tpcx,\tpcy) -- (\tpdx,\tpdy) -- cycle;
\filldraw[fill=white, draw=black,thick] (\tax,\tay) -- (\tbx,\tby) -- (\tcx,\tcy) -- cycle;

\node[anchor=north west] at (\tbx,\tby){\textcolor{blue!15!red!90!black}{$K$}} ;
\node[anchor=north west] at (\tosbx,\tosby){\textcolor{blue!15!red!90!black}{${S_K'}$}} ;
\node[anchor=south west, yshift=0.4*\gridsize cm] at (\tpax,\tpay){\textcolor{blue!15!red!90!black}{${S_K}$}} ;

\draw[very thick] (\xb,\xt) -- (topRight) -- (\xt,\xb) ;
\node[anchor=south west] at (\xt,\xb){$\partial \Omega$} ;

\end{tikzpicture}
        \caption{}
        \label{fig:os-boundary}
    \end{subfigure}
    \caption{Oversampling patches for MsFEM in 2D. Left: the patch for the mesh element~$K$ is obtained from $K$ by homothety. Right: The triangle $S_K'$ partially lies outside the domain $\Omega$ and the oversampling patch~$S_K$ is not homothetic to~$K$. It is not even a triangle.}
    \label{fig:os}
\end{figure}
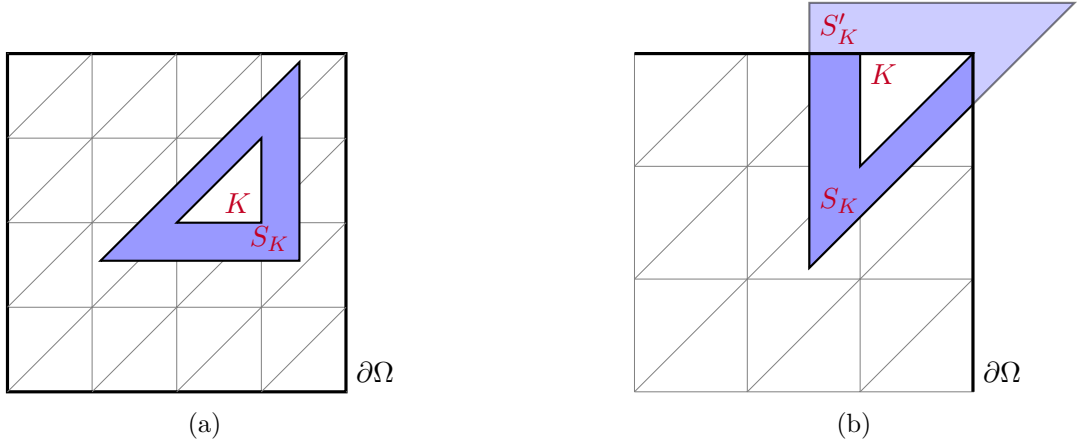

For most mesh elements~$K$, the patch~$S_K$ in Def.~\ref{def:ospatch} is a simplex. However, for mesh elements close to the boundary~$\partial \Omega$, alternative constructions should be considered. We have not found any explicit description of such a construction in the literature. This complicates the reproducibility of the method as well as a rigorous convergence analysis. The precise definitions of this section provide a first step to address these issues. A fully rigorous convergence analysis of the MsFEM with oversampling as described here is the subject of ongoing investigations~\cite{biezemans_difficult_nodate}.

\subsubsection{Degrees of freedom on oversampling patches}
\label{sec:dof-ospatch}
Definition~\ref{def:DOF} provides the definition of \DOF~operators on any simplex. For the MsFEM, we wish to compute multiscale functions on oversampling patches~$S_K$, in which case Def.~\ref{def:DOF} may be insufficient. We illustrated this in Fig.~\ref{fig:os-boundary}. Indeed, the number of vertices/faces of the oversampling patch may be larger than~$d+1$. In order to associate a multiscale basis function to every $\Pone$ basis function, we still need a notion of \DOF~operator such that $\Gamma(S_K,\cdot)$ is a linear bijection from~$\Pone(S_K)$ to~$\bbR^{d+1}$. Therefore, we extend the definition of the degrees of freedom operators~$\Gamma^L$ and~$\Gamma^{CR}$ in Def.~\ref{def:DOF-os-lin} and~\ref{def:DOF-os-cr}.

\begin{definition}
    \label{def:DOF-os-lin}
    Let $K\in\mesh$ and let~$S_K$ be its associated oversampling patch. Let~$x_0,\dots,x_d$ be a selection of~$d+1$ distinct vertices of~$S_K$. We define the \DOF~operator~$\Gamma^L$ by 
    \begin{equation*}
        \forall \, v \in \Pone(S_K),
        \qquad
        \Gamma^L(S_K,v) = \left( v(x_0),\dots,v(x_d) \right).
    \end{equation*}
\end{definition}

We note that any choice of~$d+1$ nodal values unequivocally characterizes an affine function on~$S_K$. Hence, $\Gamma^L(S_K,\cdot)$ is indeed a bijection. Now the precise choice of the vertices in Def.~\ref{def:DOF-os-lin} is unimportant, because~$\Gamma^L(S_K,\cdot)$ will only be used in the sequel to describe the trace of~$\Pone$ functions on~$\partial S_K$ in boundary value problems. For any~$\Pone$ function, this trace is uniquely defined by its values in~$d+1$ distinct vertices of~$S_K$. Finally, when~$S_K$ is a simplex, it has only~$d+1$ vertices and Def.~\ref{def:DOF-os-lin} reduces to Def.~\ref{def:DOF}.

\medskip

To generalize the notion of degrees of freedom for the Crouzeix-Raviart space to non-simplicial patches, we need to introduce some additional notation. 
On the boundary of a non-simplicial oversampling patch, we can identify some faces that collapse to a single vertex if we shrink~$S_K$ to~$K$. We call these faces the additional faces and denote the set containing them by~$\mathcal{F}_a(S_K)$. The other faces of~$S_K$ are referred to as the dilated faces, collected in the set~$\mathcal{F}_d(S_K)$. 
When the patch~$S_K$ does not touch~$\partial \Omega$, we have~$\mathcal{F}_d(S_K)=\mathcal{F}(S_K)$ and~$\mathcal{F}_a(S_K) = \emptyset$.
In Fig.~\ref{fig:os-CR}, for example, the additional faces are exactly those faces that lie on~$\partial\Omega$. This is not always the case, as is illustrated by Fig.~\ref{fig:os-lin}.

For the definition of~$\Gamma^{CR}(S_K,\cdot)$, we shall rely on the existence of~$d+1$ dilated faces, because we need $\Gamma^{CR}(S_K,\cdot)$ to be a bijection between $\Pone(S_K)$ and~$\bbR^{d+1}$. This imposes a constraint on the choice of the homothety ratio used to construct~$S_K$. For example, in the case of Fig.~\ref{fig:os-CR}, the lower right dilated face falls outside~$\Omega$ if the homothety ratio is too large, and the oversampling patch~$S_K$ only has two dilated faces (edges here) and two additional faces. We do not consider this case hereafter.
    
\begin{figure}[ht]
    \centering
    \begin{subfigure}{.4\textwidth}
            \centering
        %%
%%%%%%%%%%%%%%%%%%%%%%%%%%%%%%%%%%
%% An oversampling patch with two additional faces
%%%%%%%%%%%%%%%%%%%%%%%%%%%%%%%%%%
%%

\begin{tikzpicture}

\tikzmath{
% changeable parameters
    \xb = 0; \xt = 4.5; 
    \ncells = 2;
    \gridsize = (\xt-\xb)/\ncells;
    \os = 2.35; %over-sampling ratio
} 

\coordinate (botLeft) at (\xb,\xb);
\coordinate (topRight) at (\xt,\xt);

\draw[step=\gridsize cm,gray,very thin] (botLeft) grid (topRight);

\foreach \x in {0,...,\ncells} {
    \tikzmath{
        \xs = \xb + \x * \gridsize;
        \ys = \xb;
        \xe = \xt;
        \ye = \xt - \x * \gridsize;
    }
    \draw[gray, very thin] (\xs,\ys) -- (\xe,\ye);
}
\foreach \x in {-\ncells,...,-1} {
    \tikzmath{
        \xs = \xb;
        \ys = \xb - \x * \gridsize;
        \xe = \xt + \x * \gridsize;
        \ye = \xt;
    }
    \draw[gray, very thin] (\xs,\ys) -- (\xe,\ye);
}

\tikzmath{ %definition of the triangle to be oversampled
    \choosex = \ncells-1; \choosey = 0;
    \drawcorr = 0.015;
    \tax = \xb + \choosex * \gridsize;
    \tay = \xb + \choosey * \gridsize;
    \taydraw = \tay+\drawcorr*\gridsize;
    \tbx = \tax;
    \tby = \xb + (\choosey+1) * \gridsize;
    \tcx = \xb + (\choosex+1) * \gridsize;
    \tcxdraw = \tcx-\drawcorr*\gridsize;
    \tcy = \tby;
    % K = b___c
    %     |  /
    %     | /
    %     |/
    %     a,    a = (tax,tay) etc.
}

\tikzmath{ %definition of the oversampling domain before intersecting with \Omega
    \xbar = (\tax+\tbx+\tcx)/3;
    \ybar = (\tay+\tby+\tcy)/3;
    \tosax = \xbar + \os*(\tax-\xbar);
    \tosay = \ybar + \os*(\tay-\ybar);
    \tosbx = \xbar + \os*(\tbx-\xbar);
    \tosby = \ybar + \os*(\tby-\ybar);
    \toscx = \xbar + \os*(\tcx-\xbar);
    \toscy = \ybar + \os*(\tcy-\ybar);
}

\tikzmath{ %definition of the oversampling patch
    \tpax = \tosax;
    \tpay = \xb;
    \tpbx = \tosbx;
    \tpby = \tosby;
    \tpcx = \xt;
    \tpcy = \toscy;
    \tpdx = \tcx; \tpdy = \tcx + \tosay - \tosax;
        % Solve y=ax+b with a=(toscy-tosay)/(toscx-tosax)=1, b=tosay-a*tosax, x=tcx
    \tpex = \tay - \tosay + \tosax; \tpey = \tay;
        % Solve y=ax+b with y=\tay
}

\fill[blue!40] (\tpax,\tpay) -- (\tpbx,\tpby) -- (\tpcx,\tpcy) -- (\tpdx,\tpdy) -- (\tpex,\tpey) -- cycle;
\filldraw[fill=white, draw=black] (\tax,\taydraw) -- (\tbx,\tby) -- (\tcxdraw,\tcy) -- cycle;
\draw[densely dashed, blue!15!red!80!black,very thick] (\tpax,\tpay) -- (\tpbx,\tpby) -- (\tpcx,\tpcy) (\tpdx,\tpdy) -- (\tpex,\tpey);

\node[anchor=north west] at (\tbx,\tby){\textcolor{black}{$K$}} ;
\node[anchor=south west, yshift=0.3*\gridsize cm] at (\tpax,\tpay){\textcolor{black}{${S_K}$}} ;

\draw[very thick] (topRight) -- (\xt,\xb) -- (botLeft) ;
\node[anchor=south west] at (\xt,\xb){$\partial \Omega$} ;

\end{tikzpicture}
        \caption{}
        \label{fig:os-CR}
    \end{subfigure}
    \hspace{0.1\textwidth}
    \begin{subfigure}{.4\textwidth}
            \centering
        %%
%%%%%%%%%%%%%%%%%%%%%%%%%%%%%%%%%%
%% An oversampling patch with a dilated face lying on the boundary of the domain
%%%%%%%%%%%%%%%%%%%%%%%%%%%%%%%%%%
%%

\begin{tikzpicture}

\tikzmath{
% changeable parameters
    \xb = 0; \xt = 4.5; 
    \ncells = 2;
    \gridsize = (\xt-\xb)/\ncells;
    \os = 2.35; %over-sampling ratio
} 

\coordinate (botLeft) at (\xb,\xb);
\coordinate (topRight) at (\xt,\xt);

\draw[step=\gridsize cm,gray,very thin] (botLeft) grid (topRight);

\foreach \x in {0,...,\ncells} {
    \tikzmath{
        \xs = \xb + \x * \gridsize;
        \ys = \xb;
        \xe = \xt;
        \ye = \xt - \x * \gridsize;
    }
    \draw[gray, very thin] (\xs,\ys) -- (\xe,\ye);
}
\foreach \x in {-\ncells,...,-1} {
    \tikzmath{
        \xs = \xb;
        \ys = \xb - \x * \gridsize;
        \xe = \xt + \x * \gridsize;
        \ye = \xt;
    }
    \draw[gray, very thin] (\xs,\ys) -- (\xe,\ye);
}

\tikzmath{ %definition of the triangle to be oversampled
    \choosex = \ncells-1; \choosey = \ncells-1;
    \drawcorr = 0.015;
    \tax = \xb + \choosex * \gridsize;
    \tay = \xb + \choosey * \gridsize;
    \tbx = \tax;
    \tby = \xb + (\choosey+1) * \gridsize;
    \tcx = \xb + (\choosex+1) * \gridsize;
    \tcxdraw = \tcx - \drawcorr*\gridsize;
    \tcy = \tby;
    % K = b___c
    %     |  /
    %     | /
    %     |/
    %     a,    a = (tax,tay) etc.
}

\tikzmath{ %definition of the oversampling domain before intersecting with \Omega
    \xbar = (\tax+\tbx+\tcx)/3;
    \ybar = (\tay+\tby+\tcy)/3;
    \tosax = \xbar + \os*(\tax-\xbar);
    \tosay = \ybar + \os*(\tay-\ybar);
    \tosbx = \xbar + \os*(\tbx-\xbar);
    \tosby = \ybar + \os*(\tby-\ybar);
    \toscx = \xbar + \os*(\tcx-\xbar);
    \toscy = \ybar + \os*(\tcy-\ybar);
}

\tikzmath{ %definition of the oversampling patch
    \tpax = \tosax;
    \tpay = \tosay;
    \tpbx = \tosbx;
    \tpby = \xt;
    \tpcx = \xt;
    \tpcy = \xt;
    \tpdx = \xt;
    \tpdy = \tpay - \tpax + \xt;
}

\fill[blue!40] (\tpax,\tpay) -- (\tpbx,\tpby) -- (\tpcx,\tpcy) -- (\tpdx,\tpdy) -- cycle;
\filldraw[fill=white, draw=black] (\tax,\tay) -- (\tbx,\tby) -- (\tcxdraw,\tcy) -- cycle;
\draw[densely dashed, blue!15!red!80!black,very thick] (\tpdx,\tpdy) -- (\tpax,\tpay) -- (\tpbx,\tpby) -- (\tpcx,\tpcy);

\node[anchor=north west] at (\tbx,\tby){\textcolor{black}{$K$}} ;
\node[anchor=south west, yshift=0.3*\gridsize cm] at (\tpax,\tpay){\textcolor{black}{${S_K}$}} ;

\draw[very thick] (\xb,\xt) -- (\tpbx,\tpby);
\draw[very thick] (topRight) -- (\xt,\xb) ;
\node[anchor=south west] at (\xt,\xb){$\partial \Omega$} ;

\end{tikzpicture}
        \caption{}
        \label{fig:os-lin}
    \end{subfigure}
    \hspace{0.1\textwidth}
    \caption{Non-simplicial oversampling patches in 2D. The dilated edges of the patch~$S_K$, those that `correspond' to the edges of the original triangle~$K$, are dashed and drawn in red.}
    \label{fig:os-patch-functions}
\end{figure}
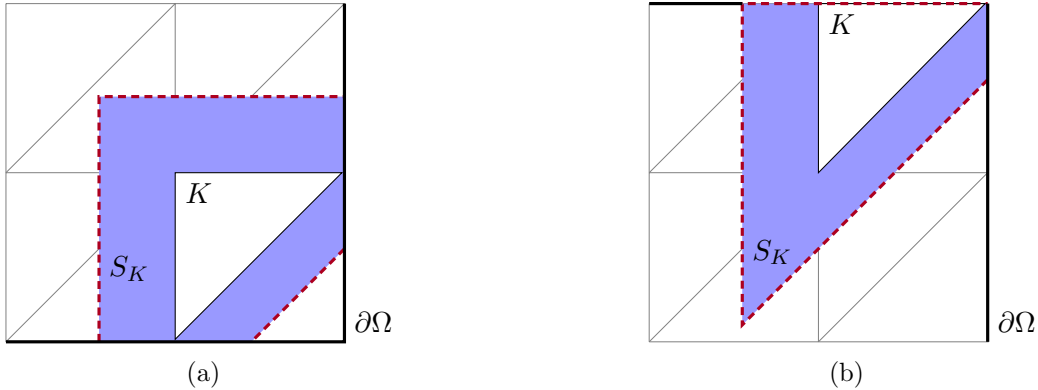

\begin{definition}
    \label{def:DOF-os-cr}
    Let $K\in\mesh$ and let~$S_K$ be its associated oversampling patch. We assume that~$S_K$ has $d+1$ dilated faces, and we denote them by $e_0,\dots,e_d$. We define the \DOF~operator~$\Gamma^{CR}$ by 
    \begin{equation*}
        \forall \, v \in \Pone(S_K),
        \qquad
        \Gamma^{CR}(S_K,v) = \left( \frac{1}{|e_0|} \int_{e_0} v,\dots,\frac{1}{|e_d|} \int_{e_d} v \right).
    \end{equation*}
\end{definition}

When~$S_K$ is a simplex, we have~$\mathcal{F}_d(S_K) = \mathcal{F}(S_K)$, and Def.~\ref{def:DOF-os-cr} coincides with the respective elements of Def.~\ref{def:DOF}.

\subsubsection{Numerical correctors: first oversampling strategy}
We now provide the precise assumptions under which we will consider local problems, i.e., the analogues of~\eqref{eq:MsFEM-basis} defining the MsFEM-lin basis functions and the definition of the numerical correctors in~\eqref{eq:diffusion-MsFEM-correctors}. In fact, since the numerical correctors play an essential role in the construction of non-intrusive MsFEM approaches, we define the numerical correctors first and use them to define the multiscale basis functions in Def.~\ref{def:msbasis-OS}.

We discuss two definitions of the numerical correctors, corresponding to two ways to define the oversampling technique for the MsFEM. 
The functional settings for these constructions are provided by Def.~\ref{def:sampling-space-OS-dofe} and~\ref{def:sampling-OS-DOFc}. These definitions involve a `sampling space', whose name is inspired by the idea that only a limited number of local problems will be solved to encode the microstructure of the PDE in the numerical model. The choice of sampling space has to accommodate for the boundary conditions that one wishes to impose on the numerical correctors and basis functions (e.g.~essential or natural; see Examples~\ref{ex:msfem-lin-dofe} and~\ref{ex:msfem-cr-dofe}).

\begin{definition}
    \label{def:sampling-space-OS-dofe}
	Let~$K\in\mesh$, let~$S_K$ be its associated oversampling patch and let~$\Gamma$ be a \DOF~operator from Def.~\ref{def:DOF}, \ref{def:DOF-os-lin} or~\ref{def:DOF-os-cr}. A subspace~$V_K$ of~$H^1(S_K)$ and bilinear form~$s^\varepsilon_K : V_K \times V_K \to \bbR$ are called \textbf{sampling space} and \textbf{sampling form}, respectively, if they satisfy the following:
	\begin{enumerate}
	\item the space $V_K$ contains the space of affine functions~$\Pone(S_K)$;
	\label{def:sampling-space-affine-OS}
	\item the operator~$\Gamma(S_K,\cdot)$ is well-defined on~$V_K$;
	\item the \textbf{\DOF-extended local problem}: find $v \in V_K$ such that 
	\begin{equation}
		\left\{
		\begin{IEEEeqnarraybox}[][c]{uts}
			\IEEEstrut
			$s_K^\varepsilon(v, w)$
			&$=$
			&$\langle g, w\rangle$ 
			for all $w \in V_{K,0}$,
			\\
			$\Gamma(S_K,v)$
			&$=$
			&$\text{given}$,
			\IEEEstrut
			\IEEEeqnarraynumspace %preserve the same alignment as obtained in an equation environment
		\end{IEEEeqnarraybox}
		\right.
		\label{eq:sampling-space-pde-OS-dofe}
	\end{equation}
	has a unique solution for any $ g \in (H^1(S_K))'$. Here, $
	V_{K,0} = \{ w \in V_K \mid \Gamma(S_K,w) = 0\}
	$ is the \textbf{sampling test space}. 
	\end{enumerate}
\end{definition}

Problem~\eqref{eq:sampling-space-pde-OS-dofe} is called `\DOF-extended' because the degrees of freedom, controlling the boundary conditions associated to the local problem, are imposed on the oversampling patch~$S_K$ rather than the (generally smaller) mesh element~$K$. 

The sampling form~$s_K^\varepsilon$ shall be used to encode the oscillations of the bilinear form~$a^\varepsilon$ and thus the microstructure of the problem in the multiscale finite element functions. 
There is some flexibility in choosing the sampling form; one may choose to include all the same terms as those in the bilinear form~$a^\varepsilon_K$ of the original problem~\eqref{eq:gen-pb}, or only some of them.
When the MsFEM was first proposed in~\cite{hou_multiscale_1997}, it was suggested that~$s_K^\varepsilon$ should include those terms that correspond to the highest-order terms of the PDE that is to be solved. In the context of the advection-diffusion equation, one may thus choose to include in our MsFEM framework only the diffusion terms, or both the diffusion and advection terms. Both options have been studied e.g.~in~\cite{le_bris_numerical_2017,le_bris_multiscale_2019}.

\medskip

In the functional setting of Def.~\ref{def:sampling-space-OS-dofe}, the generalization of~\eqref{eq:diffusion-MsFEM-correctors} to define the numerical correctors for the general MsFEM framework is as follows.

\begin{definition} 
\label{def:corr-dofe}
For all $K\in\mesh$, for any~$0\leq\alpha\leq d$, we introduce the function~$\VSK{\alpha}{1} \in V_{K,0}$ as the unique solution to the \textbf{corrector problem}
\begin{equation}
    \tcboxmath[colback=white, colframe=black]{
        \forall \, w \in V_{K,0},
        \quad
        s_K^\varepsilon \left( \VSK{\alpha}{1}, w \right) =
            \begin{cases}
                \IEEEstrut
                {-s_K^\varepsilon} \left( 1, w \right)
                & \text{if } \alpha=0,
                \\
                    {-s_K^\varepsilon} \left( x^\alpha - x^\alpha_{c,K}, w \right)
                    & \text{if } 1\leq\alpha\leq d.
            \end{cases}
    }
\label{eq:MsFEM-gen-correctors-dofe}
\end{equation}
The \textbf{\DOF-extended numerical corrector}~$\VKOS{\alpha}{1}$ is defined as the restriction of~$\VSK{\alpha}{1}$ to~$K$, extended to all of~$\Omega$ by~$0$.
\end{definition}

Note that the above definition introduces one more numerical corrector than introduced in~\eqref{eq:diffusion-MsFEM-correctors} (namely the corrector for~$\alpha=0$). The precise definition of the numerical correctors is chosen such that the analogous expansion of~\eqref{eq:diffusion-MsFEM-Vxy-grad} for the general framework (see~\eqref{eq:MsFEM-Vxy}) leads to a PDE for the multiscale basis functions analogous to~\eqref{eq:MsFEM-basis}; we show this in Lemma~\ref{lem:msbasis-OS-vf-dofe} and (for a second oversampling strategy introduced below) in Lemma~\ref{lem:msbasis-OS-vf-dofc}. In the following example, we see that Def.~\ref{def:corr-dofe} is indeed a generalization of the numerical correctors defined by~\eqref{eq:diffusion-MsFEM-correctors} in Sec.~\ref{sec:msfem-effective}.

\begin{example}[MsFEM-lin for diffusion problems]
    We consider $V_H = V_H^{L}$ and $\Gamma = \Gamma^L$ from Def.~\ref{def:DOF}.
    For the diffusion problem~\eqref{eq:diffusion-pde}, we have~$a^\varepsilon = a^{\varepsilon,\dif}$ and we set~$s_K^\varepsilon = a_K^{\varepsilon,\dif}$ (see Example~\ref{ex:diffusion-vf-gen}).
    The sampling space for the MsFEM-lin is defined as 
    \begin{equation*}
        V_K 
        =
        V_K^L
        \coloneqq
        \left\{
            v \in H^1(S_K) \ \mid \ \exists \, w \in \Pone(S_K) \text{ such that } v\vert_{\partial S_K} = w\vert_{\partial S_K}
        \right\}.
    \end{equation*}
    Then the sampling test space~$V^{L}_{K,0}$ is the space~$H^1_0(S_K)$. 
    In this case, it holds $a_K^{\varepsilon,\dif}(1,w)=0$ for all $w \in V^{L}_{K,0}$.
    Consequently, the \DOF-extended numerical corrector~$\VKOS{0}{1}$ is identically equal to~$0$; we obtain indeed exactly~$d$ numerical correctors as in Sec.~\ref{sec:msfem-effective}.
    For the non-trivial numerical correctors, Def.~\ref{def:corr-dofe} corresponds to the weak formulation of the following boundary value problem:
    \begin{equation}
        -{\operatorname{div}(A^\varepsilon \nabla \VSK{\alpha}{1})} = \operatorname{div}(A^\varepsilon e_\alpha) \
        \text{ in } S_K,
        \quad
        \VSK{\alpha}{1} = 0 \
        \text{ on } \partial S_K,
        \label{eq:msfem-lin-dofe}
    \end{equation}
    which is clearly well-posed.
    \label{ex:msfem-lin-dofe}
\end{example}

\begin{example}[MsFEM-CR for diffusion problems]
    Taking~$a^\varepsilon$, $a^\varepsilon_K$ and~$s_K^\varepsilon$ as in the previous example, we construct the MsFEM-CR with the sampling space~$V_K^{CR} \coloneqq H^1(S_K)$. 
    With $V_H = V_H^{CR}$ and $\Gamma = \Gamma^{CR}$ from Def.~\ref{def:DOF}, the corrector problem~\eqref{eq:MsFEM-gen-correctors-dofe} for~$\alpha=0$ reduces to~$\VKOS{0}{1}=0$, as in Example~\ref{ex:msfem-lin-dofe}.
    For $1 \leq \alpha \leq d$, the \DOF-extended numerical corrector~$\VKOS{\alpha}{1}$ is obtained from the boundary value problem:
    \begin{equation}
        \left\{
        \begin{IEEEeqnarraybox}[][c]{uts?s}
            \IEEEstrut
            $-{\operatorname{div}(A^\varepsilon \nabla \VSK{\alpha}{1})}$
            &$=$
            &$\operatorname{div}(A^\varepsilon e_\alpha)$ 
            &in $S_K$,
            \\
            $\vec{n} \cdot A^\varepsilon \nabla \VSK{\alpha}{1}$
            &$=$
            &$-{\vec{n} \cdot A^\varepsilon e_\alpha}$
            &on each $h \in \mathcal{F}_a(S_K)$, \\
            $\vec{n} \cdot A^\varepsilon \nabla \VSK{\alpha}{1}$
            &$=$
            &$c_h - \vec{n} \cdot A^\varepsilon e_\alpha $
            &on each $h \in \mathcal{F}_d(S_K)$, \\
            $\displaystyle \frac{1}{|h|} \int_h \VSK{\alpha}{1}$
            &$=$
            &$0$
            &for each $h \in \mathcal{F}_d(S_K)$,
            \IEEEstrut
            \IEEEeqnarraynumspace %preserve the same alignment as obtained in an equation environment
        \end{IEEEeqnarraybox}
        \right.
        \label{eq:msfem-cr-dofe}
    \end{equation}
    where~$\vec{n}$ denotes the outward unit vector on~$\partial S_K$ and $c_h$ is a constant whose value is uniquely determined by the above problem. 
    We note that the condition for the flux on the additional faces of~$S_K$ is entirely determined by the right-hand side in~\eqref{eq:MsFEM-gen-correctors-dofe}, whereas the flux on the dilated faces of~$S_K$ involves an additional constant, due to the fact that the test functions in~$V_{K,0}^{CR}$ cannot take arbitrary values on the dilated faces. Indeed, their mean vanishes on these faces according to Def.~\ref{def:sampling-space-OS-dofe}.
    
    When $S_K=K$ and when the faces of~$K$ do not lie on~$\partial \Omega$, this corresponds to the setting of the original MsFEM-CR defined in~\cite{le_bris_msfem_2013}. The latter work also provides an alternative characterization of the multiscale Crouzeix-Raviart space.

    When a face~$e$ of $K$ lies on~$\partial\Omega$, the basis functions that we will define below do not satisfy~$\phiEps{e}=0$ on~$e$, but only satisfy a weak boundary condition in the average sense on~$e$ (and so does the corresponding MsFEM approximation to~\eqref{eq:gen-pb} defined below). This does not correspond to the original definition of the MsFEM-CR in~\cite{le_bris_msfem_2013, le_bris_msfem_2014}. The MsFEM-CR with local boundary conditions as defined here was studied in~\cite{degond_crouzeix-raviart_2015, muljadi_nonconforming_2015, jankowiak_non-conforming_2018}.
    \label{ex:msfem-cr-dofe}
\end{example}

\begin{remark}
    In both Examples~\ref{ex:msfem-lin-dofe} and~\ref{ex:msfem-cr-dofe}, the numerical corrector~$\VKOS{0}{1}$ vanishes, because $s^\varepsilon_K(1,w)=0$ for all~$w$ in the sampling test space. This is no longer the case e.g.~for an MsFEM for advection-diffusion problems in which the sampling problem uses the skew-symmetrized bilinear form defined in~\eqref{eq:ex-skew-symmetric}. In this case, the numerical corrector~$\VKOS{0}{1}$ does not vanish. In the corresponding effective numerical scheme that we will derive in~\eqref{eq:gen-MsFEM-eff-form-Gal}, this leads to a term of order~0 even if such a term is not present in the advection-diffusion equation itself.
    \label{rem:VK0}
\end{remark}

When $S_k=K$ (i.e., in the absence of oversampling), the \DOF~operator allows us to prescribe certain continuity properties on the faces of the mesh elements~$K$. More precisely, when the MsFEM-lin with \DOF~operator~$\Gamma^L$ is employed, the numerical correctors~$\VKOS{\alpha}{1}$ vanish at the vertices of the mesh, and, with the correct choice of sampling space (see Example~\ref{ex:msfem-lin-dofe}), they vanish on all faces of~$K$ and are thus continuous on~$\Omega$. When the MsFEM-CR with \DOF~operator~$\Gamma^{CR}$ is considered, we obtain weak continuity of the numerical correctors over all faces of the mesh. The definition of the multiscale basis functions that we give below (see Def.~\ref{def:msbasis-OS}, in the vein of the expansion~\eqref{eq:diffusion-MsFEM-Vxy-grad}) shows that the continuity properties of the $\Pone$ basis functions of the underlying $\Pone$ space are not perturbed when building the multiscale basis functions.

In the general case, when the oversampling patch~$S_K$ is larger than~$K$, we cannot preserve any of these continuity properties if we use \DOF-extended local problems for our local computations, since the values on~$\partial K$ are not controlled by the degrees of freedom~$\Gamma(S_K,\cdot)$ on~$\partial S_K$. Therefore, we introduce another variant of the local problems to define \DOF-continuous numerical correctors in the next section.

\subsubsection{Numerical correctors: second oversampling strategy}
\begin{definition}
Let $K \in \mesh$ and let~$V_K$ and~$s_K^\varepsilon$ be a sampling space and sampling form, respectively, according to Def.~\ref{def:sampling-space-OS-dofe}. Additionally, suppose that the operator~$\Gamma(K,\cdot)$ is well-defined on~$V_K$. Then a \textbf{\DOF-continuous local problem} is to find $v\in V_K$ such that
\begin{equation}
	\left\{
	\begin{IEEEeqnarraybox}[][c]{uts}
		\IEEEstrut
		$s_K^\varepsilon(v, w)$
		&$=$
		&$\langle g, w\rangle$ 
		for all $w \in V_{K,0}$,
		\\
		$\Gamma(K,v)$
		&$=$
		&$\text{given}$,
		\IEEEstrut
		\IEEEeqnarraynumspace %preserve the same alignment as obtained in an equation environment
	\end{IEEEeqnarraybox}
	\right.
	\label{eq:sampling-space-pde-OS-dofc}
\end{equation}
for some $ g \in (H^1(S_K))'$.
\label{def:sampling-OS-DOFc}
\end{definition}

\begin{definition} 
\label{def:corr-dofc}
Suppose any \DOF-continuous local problem in Def.~\ref{def:sampling-OS-DOFc} is well-posed. Then we introduce, for all $K\in\mesh$ and all $0 \leq \alpha \leq d$, the functions~$\VSK{\alpha}{2}$ as the unique functions in $V_K$ with $\Gamma\left(K, \VSK{\alpha}{2}\right)=0$ satisfying the corrector problem~\eqref{eq:MsFEM-gen-correctors-dofe}. We define the \textbf{\DOF-continuous numerical correctors}~$\VKOS{\alpha}{2}$ as the restriction of~$\VSK{\alpha}{2}$ to~$K$, extended to all of~$\Omega$ by~$0$.
\end{definition}

We emphasize that the local problems of Def.~\ref{def:corr-dofe} and~\ref{def:corr-dofc} use test functions~$w$ in \emph{the same space}~$V_{K,0}$. This means that the test functions satisfy $\Gamma(S_K,w)=0$ rather than $\Gamma(K,w)=0$.
The difference between \DOF-extended and \DOF-continuous numerical correctors is that the former satisfy $\Gamma{\left(S_K,\VKOS{\alpha}{1}\right)}=0$, whereas the latter satisfy $\Gamma{\left(K,\VKOS{\alpha}{2}\right)}=0$.

\begin{remark}
    \label{rem:DOFec-noOS}
    Clearly, when~$S_K=K$, there is no difference between the \DOF-extended and \DOF-continuous problems~\eqref{eq:sampling-space-pde-OS-dofe} and~\eqref{eq:sampling-space-pde-OS-dofc}.
    We shall in this case simply refer~\eqref{eq:sampling-space-pde-OS-dofe} (or~\eqref{eq:sampling-space-pde-OS-dofc}) as local problems, and we write $\VKOS{\alpha}{0} = \VK{\alpha}$ for the numerical correctors of MsFEMs without oversampling.
\end{remark}

\begin{example}[MsFEM-lin for diffusion problems]
    Continuing Example~\ref{ex:msfem-lin-dofe}, consider now the \DOF-continuous numerical corrector~$\VKOS{\alpha}{2}$. Equation~\eqref{eq:sampling-space-pde-OS-dofc} solves the following problem for~$1 \leq \alpha \leq d$: there exists~$w \in \Pone(S_K)$ such that 
    \begin{equation*}
        -{\operatorname{div}(A^\varepsilon \nabla \VKOS{\alpha}{2})} = \operatorname{div}(A^\varepsilon e_\alpha) \
        \text{ in } S_K,
        \quad
        \VKOS{\alpha}{2} = w \
        \text{ on } \partial S_K,
        \quad 
        \VKOS{\alpha}{2} = 0 \
        \text{ at the vertices of } K.
    \end{equation*}
    The boundary condition on~$\partial S_K$ is complemented by a condition at the vertices of~$K$. Except when~$A^\varepsilon$ is constant (and a solution is~$\VKOS{\alpha}{2}=0$), it is not evident whether a solution to this problem exists. For $\alpha=0$, the numerical corrector~$\VKOS{0}{2}$ vanishes, as in the \DOF-extended case.
    \label{ex:msfem-lin-dofc}
\end{example}

\begin{example}[MsFEM-CR for diffusion problems]
    For the MsFEM-CR considered in Example~\ref{ex:msfem-cr-dofe}, the \DOF-continuous numerical correctors satisfy the same problem~\eqref{eq:msfem-cr-dofe} (for $1 \leq \alpha \leq d$) as the \DOF-extended numerical correctors, but with the average condition (the final equation in~\eqref{eq:msfem-cr-dofe}) replaced by $\displaystyle \frac{1}{|h|} \int_h \VKOS{\alpha}{2} = 0$ for each $h \in \mathcal{F}(K)$.
    As we saw for the MsFEM-lin in Example~\ref{ex:msfem-lin-dofc}, this is not a standard boundary value problem on~$S_K$.
    For the case~$\alpha=0$, we have $\VKOS{0}{2}=0$, which clearly satisfies the constraints $\displaystyle \frac{1}{|h|} \int_h \VKOS{0}{2} = 0$ for each $h \in \mathcal{F}(K)$.
    \label{ex:msfem-cr-dofc}
\end{example}

Examples~\ref{ex:msfem-lin-dofe} and~\ref{ex:msfem-cr-dofe} show that a \DOF-extended local problem is typically equivalent to a PDE with boundary conditions on~$S_K$. Under reasonable assumptions, these problems have a unique solution as required by Def.~\ref{def:sampling-space-OS-dofe}. We have seen in Examples~\ref{ex:msfem-lin-dofc} and~\ref{ex:msfem-cr-dofc} that this is not the case for \DOF-continuous problems, for which one finds some boundary conditions on~$ \partial S_K$ (because the degrees of freedom of test functions in~$V_{K,0}$ are prescribed on~$S_K$) and another set of conditions on~$\partial K$ that are explicitly imposed through the degrees of freedom on~$K$ in~\eqref{eq:sampling-space-pde-OS-dofc}. Well-posedness is not obvious in general, and cannot always be deduced from well-posedness of the \DOF-extended counterpart~\eqref{eq:sampling-space-pde-OS-dofe}. We address the well-posedness of \DOF-continuous problems in more detail in Sec.~\ref{sec:os-glue-correctors}. The advantage of \DOF-continuous oversampling is that it imposes certain continuity properties on the multiscale basis functions, and we will see in Sec.~\ref{sec:gen-num} that it yields better numerical approximations than \DOF-extended oversampling.

\subsubsection{Well-posedness of \DOF-continuous numerical correctors}
\label{sec:os-glue-correctors}
We have seen in Examples~\ref{ex:msfem-lin-dofc} and~\ref{ex:msfem-cr-dofc} that \DOF-continuous local problems lead to non-standard boundary conditions. This poses not only a theoretical issue, but also a computational challenge. To complete our study of the general MsFEM framework, we now present a computational strategy to obtain the \DOF-continuous numerical correctors, and we use this strategy to discuss the well-posedness of the associated local problems. 

In Def.~\ref{def:sampling-space-OS-dofe} we assume the well-posedness of \DOF-extended problems, and we have seen in Examples~\ref{ex:msfem-lin-dofe} and~\ref{ex:msfem-cr-dofe} that this is a natural assumption. It is also natural to assume that we can compute \DOF-extended numerical correctors numerically. We compute the \DOF-continuous numerical correctors from the \DOF-extended numerical correctors, by subtracting a linear combination of suitable functions~$W^\beta$ from the \DOF-extended numerical correctors. The~$W^\beta$ must all satisfy the homogeneous equation
$
	s_K^\varepsilon \left( W^\beta , w \right) = 0
$
for all
$
	w \in V_{K,0},
$
in order not to perturb the local problem~\eqref{eq:MsFEM-gen-correctors-dofe} that is already satisfied by both types of numerical correctors.
We shall use the functions $W^0 \coloneqq 1 + \VSK{0}{1}$ and $W^\beta \coloneqq x^\beta - x^\beta_{c,K} + \VSK{\beta}{1}$ for $1 \leq \beta \leq d$, where~$\VSK{\beta}{1}$ is defined in Def.~\ref{def:corr-dofe}. The precise strategy is as follows.

Fix~$0 \leq \alpha \leq d$. We look for coefficients~$c_0^\alpha, \dots, c_d^\alpha$ such that
$
	\displaystyle 
	\VSK{\alpha}{2} 
	=
	\VSK{\alpha}{1}
	- \sum_{\beta=0}^d c_\beta^\alpha \, W^\beta
$
on~$K$, where we recall that $\VSK{\alpha}{2}$ is defined by Def.~\ref{def:corr-dofc}. Note that both sides of the equation clearly satisfy~\eqref{eq:MsFEM-gen-correctors-dofe}. The desired equality thus holds if and only if
$
	\displaystyle
	\Gamma \left( K,
		\VSK{\alpha}{1} - \sum_{\beta=0}^d c_\beta^\alpha \, W^\beta
	\right)
	=0
$.
Since the \DOF~operators are linear, this leads to the linear system
\begin{equation}
	\underbrace{ 
		\begin{bmatrix}
			\vert & \vert & & \vert \\
			\Gamma(K,W^0) & \Gamma(K,W^1) & \hdots & \Gamma(K,W^d) \\
			\vert & \vert & & \vert \\
		\end{bmatrix}
	}_{
		\displaystyle
		\eqqcolon \mathds{M}
	}
	\begin{bmatrix}
		c_0^\alpha \\ c_1^\alpha \\ \vdots \\ c_d^\alpha
	\end{bmatrix}
	= 
	\Gamma\left(K,\VSK{\alpha}{1}\right).
\label{eq:glue-system}
\end{equation}
Invertibility of the matrix~$\mathds{M}$ is thus a sufficient condition for the existence of all \DOF-continuous numerical correctors, and the resolution of the linear system~\eqref{eq:glue-system} for each~$\alpha$ (where all \DOF-extended numerical correctors are replaced by their numerical approximation) allows to compute the \DOF-continuous numerical correctors numerically.

Before studying the invertibility of the matrix~$\mathds{M}$ in a few special cases, let us consider the matrix composed of the degrees of freedom on~$S_K$, i.e., the matrix
\begin{equation*}
	\widetilde{\mathds{M}} 
	\coloneqq 
	\begin{bmatrix}
		\vert & \vert & & \vert \\
		\Gamma(S_K,W^0) & \Gamma(S_K,W^1) & \hdots & \Gamma(S_K,W^d) \\
		\vert & \vert & & \vert \\
	\end{bmatrix}.
	\end{equation*}
By definition of the functions~$\VSK{\beta}{}$, we have $\Gamma(S_K,W^\beta) = \Gamma(S_K,x^\beta - x_{c,K}^\beta)$ for~$1\leq \beta \leq d$, and $\Gamma(S_K, W^{0}) = \Gamma(S_K,1)$. 
Note that the constant function together with the coordinate functions~$x^\beta-x^\beta_{c,K}$ ($1 \leq \beta \leq d$) span $\Pone(S_K)$. Since~$\Gamma(S_K,\cdot)$ is a bijection, the vectors $\Gamma(S_K, W^0),\dots,\Gamma(S_K,W^{d})$ are linearly independent. Hence the matrix~$\widetilde{\mathds{M}}$ is invertible.
One may hope that the linear independence of the vectors~$\Gamma(S_K,W^\beta)$ is preserved for the degrees of freedom on the interior boundary~$\partial K$ instead of~$\partial S_K$, yielding invertibility of~$\mathds{M}$. We found this to hold for all numerical tests that we performed, involving both the MsFEM-lin and the MsFEM-CR.

We can prove invertibility of~$\mathds{M}$ in a few special cases. When~$s_K^\varepsilon$ is the sampling form that was used in Example~\ref{ex:msfem-lin-dofe} (corresponding to a diffusion problem; we will consider this case until the end of this section) and if~$A^\varepsilon$ is constant, all numerical correctors vanish on~$S_K$ and the foregoing argument for the matrix~$\widetilde{\mathds{M}}$ shows invertibility of~$\mathds{M}$.

In the periodic setting (see Sec.~\ref{sec:homogenization}), even though~$A^\varepsilon$ itself is not constant, its homogenized limit~$A^\star$ is. In this case, the~$\VSK{\beta}{1}$ converge to zero weakly in~$H^1(S_K)$. (We show this in Lemma~\ref{lem:corr-to-0} in the absence of oversampling, but the argument can be generalized to \DOF-extended oversampling.) 
Now consider the MsFEM-CR. The weak convergence of the~$\VSK{\beta}{1}$ in~$H^1(S_K)$ ensures weak convergence on each face of~$K$ in the~$H^{1/2}$-norm by continuity of the trace operator. Since the embedding of $H^{1/2}(\partial K)$ in $L^2(\partial K)$ is compact, the~$\VSK{\beta}{1}$ converge to $0$ strongly in~$L^2$ on each face of~$K$. Consequently, the degrees of freedom~$\Gamma(K, \VSK{\beta}{1})$ (the averages on the faces of~$K$) converge to zero as $\varepsilon\to0$. Thus, $\Gamma(K,W^0) \to \Gamma(K,1)$ and $\Gamma(K,W^\beta) \to \Gamma(K,x^\beta - x_{c,K}^\beta)$ as $\varepsilon\to0$ for all $1 \leq \beta \leq d$ and, by the above argument for the matrix~$\widetilde{\mathds{M}}$, the matrix~$\mathds{M}$ is invertible in this limit. By continuity of the determinant function, the matrix~$\mathds{M}$ is invertible when~$\varepsilon$ is small enough, and the \DOF-continuous basis functions exist in this regime.

The study of the \DOF-continuous numerical correctors for the MsFEM-lin is more delicate, since pointwise operations are involved in evaluating the degrees of freedom, which are ill-defined on~$H^1(S_K)$. One can invoke the De Giorgi-Nash result, which can be found e.g.\ in~\cite[Theorem 8.22]{gilbarg_elliptic_2001}, to see that the multiscale basis functions, obtained from the numerical correctors in Def.~\ref{def:msbasis-OS} below, are in fact continuous for any bounded diffusion tensor. (See Example~\ref{ex:msfem-lin-basis} for a definition of the multiscale basis functions for the MsFEM-lin independent of the numerical correctors.) Pointwise evaluation is then justified. It would therefore be convenient to study the \DOF-continuous basis functions directly, without the intermediate step of the numerical correctors. We do not further pursue this topic here.

\subsection{The multiscale basis functions}
\label{sec:msbasis}
We can now define the multiscale basis functions for the approximation of the abstract problem~\eqref{eq:gen-pb} in terms of the numerical correctors. We recall that in Sec.~\ref{sec:msfem}, the numerical correctors were derived from the definition of the basis functions. We give an equivalent definition of the multiscale basis functions, independent of the numerical correctors, in Lemmas~\ref{lem:msbasis-OS-vf-dofe} and~\ref{lem:msbasis-OS-vf-dofc}. Recall that $\phiPone{1},\dots,\,\phiPone{N}$ is a basis of the space~$V_{H}$ (see Def.~\ref{def:underlying-basis}). We can suppose that the first $N_{0}$ basis functions form a basis of $V_{H,0}$. The following definition is the generalization of~\eqref{eq:diffusion-MsFEM-Vxy-grad} to the general MsFEM framework.

\begin{definition}
    For each $i=1,\dots,N$, the \textbf{multiscale basis function}~$\phiEps{i}$ is defined by
    \begin{equation}
        \tcboxmath[colback=white, colframe=black]{
            \forall \, K \in \mesh,
            \qquad
            \left. \phiEps{i} \right\vert_K
            =
            \left. \phiPone{i} \right\vert_K + \phiPone{i}(x_{c,K}) \, \VKOS{0}{0} + \sum_{\alpha=1}^d \partial_\alpha \left(\left. \phiPone{i} \right\vert_K \right) \VKOS{\alpha}{0},
        }
        \label{eq:MsFEM-Vxy}
    \end{equation}
    where $\bullet = \mathsf{e}$ corresponds to \DOF-extended multiscale basis functions and $\bullet = \mathsf{c}$ corresponds to \DOF-continuous multiscale basis functions.
\label{def:msbasis-OS}
\end{definition}

The \DOF-extended multiscale basis functions satisfy a variational problem on the oversampling patches~$S_K$ as shown by the following lemma.

\begin{lemma}
    Let~$K$ be any mesh element and let $1 \leq i \leq N$. Consider an MsFEM with \DOF-extended basis functions. Define an extension of~$\phiEps{i}$ from~$K$ to~$S_K$ by  
    \begin{equation}
        \widehat{\phiEps{i}}
        =
        \widehat{\left.\phiPone{i}\right\vert_K} + \phiPone{i} \left(x_{c,K} \right) \, \VSK{0}{1} + \sum_{\alpha=1}^d \partial_\alpha \left(\left. \phiPone{i} \right\vert_K \right) \VSK{\alpha}{1},
        \quad 
        \text{in } S_K,
        \label{eq:MsFEM-VxyOS-dofe}
    \end{equation}
    where $\widehat{ \left. \phiPone{i} \right\vert_K }$ denotes the affine extension of $\left. \phiPone{i} \right\vert_K$ to~$S_K$, and~$\VKOS{\alpha}{1}$ is as in Def.~\ref{def:corr-dofe}.
    Then $\widehat{\phiEps{i}}$ is the unique solution in~$V_K$ to
    \begin{equation}
        \left\{
        \begin{IEEEeqnarraybox}[][c]{uts}
            \IEEEstrut
            $s_K^\varepsilon \left(\widehat{\phiEps{i}}, w \right)$
            &$=$
            &$0$ 
            for all $w \in V_{K,0}$,
            \\
            $\Gamma \left( S_K, \widehat{\phiEps{i}} \right)$
            &$=$
            &$\Gamma \left( S_K , \widehat{\left.\phiPone{i}\right\vert_K} \right)$.
            \IEEEstrut
            \IEEEeqnarraynumspace %preserve the same alignment as obtained in an equation environment
        \end{IEEEeqnarraybox}
        \right.
        \label{eq:msbasis-OS-vf-dofe}
    \end{equation}
    \label{lem:msbasis-OS-vf-dofe}
\end{lemma}

In the case of the MsFEM-lin for the diffusion problem~\eqref{eq:diffusion-pde}, problem~\eqref{eq:msbasis-OS-vf-dofe} with $S_K=K$ coincides with the definition of the multiscale basis functions in~\eqref{eq:MsFEM-basis}; see Example~\ref{ex:msfem-lin-basis}.

\begin{proof}
    Problem~\eqref{eq:msbasis-OS-vf-dofe} has a unique solution in view of Def.~\ref{def:sampling-space-OS-dofe}. It thus suffices to show that~$\widehat{\phiEps{i}}$ satisfies~\eqref{eq:msbasis-OS-vf-dofe}.
    Since the numerical correctors~$\VKOS{\alpha}{1}$ belong to $V_{K,0}$ for all $0 \leq \alpha \leq d$, it is clear from~\eqref{eq:MsFEM-VxyOS-dofe} that $\Gamma\left(S_K,\widehat{\phiEps{i}}\right)=\Gamma \left( S_K , \widehat{\left.\phiPone{i}\right\vert_K} \right)$. 
    
    Inserting~\eqref{eq:MsFEM-VxyOS-dofe} into~\eqref{eq:msbasis-OS-vf-dofe} and applying~\eqref{eq:MsFEM-gen-correctors-dofe} to all~$\VKOS{\alpha}{1}$, we find, for any test function $w\in V_{K,0}$,
    % % THE FORMULAS BELOW WERE SUBMITTED BUT COMPLETELY WRONG
    % \begin{align*}
    %     s_K^\varepsilon \left(\widehat{\phiEps{i}}, w \right)
    %     &=
    %     s_K^\varepsilon \left(\widehat{\phiEps{i}}, \widehat{\left.\phiPone{i}\right\vert_K} \right)
    %     +
    %     \phiPone{i} \left( x_{c,K} \right) s_K^\varepsilon \left( \widehat{\phiEps{i}}, \VSK{0}{1} \right) 
    %     + 
    %     \sum_{\alpha=1}^d 
    %     \left.\left(  \partial_\alpha \phiPone{i} \right) \right\vert_K s_K^\varepsilon \left( \widehat{\phiEps{i}}, \VSK{\alpha}{1}\right)\\
    %     &=
    %     s_K^\varepsilon \left(\widehat{\phiEps{i}}, \widehat{\left.\phiPone{i}\right\vert_K} \right)
    %     -
    %     \phiPone{i} \left( x_{c,K} \right) s_K^\varepsilon \left( \widehat{\phiEps{i}}, 1 \right) 
    %     - 
    %     \sum_{\alpha=1}^d 
    %     \left.\left( \partial_\alpha  \phiPone{i} \right) \right\vert_K s_K^\varepsilon \left( \widehat{\phiEps{i}}, x^\alpha - x^\alpha_{c,K}\right)\\
    %     &=
    %     s_K^\varepsilon \left(\widehat{\phiEps{i}}, \widehat{\left.\phiPone{i}\right\vert_K} \right) 
    %     -
    %     s_K^\varepsilon \left( 
    %         \widehat{\phiEps{i}}, \,
    %         \phiPone{i} \left( x_{c,K} \right) 
    %         + \sum_{\alpha=1}^d \left.\left( \partial_\alpha \phiPone{i} \right) \right\vert_K \left( x^\alpha - x^\alpha_{c,K} \right)
    %     \right).
    % \end{align*}
    \begin{align*}
        s_K^\varepsilon \left(\widehat{\phiEps{i}}, w \right)
        &=
        s_K^\varepsilon \left(\widehat{\left.\phiPone{i}\right\vert_K}, w \right)
        +
        \phiPone{i} \left( x_{c,K} \right) s_K^\varepsilon \left( \VSK{0}{1},w \right)
        + 
        \sum_{\alpha=1}^d 
        \left.\left(  \partial_\alpha \phiPone{i} \right) \right\vert_K s_K^\varepsilon \left(  \VSK{\alpha}{1},w\right)\\
        &=
        s_K^\varepsilon \left(\widehat{\left.\phiPone{i}\right\vert_K}, w \right)
        -
        \phiPone{i} \left( x_{c,K} \right) s_K^\varepsilon \left( 1,w \right) 
        - 
        \sum_{\alpha=1}^d 
        \left.\left( \partial_\alpha  \phiPone{i} \right) \right\vert_K s_K^\varepsilon \left( x^\alpha - x^\alpha_{c,K},w\right)\\
        &=
        s_K^\varepsilon \left(\widehat{\left.\phiPone{i}\right\vert_K}, w \right)
        -
        s_K^\varepsilon \left(  \phiPone{i} \left( x_{c,K} \right)  + \sum_{\alpha=1}^d \left.\left( \partial_\alpha \phiPone{i} \right) \right\vert_K \left( x^\alpha - x^\alpha_{c,K} \right), w \right).
    \end{align*}
    Here we use that~$s_K^\varepsilon$ is a bilinear form on~$V_K$, that all piecewise affine functions are contained in $V_K$ according to Def.~\ref{def:sampling-space-OS-dofe} (this ensures that $\widehat{\phiEps{i}}$ indeed lies in the domain of $s_K^\varepsilon$), and the property that $\nabla \phiPone{i}$ is piecewise constant. Finally, we use~\eqref{eq:P1-expansion} for $\varphi=\widehat{\phiPone{i}}$ to conclude that
    \begin{equation*}
        s_K^\varepsilon \left(\widehat{\phiEps{i}}, w \right)
        =
        s_K^\varepsilon \left(\widehat{\left.\phiPone{i}\right\vert_K},w \right)  
        -
        s_K^\varepsilon \left(\widehat{\left.\phiPone{i}\right\vert_K},w \right)  
        =
        0,
    \end{equation*}
    which establishes the desired variational formulation satisfied by~$\widehat{\phiEps{i}}$.
\end{proof}

If the \DOF-continuous problems~\eqref{eq:sampling-space-pde-OS-dofc} are well-posed, we obtain by the same arguments the following result for \DOF-continuous multiscale basis functions.

\begin{lemma}
    Let~$K$ be any mesh element and let $1 \leq i \leq N$. Assume that any \DOF-continuous local problem~\eqref{eq:sampling-space-pde-OS-dofc} is well-posed. Consider an MsFEM with \DOF-continuous basis functions. Define an extension of~$\phiEps{i}$ from~$K$ to~$S_K$ by  
    \begin{equation*}
        \widehat{\phiEps{i}}
        =
        \widehat{\left.\phiPone{i}\right\vert_K} + \phiPone{i} \left(x_{c,K} \right) \, \VSK{0}{2} + \sum_{\alpha=1}^d \partial_\alpha \left(\left. \phiPone{i} \right\vert_K \right) \VSK{\alpha}{2},
        \quad 
        \text{in } S_K,
    \end{equation*}
    where~$\VKOS{\alpha}{2}$ is as in Def.~\ref{def:corr-dofc}, and $\displaystyle \widehat{\left.\phiPone{i}\right\vert_K}$ is as defined in Lemma~\ref{lem:msbasis-OS-vf-dofe}.
    Then $\widehat{\phiEps{i}}$ is the unique solution in~$V_K$ to
    \begin{equation}
        \left\{
        \begin{IEEEeqnarraybox}[][c]{uts}
            \IEEEstrut
            $s_K^\varepsilon \left(\widehat{\phiEps{i}}, w \right)$
            &$=$
            &$0$ 
            for all $w \in V_{K,0}$,
            \\
            $\Gamma \left( K,\widehat{\phiEps{i}} \right)$
            &$=$
            &$\Gamma \left( K , \phiPone{i} \right)$.
            \IEEEstrut
            \IEEEeqnarraynumspace %preserve the same alignment as obtained in an equation environment
        \end{IEEEeqnarraybox}
        \right.
        \label{eq:msbasis-OS-vf-dofc}
    \end{equation}
    \label{lem:msbasis-OS-vf-dofc}
\end{lemma}

\begin{example}[MsFEM-lin for diffusion problems]
    In the setting of Example~\ref{ex:msfem-lin-dofe}, any \DOF-extended multiscale basis function~$\phiEps{i}$ for the MsFEM-lin constructed in~\eqref{eq:msbasis-OS-vf-dofe} is obtained, in each mesh element~$K$, as the restriction of a function~$\widehat{\phiEps{i}}$, which is the unique solution in~$H^1(S_K)$ to 
    \begin{equation*}
        -{\operatorname{div}(A^\varepsilon \nabla \widehat{\phiEps{i}})} = 0 \
        \text{ in } S_K,
        \quad
        \widehat{\phiEps{i}} = \widehat{\phiPone{i}} \
        \text{ on } \partial S_K.
    \end{equation*}
    For a \DOF-continuous basis function, $\widehat{\phiEps{i}}$ solves the same PDE in~$S_K$, is affine on~$\partial S_K$, and satisfies $\widehat{\phiEps{i}}(x_j) = \phiPone{i}(x_j)$ at all vertices~$x_j$ of~$K$. 
    \label{ex:msfem-lin-basis}
\end{example}

\begin{example}[MsFEM-CR for diffusion problems]
    In the continuation of Example~\ref{ex:msfem-cr-dofe}, the \DOF-extended multiscale basis function~$\phiEps{i}$ for the MsFEM-CR is the restriction to~$K$ of $\widehat{\phiEps{i}}$, the unique solution in~$H^1(S_K)$ to 
    \begin{equation*}
        \left\{
        \begin{IEEEeqnarraybox}[][c]{uts?s}
            \IEEEstrut
            $-{\operatorname{div}(A^\varepsilon \nabla \widehat{\phiEps{i}})}$
            &$=$
            &$0$ 
            &in $S_K$,
            \\
            $\vec{n} \cdot A^\varepsilon \nabla \widehat{\phiEps{i}}$
            &$=$
            &$0$
            &on each $h \in \mathcal{F}_a(S_K)$, \\
            $\vec{n} \cdot A^\varepsilon \nabla \widehat{\phiEps{i}}$
            &$=$
            &$c_h$
            &on each $h \in \mathcal{F}_d(S_K)$, \\
            $\displaystyle \frac{1}{|h|} \int_h \widehat{\phiEps{i}}$
            &$=$
            &$\displaystyle \frac{1}{|h|} \int_h \widehat{\phiPone{i}}$
            &for each $h \in \mathcal{F}_d(S_K)$,
            \IEEEstrut
            \IEEEeqnarraynumspace %preserve the same alignment as obtained in an equation environment
        \end{IEEEeqnarraybox}
        \right.
    \end{equation*}
    where the constants~$c_h$ are uniquely determined by the problem. We recall that the sets of faces~$\mathcal{F}_a(S_K)$ and~$\mathcal{F}_d(S_K)$ are defined in Sec.~\ref{sec:dof-ospatch}. For \DOF-continuous basis functions, the last condition is applied to the faces $h \in \mathcal{F}(K)$ (and all other conditions remain unchanged).
    \label{ex:msfem-cr-basis}
\end{example}

Our general framework allows two characterizations  of the multiscale basis functions, namely~\eqref{eq:MsFEM-Vxy} and~\eqref{eq:msbasis-OS-vf-dofe} or~\eqref{eq:msbasis-OS-vf-dofc}, as was the case for the MsFEM studied in Sec.~\ref{sec:msfem} (where $\phiEps{i}$ is given by~\eqref{eq:MsFEM-basis} or~\eqref{eq:diffusion-MsFEM-Vxy-grad}). The essential advantage of~\eqref{eq:MsFEM-Vxy} is that the microscale is fully encoded in the numerical correctors~$\VKOS{\alpha}{0}$, that can be computed element per element without any global information. In particular, the \emph{global} index~$i$ of the multiscale basis function~$\phiEps{i}$ is irrelevant for the computation of the numerical correctors. The expression in~\eqref{eq:MsFEM-Vxy} is therefore the crucial relationship that we will employ to develop non-intrusive MsFEMs within the general framework in Sec.~\ref{sec:noni-gen}, just as was~\eqref{eq:diffusion-MsFEM-Vxy-grad} in Sec.~\ref{sec:msfem-effective} and~\ref{sec:msfem-diffusion-nonin}.

The second formulation of the multiscale basis functions, as solutions to the local problems~\eqref{eq:msbasis-OS-vf-dofe} or~\eqref{eq:msbasis-OS-vf-dofc}, provides a more direct interpretation of the multiscale basis functions in terms of the sampling form chosen. It also gives a relation between the degrees of freedom of the $\Pone$ basis functions and the associated multiscale basis function. This is useful in particular for the well-posedness of the MsFEM, that we study in Lemma~\ref{lem:MsFEM-OS-well-posedness}.

\begin{remark}
    Our definition of the multiscale basis functions in~\eqref{eq:MsFEM-Vxy} is reminiscent of the Variational Multiscale Method, a framework developed in~\cite{hughes_multiscale_1995, hughes_variational_1998} to adapt Galerkin approximations on low-dimensional spaces to the presence of multiscale features. In this context, our formulation of the MsFEM also exhibits a link with residual-free bubbles, see e.g.~\cite{brezzi_choosing_1994, brezzi_bint_1997, hughes_variational_1998}.    
\end{remark}

\begin{remark}
\label{rem:OS-literature}
The first introduction of the MsFEM in~\cite{hou_multiscale_1997} corresponds to the idea of oversampling with \DOF-continuous basis functions. Although their existence cannot be established in general, they are computed numerically by taking linear combinations of \DOF-extended basis functions (following an analogous strategy to the one we discussed in Sec.~\ref{sec:os-glue-correctors}). The MsFEM with \DOF-extended basis functions is studied in the works~\cite{efendiev_convergence_2000,hou_removing_2004} dealing with the convergence analysis of the MsFEM-lin with oversampling.

Let us also note that the combination of Crouzeix-Raviart MsFEM and oversampling has, to the best of our knowledge, not yet been proposed in the literature. This method, for which the basis functions are given explicitly in Example~\ref{ex:msfem-cr-basis}, is a natural by-product of the identification of the abstract MsFEM framework.
\end{remark}

\subsection{The global problem}

We can now define the multiscale trial and test spaces, respectively~$V_H^\varepsilon$ and~$V_{H,0}^\varepsilon$, as follows:
\begin{equation*}
    V_H^\varepsilon = \left\{
        \phiEps{i} \ \mid \ 1 \leq i \leq N
    \right\},
    \qquad
    V_{H,0}^\varepsilon = \left\{
        \phiEps{i} \ \mid \ 1 \leq i \leq N_0
    \right\}.
\end{equation*}
We recall that we have assumed the first~$N_0$ basis functions of~$V_H$ to form a basis of~$V_{H,0}$ in Sec.~\ref{sec:msbasis}.
Note that we only use~$V_{H,0}^\varepsilon$ in the present section, because~\eqref{eq:gen-pb} is posed with homogeneous Dirichlet boundary conditions, but that the larger space~$V_H^\varepsilon$ is useful for more general boundary conditions (see Sec.~\ref{sec:noni-gen-more}). Applying~\eqref{eq:MsFEM-Vxy}, we have the equivalent characterization in terms of the $\Pone$ space~$V_H$,
\begin{equation*}
	V_H^\varepsilon
	=
	\left\{
		\left.
		v_H^\varepsilon
		=
		v_H + 
		\sum_{K\in\mesh} \left(
			v_H(x_{c,K}) \, \VKOS{0}{0}
			+ \sum_{\alpha=1}^d
				\partial_\alpha \left( \left. v_H \right\vert_K \right)
				\VKOS{\alpha}{0}
		\right)
		\,\right\vert\,
		v_H \in V_H
	\right\}.
\end{equation*}

\begin{definition}
\label{def:gen-MsFEM}
Let~$V_H$ be an underlying $\Pone$ space defined in Def.~\ref{def:underlying-space} with the associated \DOF~operator~$\Gamma$ from Def.~\ref{def:DOF} and Def.~\ref{def:DOF-os-lin}-\ref{def:DOF-os-cr}. Define for each mesh element~$K \in \mesh$ an oversampling patch (Def.~\ref{def:ospatch}), a sampling space and sampling form in accordance with Def.~\ref{def:sampling-space-OS-dofe}. Let the multiscale basis functions~$\phiEps{i}$ be given as in Def.~\ref{def:msbasis-OS}. Then a \textbf{Multiscale Finite Element Method} (MsFEM) for problem~\eqref{eq:gen-pb} is: find~$u^\varepsilon_H \in V_{H,0}^\varepsilon$ such that
\begin{equation}
    \forall \, v_H^\varepsilon \in V_{H,0}^\varepsilon,
    \qquad
    \sum_{K\in\mesh} a_K^{\varepsilon} \left( u^\varepsilon_H,v_H^\varepsilon \right) = F \left( v_H^\varepsilon \right).
    \label{eq:gen-MsFEM-OS}
\end{equation}
\end{definition}

In the following lemma, we investigate the well-posedness of the MsFEM.

\begin{lemma}
\label{lem:MsFEM-OS-well-posedness}
Consider an MsFEM without oversampling, or an MsFEM with oversampling using \DOF-continuous basis functions (assuming the associated basis functions are well-defined). When~$a^\varepsilon$ satisfies~\eqref{eq:gen-pb-coer}, the MsFEM~\eqref{eq:gen-MsFEM-OS} has a unique solution.
\end{lemma}

\begin{proof}
Note that, with \DOF-continuous oversampling, but also without oversampling, the multiscale basis functions satisfy~\eqref{eq:msbasis-OS-vf-dofc}. In particular, all degrees of freedom of $u_H^\varepsilon$ related to the boundary vanish. Also note that, the dimension of~$V_{H,0}^\varepsilon$ being finite, it suffices to show that $u^\varepsilon_H=0$ is the unique solution to problem~\eqref{eq:gen-MsFEM-OS} with~$F=0$. 

If $0 = F(u^\varepsilon_H) = a^\varepsilon(u_H^\varepsilon, u_H^\varepsilon)$, it follows from~\eqref{eq:gen-pb-coer} that $u^\varepsilon_H$ is piecewise constant.
Let us write $\displaystyle u_H^\varepsilon = \sum_{i=1}^{N_{0}} \alpha_i \, \phiEps{i}$ for some coefficients~$\alpha_i\in\bbR$ and introduce the function $ \displaystyle u_H = \sum_{i=1}^{N_{0}} \alpha_i \, \phiPone{i} \in V_H$. Because of~\eqref{eq:msbasis-OS-vf-dofc}, we have $\Gamma(K,u_H^\varepsilon) = \Gamma(K,u_H)$ for all mesh elements~$K$. Since~$u_H^\varepsilon$ is piecewise constant and~$\Gamma(K,\cdot)$ is a bijection from~$\Pone(K)$ to~$\bbR^{d+1}$ (recall Def.~\ref{def:DOF}), it follows that $u_H^\varepsilon=u_H$. In particular, the multiscale function~$u_H^\varepsilon$ in fact belongs to the underlying $\Pone$ space~$V_H$.

We remarked immediately below Def.~\ref{def:underlying-space} that, for either of the two spaces~$V_H=V_H^{L}$ or~$V_H^{CR}$, the above implies that~$u^\varepsilon_H$ is constant throughout~$\Omega$. Since the degrees of freedom of~$u^\varepsilon_H$ associated to the boundary vanish, we readily deduce that~$u^\varepsilon_H=0$.
\end{proof}

We do not know of the existence of a result on the well-posedness of MsFEMs with oversampling using \DOF-extended multiscale basis functions. In~\cite{hou_removing_2004}, the authors establish an inf-sup result for a variant of the MsFEM-lin-OS with $\Pone$ test functions (see also Def.~\ref{def:gen-MsFEM-OS-testP1}). This result is obtained for a periodic diffusion coefficient in the limit of sufficiently small~$\varepsilon$.
\section{Non-intrusive MsFEM for the general framework}
	\label{sec:noni-gen}

%%
%%%%%%%%%%%%%%%%%%%%%%%%%%%%%%%%%%
%% Derivation and definition of non-intrusive MsFEM approaches within the general MsFEM framework
%%%%%%%%%%%%%%%%%%%%%%%%%%%%%%%%%%
%%

We show in this section how to develop a non-intrusive approach for the general MsFEM framework of Sec.~\ref{sec:gen-framework}. We have seen in Lemma~\ref{lem:IBP-bubbles} that, for a particular MsFEM variant, the non-intrusive Galerkin MsFEM approach coincides with a Petrov-Galerkin MsFEM. This does not hold for all MsFEMs in the general framework. We first develop a non-intrusive MsFEM approach for a Petrov-Galerkin MsFEM in the general framework. We show that the non-intrusive approach for the Petrov-Galerkin MsFEM is actually equivalent to the Petrov-Galerkin MsFEM itself. In a second step, we introduce a non-intrusive approximation of the Galerkin MsFEM. Before doing so, let us  summarize the main steps of Sec.~\ref{sec:msfem-effective} and~\ref{sec:msfem-diffusion-nonin} to obtain a non-intrusive MsFEM approach:
\begin{enumerate}[label=(\arabic*)]
	\item 
	the expansion~\eqref{eq:diffusion-MsFEM-Vxy-grad} allows to recast the matrix~$\mathds{A}^\varepsilon$ of the linear system for the MsFEM as the matrix~$\mathds{A}^\Pone$ associated to the~$\Pone$ discretization of an effective problem;
	\label{proc:noni-matrix}
	\item 
	we approximate the right-hand side~$\mathds{F}^\varepsilon$ of the MsFEM problem by the right-hand side~$\mathds{F}^\Pone$ of this~$\Pone$ discretization;
	\label{proc:noni-rhs}
	\item
	the post-processing step~\eqref{eq:msfem-diff-noni-post} applied to the~$\Pone$ approximation of the effective problem yields the MsFEM approximation.
\end{enumerate}

\subsection{The Petrov-Galerkin MsFEM}
\label{sec:noni-gen-pg}

We recall that the abstract continuous problem for which we developed the MsFEM in Sec.~\ref{sec:gen-framework} is given by~\eqref{eq:gen-pb} and that it can be rewritten in terms of the bilinear forms~$a_K^\varepsilon$ satisfying~\eqref{eq:gen-pb-coer}. Petrov-Galerkin variants of the multiscale finite element method with $\Pone$ test functions were previously studied in~\cite{hou_removing_2004,hesthaven_high-order_2014}. In our general MsFEM framework, the adaptation of Def.~\ref{def:gen-MsFEM} to a Petrov-Galerkin MsFEM is the following.

\begin{definition}
\label{def:gen-MsFEM-OS-testP1}
Let~$V_H$ be an underlying $\Pone$ space defined in Def.~\ref{def:underlying-space} with the associated \DOF~operator~$\Gamma$ from Def.~\ref{def:DOF} and Def.~\ref{def:DOF-os-lin}-\ref{def:DOF-os-cr}. Define for each mesh element~$K \in \mesh$ an oversampling patch (Def.~\ref{def:ospatch}), a sampling space and sampling form in accordance with Def.~\ref{def:sampling-space-OS-dofe}. Let the multiscale basis functions~$\phiEps{i}$ be given as in Def.~\ref{def:msbasis-OS}. Then a \textbf{Petrov-Galerkin Multiscale Finite Element Method} (PG-MsFEM) for problem~\eqref{eq:gen-pb} is: find~$u^\varepsilon_H \in V_{H,0}^\varepsilon$ such that
\begin{equation}
    \forall \, v_H \in V_{H,0},
    \qquad
    \sum_{K\in\mesh} a_K^{\varepsilon}{\left( u^\varepsilon_H, v_H \right)} = F{\left( v_H \right)}.
    \label{eq:gen-MsFEM-OS-testP1}
\end{equation}
\end{definition}

When confusion may arise, we shall refer to the MsFEM defined in Def.~\ref{def:gen-MsFEM} as the Galerkin MsFEM (G-MsFEM). To study well-posedness of the PG-MsFEM, it is most convenient to relate this method to the G-MsFEM. Therefore, we postpone well-posedness of~\eqref{eq:gen-MsFEM-OS-testP1} to Lemma~\ref{lem:IBP-bubbles-gen}.

We now execute step~\ref{proc:noni-matrix} of the summary of the non-intrusive MsFEM approach at the beginning of this section. The matrix~$\mathds{A}^\varepsilon$ of the linear system associated to~\eqref{eq:gen-MsFEM-OS-testP1} is defined by
\begin{equation}
	\mathds{A}^\varepsilon_{j,i}
	=
	\sum_{K\in\mesh} a^\varepsilon_K{ \left(\phiEps{i}, \, \phiPone{j} \right)},
	\quad
	1 \leq i,j \leq N_{0}.
	\label{eq:system-msfem-gen-testP1}
\end{equation}
To find an effective $\Pone$ formulation with the same linear system, we will use the definition~\eqref{eq:MsFEM-Vxy} of the multiscale basis functions in the general framework, but first we combine it with~\eqref{eq:P1-expansion} applied to~$\varphi=\phiPone{i}$ to rewrite~\eqref{eq:MsFEM-Vxy} as 
\begin{align}
    \left. \phiEps{i} \right\vert_K 
    &= 
    \phiPone{i}(x_{c,K}) + \phiPone{i} \left( x_{c,K} \right) \VKOS{0}{0} + \sum_{\alpha=1}^d \left. \partial_\alpha \phiPone{i} \right\vert_K \left( x^\alpha - x^\alpha_{c,K} + \VKOS{\alpha}{0} \right)
    \nonumber\\
    &=
    \phiPone{i} \left( x_{c,K} \right) \, \VVK{0} + \sum_{\alpha=1}^d \left. \partial_\alpha \phiPone{i} \right\vert_K \VVK{\alpha},
    \label{eq:MsFEM-VVxy}
\end{align}
where 
\begin{equation}
	\VVK{0} \coloneqq 1+\VKOS{0}{0},
	\quad 
	\VVK{\alpha} \coloneqq x^\alpha-x^\alpha_{c,K} + \VKOS{\alpha}{0},
	\label{eq:VVK-def}
\end{equation}
for all $1 \leq \alpha \leq d$ and each~$K \in \mesh$. We recall that $\bullet \in \{\mathsf{e},\mathsf{c}\}$ indicates the choice of \DOF-extended or \DOF-continuous basis functions.
Inserting~\eqref{eq:MsFEM-VVxy} into~\eqref{eq:system-msfem-gen-testP1} for~$\phiEps{i}$ and~\eqref{eq:P1-expansion} for~$\varphi=\phiPone{j}$ yields
\begin{align*}
	\mathds{A}^\varepsilon_{j,i}
	&=
	\sum_{K\in\mesh} \left(
		\phiPone{i}\left(x_{c,K}\right) \, a_K^\varepsilon{\left(\VVK{0},1\right) \, \phiPone{j}\left(x_{c,K}\right)}
		+ \sum_{\alpha=1}^d \left.\left(\partial_\alpha\phiPone{i}\right)\right\vert_K \, a_K^\varepsilon {\left(\VVK{\alpha},1\right)} \, \phiPone{j}\left(x_{c,K}\right) 
		\right.
		\nonumber\\
		&\hspace{1.5cm}+ 
		\sum_{\beta=1}^d \phiPone{i}\left(x_{c,K}\right) \, a_K^\varepsilon{\left(\VVK{0},x^\beta-x_{c,K}^\beta\right)} \, \left.\left(\partial_\beta\phiPone{j}\right)\right\vert_K
		\nonumber\\
		&\hspace{1.5cm}+ 
		\left. \sum_{\alpha,\beta=1}^d
			\left.\left(\partial_\alpha\phiPone{i}\right)\right\vert_K \, a^\varepsilon_K{\left(\VVK{\alpha},x^\beta-x_{c,K}^\beta\right)} \, \left.\left(\partial_\beta\phiPone{j}\right)\right\vert_K
	\right),
\end{align*}
and therefore,
\begin{tcolorbox}[ams equation]
	\mathds{A}^\varepsilon_{j,i}
	=
	\mathlarger{\mathlarger{\sum}}_{K\in\mesh}
	|K| \, \left( \overline{M} \phiPone{i} \, \phiPone{j} \right)(x_{c,K})
	+ \int_K \phiPone{j} \, \overline{B}^1 \cdot \nabla \phiPone{i}
	+ \phiPone{i} \, \overline{B}^2 \cdot \nabla \phiPone{j}
	+ \nabla \phiPone{j} \cdot \overline{A} \, \nabla \phiPone{i},
	\label{eq:system-msfem-gen-decoupling}
\end{tcolorbox}
\noindent where we have defined the effective mass~$\overline{M}$, (adjoint) advection vector~$\overline{B}^1$ and~$\overline{B}^2$, and the effective diffusion tensor~$\overline{A}$, for all $1 \leq \alpha, \beta \leq d$ and for each $K\in\mesh$, as
\begin{equation}
    \begin{IEEEeqnarraybox}[][c]{us?us}
        \IEEEstrut
        $\left. \overline{M} \right\vert_K $
        &$\displaystyle = \frac{1}{|K|} \, a_K^\varepsilon{\left(\VVK{0},1\right)}$,
        &$\left. \overline{B}^1_\alpha \right\vert_K$ 
        &$\displaystyle = \frac{1}{|K|} \, a_K^\varepsilon{\left(\VVK{\alpha},1\right)}$,
        \\
        $\left. \overline{B}^2_\beta \right\vert_K$
        &$\displaystyle = \frac{1}{|K|} \, a_K^\varepsilon{\left(\VVK{0},x^\beta-x_{c,K}^\beta\right)}$,
        &$\left. \overline{A}_{\beta,\alpha} \right\vert_K $
        &$\displaystyle = \frac{1}{|K|} \, a^\varepsilon_K{ \left(\VVK{\alpha}, x^\beta-x_{c,K}^\beta\right)}$.
        \IEEEstrut
        \IEEEeqnarraynumspace %preserve the same alignment as obtained in an equation environment
    \end{IEEEeqnarraybox}
	\label{eq:gen-MsFEM-eff-coef}
\end{equation}
Note that~$\overline{M}$, $\overline{B}^1$, $\overline{B}^2$ and~$\overline{A}$ are all piecewise constant quantities.
All integrals in~\eqref{eq:system-msfem-gen-decoupling} can be computed exactly by evaluating the integrand at the centroid. With this quadrature rule, we observe that the term $|K| \left( \overline{M} \phiPone{i} \, \phiPone{j} \right)(x_{c,K})$ also equals the numerical approximation of the integral $\displaystyle \int_K \overline{M} \phiPone{i} \phiPone{j}$.

The new expression~\eqref{eq:system-msfem-gen-decoupling} for the matrix of the linear system motivates us to introduce the effective bilinear forms~$\overline{a}_K$ defined on~$H^1(K) \times H^1(K)$ by
\begin{equation}
\label{eq:gen-MsFEM-eff-form}
	\overline{a}_K(u,v) 
	=
	\int_K \nabla v \cdot \overline{A} \, \nabla u
		+ v \left(\overline{B}^1 \cdot \nabla u\right)
		+ u \left(\overline{B}^2 \cdot \nabla v\right)
		+ \overline{M} \, u \, v,
	\quad
	\text{for all } u,\,v \in H^1(K),
\end{equation}
and the associated $\Pone$ Galerkin approximation on the space~$V_{H,0}$:
\begin{equation}
	\text{Find } u_H\in V_{H,0} \text{ such that}
	\sum_{K\in\mesh} 
		\overline{a}_K(u_H,v_H)
	=
	F(v_H)
	\quad \text{for all } v_H \in V_{H,0}.
\label{eq:gen-FEM-effective}
\end{equation}
This discrete problem leads to a linear system with the matrix
\begin{equation*}
	\mathds{A}^\Pone_{j,i}
	=
	\overline{a} \left( \phiPone{i},\, \phiPone{j} \right)
	=
	\sum_{K\in\mesh} \overline{a}_K(\phiPone{i}, \, \phiPone{j}),
	\quad
	1 \leq i,j \leq N_0.
\end{equation*}
The identity~\eqref{eq:system-msfem-gen-decoupling} thus implies the following result, which generalizes Lemma~\ref{lem:stiffness-MsFEM-P1} to the PG-MsFEM in the general framework.
\begin{lemma}
\label{lem:gen-MsFEM-PG-eff}
The matrices~$\mathds{A}^\varepsilon$ and~$\mathds{A}^\Pone$ are identical if the integrals in~\eqref{eq:gen-MsFEM-eff-form} are evaluated at the centroid of each mesh element~$K$ for the computation of~$\mathds{A}^\Pone$. Then the PG-MsFEM~\eqref{eq:gen-MsFEM-OS-testP1} coincides with the resolution of the effective problem~\eqref{eq:gen-FEM-effective} combined with the post-processing step
\begin{equation}
	\label{eq:gen-MsFEM-post}
	\left. u_H^\varepsilon \right\vert_K
	=	
	u_H(x_{c,K}) \, \VVK{0} + \sum_{\alpha=1}^d \left. \partial_\alpha u_H \right\vert_K \VVK{\alpha}.
\end{equation} 
\end{lemma}

Note that step~\ref{proc:noni-rhs} of the summary at the beginning of this section is irrelevant for the PG-MsFEM. The computation of the right-hand side in~\eqref{eq:gen-MsFEM-OS-testP1} is clearly part of any standard FEM software. We refer to Rem.~\ref{rem:postprocessing-compute} and~\ref{rem:postprocessing-vis} for some additional comments on the post-processing step.

The computational approach described by Lemma~\ref{lem:gen-MsFEM-PG-eff} naturally fits within the non-intrusive workflow of Algorithm~\ref{alg:msfem-diff-noni}. The numerical correctors on line~\ref{alg:msfem-diff-noni-offl-corr} are, of course, replaced by those of Def.~\ref{def:corr-dofe} or Def.~\ref{def:corr-dofc}. 
Line~\ref{alg:msfem-diff-noni-offl-tensors} is replaced by the computation of all effective quantities in~\eqref{eq:gen-MsFEM-eff-coef}, where~$\VVK{\alpha}$ is related to the numerical correctors by~\eqref{eq:VVK-def}. The online phase in line~\ref{alg:msfem-diff-noni-onl} amounts to solving the $\Pone$ problem~\eqref{eq:gen-FEM-effective}, where all integrations to construct the matrix of the linear system are to be performed by evaluation at the centroid. (This is not the case for the construction of the right-hand side, however.) Finally, in the post-processing phase, we construct~$u^\varepsilon_H$ from~$u_H$ by virtue of~\eqref{eq:gen-MsFEM-post}.

Next we generalize the above expansions to design a non-intrusive approximation of the G-MsFEM.

\subsection{The non-intrusive Galerkin MsFEM}
\label{sec:noni-gen-gal}
For the G-MsFEM (introduced in Def.~\ref{def:gen-MsFEM}), we need to replace the $\Pone$ test space~$V_{H,0}$ of the PG-MsFEM by the multiscale test space~$V_{H,0}^\varepsilon$.
The matrix of the linear system associated to~\eqref{eq:gen-MsFEM-OS} is given by
\begin{equation*}
	\mathds{A}_{j,i}^{\varepsilon,\mathsf{G}}
	=
	\sum_{K\in\mesh} a_K^\varepsilon {\left( \phiEps{i},\phiEps{j} \right)},
	\quad
	1 \leq i,j \leq N_{0}.
\end{equation*}
Upon inserting~\eqref{eq:MsFEM-Vxy} for the test function~$\phiEps{j}$, we find, for all $1 \leq i,j \leq N_0$,
\begin{equation*}
	\mathds{A}_{j,i}^{\varepsilon,\mathsf{G}}
	=
	\mathds{A}_{j,i}^\varepsilon
	+ \sum_{K\in\mesh} \left(
		\phiPone{j}\left(x_{c,K}\right) \, a^\varepsilon_K{\left(\phiEps{i}, \VKOS{0}{0}\right)} 
		+
		\sum_{\beta=1}^d \left.\left(\partial_\beta \phiPone{j}\right) \right\vert_K \, a^\varepsilon_K{\left(\phiEps{i},\VKOS{\beta}{0}\right)}
		\right),
\end{equation*}
where~$\mathds{A}^\varepsilon$ is the matrix of the Petrov-Galerkin MsFEM, see~\eqref{eq:system-msfem-gen-testP1} and~\eqref{eq:system-msfem-gen-decoupling}.

An effective formulation can again be derived by inserting~\eqref{eq:MsFEM-VVxy} for the~$\phiEps{i}$. We obtain
\begin{tcolorbox}[ams equation*]
		\mathds{A}^{\varepsilon,\mathsf{G}}_{j,i}
		=
		\mathlarger{\mathlarger{\sum}}_{K\in\mesh}
			|K| \, \left( \overline{M}^\mathsf{G} \phiPone{i} \, \phiPone{j} \right)(x_{c,K})
			+ \int_K \phiPone{j} \, \overline{B}^{1,\mathsf{G}} \cdot \nabla \phiPone{i}
			+ \phiPone{i} \, \overline{B}^{2,\mathsf{G}} \cdot \nabla \phiPone{j}
			+ \nabla \phiPone{j} \cdot \overline{A}^\mathsf{G} \, \nabla \phiPone{i},
%\label{eq:system-msfem-gal-gen-decoupling}
\end{tcolorbox}
\noindent
where the effective mass, (adjoint) advection vectors and diffusion tensor are given by (using those defined in~\eqref{eq:gen-MsFEM-eff-coef})
\begin{equation}
    \begin{IEEEeqnarraybox}[][c]{us?us}
        \IEEEstrut
        $\left. \overline{M}^\mathsf{G} \right\vert_K $
        &$\displaystyle = \left. \overline{M} \right\vert_K + \frac{1}{|K|} \, a_K^\varepsilon{\left(\VVK{0},\VKOS{0}{0}\right)}$,
        &$\left. \overline{B}^{1,\mathsf{G}}_\alpha \right\vert_K$ 
        &$\displaystyle = \left. \overline{B}^{1}_\alpha \right\vert_K + \frac{1}{|K|} \, a_K^\varepsilon{\left(\VVK{\alpha},\VKOS{0}{0}\right)}$,
        \\
        $\left. \overline{B}^{2,\mathsf{G}}_\beta \right\vert_K$
        &$\displaystyle = \left. \overline{B}^{2}_\beta \right\vert_K + \frac{1}{|K|} \, a_K^\varepsilon{\left(\VVK{0},\VKOS{\beta}{0}\right)}$,
        &$\left. \overline{A}^\mathsf{G}_{\beta,\alpha} \right\vert_K$
        &$\displaystyle = \left. \overline{A}_{\beta,\alpha} \right\vert_K + \frac{1}{|K|} \, a^\varepsilon_K {\left(\VVK{\alpha}, \VKOS{\beta}{0}\right)}$.
        \IEEEstrut
        \IEEEeqnarraynumspace %preserve the same alignment as obtained in an equation environment
    \end{IEEEeqnarraybox}
	\label{eq:gen-MsFEM-eff-coef-Gal}
\end{equation}
Again, these quantities are all piecewise constant.

The above computations lead to the introduction of the effective bilinear form $\overline{a}^{\mathsf{G}} = \sum\limits_{K\in\mesh} \overline{a}^{\mathsf{G}}_K$ with
\begin{equation}
\label{eq:gen-MsFEM-eff-form-Gal}
	\overline{a}_K^\mathsf{G}(u,v) 
	=
	\int_K \nabla v \cdot \overline{A}^\mathsf{G} \, \nabla u
		+ v \left(\overline{B}^{1,\mathsf{G}} \cdot \nabla u\right)
		+ u \left(\overline{B}^{2,\mathsf{G}} \cdot \nabla v\right)
		+ \overline{M}^\mathsf{G} \, u \, v.
\end{equation}
We formulate the following effective variational problem:
\begin{equation}
	\text{Find } u_H\in V_{H,0} \text{ such that}
	\sum_{K\in\mesh} 
		\overline{a}^{\mathsf{G}}_K(u_H,v_H)
	=
	F(v_H)
	\quad \text{for all } v_H \in V_{H,0}.
\label{eq:gen-FEM-effective-gal}
\end{equation}
The associated linear system has coefficients $\mathds{A}_{j,i}^{\Pone,\mathsf{G}} =  \overline{a}^{\mathsf{G}} \left( \phiPone{i},\phiPone{j} \right)$. We have the following analogue of Lemma~\ref{lem:gen-MsFEM-PG-eff}, which generalizes Lemma~\ref{lem:stiffness-MsFEM-P1} to the G-MsFEM in the general framework.

\begin{lemma}
\label{lem:gen-MsFEM-Gal-eff}
The matrices~$\mathds{A}^{\varepsilon,\mathsf{G}}$ and~$\mathds{A}^{\Pone,\mathsf{G}}$ are identical if the integrals in~\eqref{eq:gen-MsFEM-eff-form-Gal} are evaluated at the centroid of each mesh element~$K$ in the computation of~$\mathds{A}^{\Pone,\mathsf{G}}$. 
\end{lemma}

Contrary to the matrices, the right-hand sides of the effective problem~\eqref{eq:gen-FEM-effective-gal} and the Galerkin MsFEM~\eqref{eq:gen-MsFEM-OS} are not equal in general. We apply step~\ref{proc:noni-rhs} formulated at the beginning of this section: the right-hand side of the G-MsFEM is approximated by the right-hand side of the effective problem to obtain an approximate, but non-intrusive, MsFEM. The non-intrusive G-MsFEM becomes:
\begin{equation}
	\text{Find } u^\varepsilon_H \in V_{H,0}^\varepsilon \text{ such that } 
	\sum_{K \in \mesh} 
		a_K^{\varepsilon}\left(u_H^\varepsilon, \phiEps{j} \right) 
	= 
	F\left(\phiPone{j}\right) \quad \text{for all } 1 \leq j \leq N_{0}.
	\label{eq:gen-MsFEM-noni}
\end{equation}
This problem is no longer a Galerkin approximation of~\eqref{eq:diffusion-pde}, because different test spaces are used for the bilinear and for the linear form. In view of Lemma~\ref{lem:gen-MsFEM-Gal-eff}, the non-intrusive MsFEM can equivalently be formulated as
\begin{equation*}
	\text{compute } u_H \in V_{H,0} \text{ solution to~\eqref{eq:gen-FEM-effective-gal} and compute } u^\varepsilon_H \text{ from } u_H \text{ by~\eqref{eq:gen-MsFEM-post}},
\end{equation*}
provided all integrals in~\eqref{eq:gen-MsFEM-eff-form-Gal} are evaluated at the centroid for the construction of the matrix of the linear system in~\eqref{eq:gen-FEM-effective-gal}.

The latter formulation of the non-intrusive MsFEM immediately suggests how to effectively implement the non-intrusive MsFEM in a non-intrusive way similar to Algorithm~\ref{alg:msfem-diff-noni}. For completeness, we provide the algorithm for the non-intrusive G-MsFEM in Algorithm~\ref{alg:msfem-gen-noni}.

\begin{algorithm}[ht]
\caption{Non-intrusive G-MsFEM for the general framework}
\label{alg:msfem-gen-noni}
\begin{algorithmic}[1]
        \State Let $\mesh$ be the mesh used by the legacy code, let $\bullet \in \{\mathsf{e},\mathsf{c}\}$ be the chosen oversampling variant

    \medskip
    
    \ForAll{$K \in \mesh$}
        \For{$0 \leq \alpha \leq d$}
            \State Solve for the applicable $\VKOS{\alpha}{0}$ from Def.~\ref{def:corr-dofe} or~\ref{def:corr-dofc}
        \EndFor
		\State Compute the effective tensors defined in~\eqref{eq:gen-MsFEM-eff-coef-Gal}
	\EndFor
	
	\medskip
    
    \State Use the legacy code to construct the matrix~$\mathds{A}^\Pone$ by evaluating~\eqref{eq:gen-MsFEM-eff-form-Gal} at the centroid of each mesh element and to solve for~$u_H$ defined by~\eqref{eq:gen-FEM-effective-gal}
    \State Save $\left\{u_H(x_{c,K})\right\}_{K\in\mesh}$ and $\left\{ (\partial_\alpha u_H) \vert_K \right\}_{K\in\mesh, \, 1 \leq \alpha \leq d}$

    \medskip 
    
    \State Obtain the MsFEM approximation~$u^\varepsilon_H$ from~\eqref{eq:gen-MsFEM-post}
\end{algorithmic}
\end{algorithm}

The discussion surrounding Algorithm~\ref{alg:msfem-diff-noni} regarding the advantages for the implementation of this non-intrusive MsFEM approach also applies here.

Let us now comment on the well-posedness of the MsFEMs for the general framework introduced above. We recall that the hypotheses of the general framework without oversampling, or with \DOF-continuous oversampling, provide well-posedness of the G-MsFEM~\eqref{eq:gen-MsFEM-OS} by Lemma~\ref{lem:MsFEM-OS-well-posedness}. In this case, the non-intrusive approximation~\eqref{eq:gen-MsFEM-noni} is also well-posed, because the matrices associated to both MsFEM variants are the same. 
Regarding the PG-MsFEM~\eqref{eq:gen-MsFEM-OS-testP1}, we can only establish well-posedness if the associated matrix coincides with the matrix of the corresponding Galerkin MsFEM. This is stated in the following lemma, which generalizes Lemma~\ref{lem:IBP-bubbles} to the general framework.

\begin{lemma}
\label{lem:IBP-bubbles-gen}
Consider a G-MsFEM as defined by Def.~\ref{def:gen-MsFEM} without oversampling and suppose that the sampling form~$s^\varepsilon_K$ equals the local bilinear form~$a_K^\varepsilon$. Then the matrix associated to this G-MsFEM coincides with the matrix associated to the corresponding PG-MsFEM of Def.~\ref{def:gen-MsFEM-OS-testP1}. Consequently, the non-intrusive Galerkin MsFEM~\eqref{eq:gen-MsFEM-noni} coincides with the Petrov-Galerkin MsFEM~\eqref{eq:gen-MsFEM-OS-testP1} and in particular, the Petrov-Galerkin MsFEM is well-posed.
\end{lemma}

\begin{proof}
	To prove the lemma, we show that the matrices corresponding to the linear problems defined in~\eqref{eq:gen-MsFEM-noni} and~\eqref{eq:gen-MsFEM-OS-testP1} are equal. Using that $s_K^\varepsilon = a_K^{\varepsilon}$, we have for all $1 \leq i, \, j \leq N_{0}$,
	\begin{equation}
		\mathds{A}^{\varepsilon,\mathsf{G}}_{j,i} - \mathds{A}^\varepsilon_{j,i}
		=
		\sum_{K\in\mesh} a_K^{\varepsilon}\left(\phiEps{i},\phiEps{j} - \phiPone{j}\right)
		=
		\sum_{K\in\mesh} s_K^\varepsilon \left(\phiEps{i}, \phiEps{j}-\phiPone{j} \right)
		=0.
		\label{eq:IBP-bubbles}
	\end{equation}
The last equality stems from the fact that the multiscale basis functions satisfy~$\Gamma(K,\phiEps{i}) = \Gamma(K,\phiPone{i})$ for all~$K$, so that $\phiEps{j} - \phiPone{j} \in V_{K,0}$ (see~\eqref{eq:msbasis-OS-vf-dofe} with $S_K=K$ and recall Def.~\ref{def:sampling-space-OS-dofe} for the sampling test space~$V_{K,0}^ \varepsilon$), and the variational problem in~\eqref{eq:msbasis-OS-vf-dofe} (with $S_K=K$) shows that the above quantity vanishes.
\end{proof}

\subsection{Further extensions of the non-intrusive MsFEM}
\label{sec:noni-gen-more}
We sketch some other FEM settings to which we have applied the above strategy to develop non-intrusive MsFEM approaches. For more details, we refer to~\cite{biezemans_difficult_nodate}.

\textbf{Stabilized finite element formulations.} 
In the context of advection-diffusion problems, stabilized finite element formulations add mesh-dependent terms to a discrete variational formulation (such as~\eqref{eq:gen-MsFEM-OS}) to remove numerical instabilities, for example caused by sharp boundary layers of the exact solution. See~\cite{le_bris_numerical_2017} for such a variant of the MsFEM and see \cite{brooks_streamline_1982, hughes_new_1985, quarteroni_numerical_2017} for the stabilization of single-scale problems. The expansion~\eqref{eq:MsFEM-VVxy} can also be inserted in these additional terms to find a non-intrusive implementation of the associated MsFEM.

\textbf{Petrov-Galerkin formulations.}
Other test spaces than the $\Pone$ space~$V_{H,0}$ can be considered in Petrov-Galerkin formulations. An example would be to use multiscale test functions that locally solve the adjoint problem rather than the direct problem, introducing yet another bilinear form than~$s_K^\varepsilon$ in~\eqref{eq:msbasis-OS-vf-dofe} or~\eqref{eq:msbasis-OS-vf-dofc}. See e.g.~\cite{franca_recovering_2000}. An expansion of the kind~\eqref{eq:MsFEM-VVxy} can still be found for such test functions, with a suitably adapted definition of the numerical correctors. This way, a non-intrusive formulation can be found using the techniques of this work.

\textbf{Non-homogeneous Dirichlet conditions.}
Suppose that a legacy FEM code can provide a solution to an effective problem such as~\eqref{eq:gen-FEM-effective} posed on the space~$V_{H,0}$ and complemented with non-homogeneous Dirichlet conditions for~$u_H$ on~$\partial\Omega$. This solution can directly be used to construct a multiscale approximation~$u^\varepsilon_H \in V_H^\varepsilon$ from~\eqref{eq:gen-MsFEM-post}. The translation of the Dirichlet condition to the MsFEM approximation is as follows: if \DOF-continuous oversampling is applied, the function~$u^\varepsilon_H$ satisfies~$[\Gamma(K,u_H^\varepsilon)]_j = [\Gamma(K,u_H)]_j$ for all degrees of freedom associated to the boundary. Here, $[\Gamma(K,u_H)]_j$ is determined by the legacy code. When \DOF-extended oversampling is used, the degrees of freedom associated to the boundary are equal to the sum of $[\Gamma(K,u_H)]_j$ and a perturbation due to the fact that the degrees of freedom of the numerical correctors do not vanish.

\textbf{Neumann conditions.}
To apply Neumann conditions on~$\partial \Omega$, one solves a Galerkin approximation of the variational formulation in the space~$V_H^\varepsilon$. The suitable adaptation of~\eqref{eq:gen-MsFEM-OS} can be approximated by a non-intrusive Galerkin MsFEM following the same methodology as above. The effective $\Pone$ approximation that is obtained corresponds to the resolution of an effective PDE with Neumann conditions, for which a legacy code can be used. In the case of the diffusion problem~\eqref{eq:diffusion-pde}, the Neumann boundary condition in the effective problem is imposed on the effective flux~$\vec{n} \cdot \overline{A} \nabla u_H$, where~$\overline{A}$ is defined in~\eqref{eq:diffusion-eff-msfem-gal}.

\textbf{Parabolic equations.}
When a parabolic equation is discretized in time, problems of the form~\eqref{eq:gen-pb} are typically obtained for each time step, but with a right-hand side that depends on the solution of the previous time step. This term belongs to the space~$V_H^\varepsilon$, so it varies on the microscale and cannot be integrated numerically by the legacy code that operates on the coarse mesh. The non-intrusive strategy of the foregoing sections cannot be applied directly to find a non-intrusive MsFEM. In the vein of our non-intrusive approach, one could introduce an additional approximation by replacing the multiscale solution of the previous time step by its underlying~$\Pone$ representation in the $\Pone$ space~$V_H$. Studying the effect of this approximation is beyond the scope of the present work.

\subsection{Intrusiveness of other multiscale methods}
\label{sec:noni-gen-other}

Some work on the formulation of effective $\Pone$ problems in multiscale methods, and the related question of non-intrusive approaches, can be found in the literature. We discuss here the case of the HMM and the LOD method in the context of numerical homogenization, and provide some additional references to other fields at the end of the section.

First, the HMM is less intrusive than the original MsFEM, because its main objective is to approximate~$u^\varepsilon$ on the coarse scale. The HMM directly proposes to solve a $\Pone$ problem on the coarse scale, where effective coefficients of the $\Pone$ problem are defined in terms of the solutions to local problems. This workflow corresponds to our non-intrusive MsFEM approach, and when the local problems of the HMM coincide with the computation of the numerical correctors introduced in this work, the HMM and the MsFEM for the pure diffusion problem are identical. For more general problems, there is an important difference between the two methods. In the MsFEM, the form of the effective equation and the definition of the effective coefficients follows directly from the choice of basis functions, and thus from the choice of local problems. For the HMM, the local problems and the effective equation are formulated independently, and the link between the two is only justified heuristically, drawing inspiration from homogenization theory.

The LOD method aims at approximating~$u^\varepsilon$ at both the coarse and the microscale by the use of multiscale basis functions, like the MsFEM. It is shown in~\cite{gallistl_computation_2017} that a Petrov-Galerkin LOD method (see also~\cite{elfverson_multiscale_2015}) can, with some additional approximations, be recast as the $\Pone$ discretization of an appropriate coarse-scale problem. This opens the way to non-intrusive implementations in the spirit of the present article. The LOD method and the MsFEM notably differ in the fact that the LOD basis functions are defined on a patch around the vertices of the mesh that should generally be taken larger than the support of the associated $\Pone$ functions. In contrast, the MsFEM uses fully localized basis functions (even though they may have been computed using oversampling patches), each of which has the same support as the corresponding $\Pone$ basis functions.

The question of non-intrusive implementations of multiscale algorithms is an interesting and relevant question in many more fields of scientific computing than we can discuss here. Beyond the field of finite element methods, we mention the multiscale finite volume method~\cite{jenny_multi-scale_2003,hajibeygi_iterative_2008}, in which a non-intrusive coupling between the local and global computations is natural, since the local computations lead to transmissibilities that can be used in a separate, global finite volume simulator. 
Other than numerical homogenization methods, there are mixed finite element methods for multiscale modelling~\cite{chen_mixed_2002, arbogast_implementation_2002, arbogast_multiscale_2007} and domain decomposition techniques (such as the generalized FEM, patches of finite elements, numerical zoom; see~\cite{brezzi_chimera_2003, glowinski_finite_2005, apoung_kamga_numerical_2007}), for which non-intrusive approaches can e.g.~be found in~\cite{duval_non-intrusive_2016,gupta_analysis_2012}.
Finally, we would like to mention the reduced basis method for the efficient resolution of parameterized PDEs. Non-intrusive adaptations of this method (both for finite element and finite volume schemes) have been proposed and analyzed e.g.~in~\cite{chakir_two-grid_2009, chakir_non-intrusive_2019, grosjean_error_2021}.

\section{Comparison of the classical and non-intrusive MsFEM for diffusion problems}
	\label{sec:compare-gal-pg}

%%
%%%%%%%%%%%%%%%%%%%%%%%%%%%%%%%%%%
%% Comparison of Galerkin and non-intrusive MsFEM (mostly without OS), for the diffusion problem, by means of convergence results in the general setting and in the periodic setting
%%%%%%%%%%%%%%%%%%%%%%%%%%%%%%%%%%
%%

We study in this section a particular setting within the general MsFEM framework, namely that of MsFEMs for diffusion problems. We set in this section $a^\varepsilon_K = a^{\varepsilon,\dif}_K$ defined in Example~\ref{ex:diffusion-vf-gen}, and we choose the sampling form~$s_K^\varepsilon = a^{\varepsilon,\dif}_K$. 

\subsection{The general framework for diffusion problems}
For the convenience of the reader, we first give an explicit description of the simplifications of the general framework in the diffusion setting. In Def.~\ref{def:corr-dofe} and~\ref{def:corr-dofc} for the numerical correctors, Equation~\eqref{eq:MsFEM-gen-correctors-dofe} reduces to
\begin{equation}
	a_K^{\varepsilon,\dif} \left( \VSK{\alpha}{0}, w \right)
	=
	-a_K^{\varepsilon,\dif} \left( x^\alpha, w \right),
	\label{eq:diffusion-gen-correctors}
\end{equation}
for all $w\in V_{K,0}$ (where $V_{K,0}$ is the sampling test space for either the MsFEM-lin or the MsFEM-CR; see Examples~\ref{ex:msfem-lin-dofe} and~\ref{ex:msfem-cr-dofe})
when $1 \leq \alpha \leq d$, whereas~$\VSK{0}{0} = 0$. (The notation~$\VK{\alpha}$ will be used in the absence of oversampling, see Rem.~\ref{rem:DOFec-noOS}.) This means that $\VVK{0}=1$ in~\eqref{eq:MsFEM-VVxy}. Consequently, regarding the formulation of the effective $\Pone$ problem, only the effective diffusion coefficient does not vanish in~\eqref{eq:gen-MsFEM-eff-coef} and~\eqref{eq:gen-MsFEM-eff-coef-Gal}. Its definition in~\eqref{eq:gen-MsFEM-eff-coef-Gal} is identical to the formula in~\eqref{eq:diffusion-eff-msfem-gal} for the applicable choice of the numerical correctors.

The definition of the multiscale basis functions by~\eqref{eq:MsFEM-Vxy} reduces to~\eqref{eq:diffusion-MsFEM-Vxy-grad} (again upon replacing the numerical correctors~$\VK{\alpha}$ by the relevant ones for the MsFEM under consideration). Hence, we can associate a multiscale counterpart in~$V_H^\varepsilon$ to any~$v_H \in V_H$, given by
\begin{equation}
	v_H^\varepsilon = v_H 
	+ \sum_{K\in\mesh} \sum_{\alpha=1}^d 
		(\partial_\alpha v_H)\vert_K \, \VKOS{\alpha}{0}.
\label{eq:diffusion-MsFEM-Vxy}
\end{equation}

The non-intrusive MsFEM~\eqref{eq:gen-MsFEM-noni} becomes
\begin{equation}
	\text{Find } u^\varepsilon_H \in V_{H,0}^\varepsilon \text{ such that } 
	\sum_{K \in \mesh} 
		a_K^{\varepsilon,\dif}\left(u_H^\varepsilon, \phiEps{j} \right) 
	= 
	F\left(\phiPone{j}\right) \quad \text{for all } 1 \leq j \leq N_{0}.
	\label{eq:diffusion-gen-MsFEM-noni}
\end{equation}
Lemma~\ref{lem:IBP-bubbles-gen} now amounts to the following.
\begin{lemma}
\label{lem:IBP-bubbles-diffusion}
	Let `MsFEM' refer to the MsFEM-lin or the MsFEM-CR, both without oversampling. The non-intrusive Galerkin MsFEM~\eqref{eq:diffusion-gen-MsFEM-noni} coincides with the following Petrov-Galerkin MsFEM:
	\begin{equation}
		\text{Find } u^\varepsilon_H \in V_{H,0}^\varepsilon \text{ such that } 
		\sum_{K \in \mesh} 
			a_K^{\varepsilon,\dif} \left(u_H^\varepsilon, \phiPone{j} \right) 
		= 
		F \left( \phiPone{j} \right) \quad \text{for all } 1 \leq j \leq N_{0},
		\label{eq:diffusion-gen-MsFEM-testP1}	
	\end{equation}
\end{lemma}

We will specify for all results in this section to which specific MsFEMs they apply among the MsFEM-lin and the MsFEM-CR, with or without oversampling. Lemmas~\ref{lem:diffusion-eff-coer}, \ref{lem:MsFEM-Vxy-no-zero-char} and \ref{lem:estim-Gal-PG-H} are generalizations of results in~\cite{biezemans_non-intrusive_2023}, where the MsFEM-lin without oversampling is considered.

\subsection{Convergence results}
\label{sec:compare-gal-pg-estim}
We estimate here the difference between the solutions to the (intrusive) Galerkin approximation~\eqref{eq:gen-MsFEM-OS} and the non-intrusive MsFEM~\eqref{eq:diffusion-gen-MsFEM-noni}, which coincides with the Petrov-Galerkin MsFEM~\eqref{eq:diffusion-gen-MsFEM-testP1}. We first show coercivity of the effective diffusion tensor~$\overline{A}$. 

\begin{lemma}
\label{lem:diffusion-eff-coer}
	Consider the MsFEM-lin or the MsFEM-CR, without oversampling, or the MsFEM-CR with \DOF-continuous oversampling. The effective tensor~$\overline{A}$ defined by~\eqref{eq:diffusion-eff-msfem-gal} with the appropriate numerical correctors satisfies
	\begin{equation*}
		\forall \, \xi \in \bbR^d,
		\quad
		m |\xi|^2 \leq \xi \cdot \overline{A} \, \xi.
	\end{equation*}
	Here, $m$ is the same coercivity constant as in~\eqref{ass:bounds}.
\end{lemma}

\begin{proof}
Let $\xi = (\xi_1,\dots,\,\xi_d) \in \bbR^d$, and let $K$ be any simplex of the mesh $\mesh$. We have
\begin{equation*}
    |K| \, \xi \cdot \left. \overline{A} \right\vert_K \xi
    =
    \sum_{\alpha,\beta=1}^d a^{\varepsilon,\dif}_K \left(
    \xi_\alpha \left(x^\alpha + \VK{\alpha}\right),\,
    \xi_\beta \left(x^\beta + \VK{\beta} \right)
    \right)
    = 
    \int_K (\xi + \nabla \chi^\xi) \cdot A^\varepsilon (\xi + \nabla \chi^\xi),
\end{equation*}
denoting by $\chi^\xi$ the function $\displaystyle \chi^\xi = \sum_{\alpha=1}^d \xi_\alpha \VK{\alpha}$. Using~\eqref{ass:bounds}, we obtain
\begin{equation*}
    |K| \, \xi \cdot \left. \overline{A} \right\vert_K \xi
	\geq 
    m \int_K \left\vert \xi + \nabla \chi^\xi \right\rvert^2
	\geq 
    m\,|K| \, |\xi|^2 + 
    2 m\,\int_K \xi \cdot \nabla \chi^\xi.
\end{equation*}
Using an integration by parts, we see that $\displaystyle \int_K \xi \cdot \nabla \chi^\xi = \int_{\partial K} \chi^\xi \, n\cdot\xi$, where~$n$ is the unit outward normal vector on~$\partial K$.
In the case of the MsFEM-lin, the function~$\chi^\xi$ vanishes on~$\partial K$. In the case of the MsFEM-CR with \DOF-continuous oversampling, or without oversampling, the function~$\chi^\xi$ has average zero on each face of~$K$. Since the factor~$n\cdot\xi$ is constant on each face, the integral again vanishes. In conclusion, we have~$\displaystyle \int_K \xi \cdot \nabla \chi^\xi = 0$.

We thus obtain the inequality
$    \xi \cdot \left. \overline{A} \right\vert_K \xi 
    \geq m |\xi|^2.
$
Since $K \in \mesh$ is arbitrary here, this shows coercivity of $\overline{A}$ and completes the proof.
\end{proof}

Coercivity of the effective tensor~$\overline{A}$ implies coercivity of the bilinear form~$\overline{a}^\dif$ on~$H^1_0(\Omega)$. By an application of the Lax-Milgram Theorem, we conclude that the (continuous) effective problem~\eqref{eq:diffusion-pde-effective} is well-posed for the MsFEM-lin and the MsFEM-CR without oversampling, and for the MsFEM-CR with \DOF-continuous oversampling. 

\begin{remark}
	The proof of the above lemma does not extend to the MsFEM-lin with oversampling, because there is no global information about~$\VK{\alpha}$ on the faces of~$K$.
\end{remark}

The following lemma provides a variational characterization of the bijection~\eqref{eq:diffusion-MsFEM-Vxy}.

\begin{lemma}
\label{lem:MsFEM-Vxy-no-zero-char}
	Consider the MsFEM-lin or the MsFEM-CR, both without oversampling. Let $v_H^\varepsilon\in V_H^\varepsilon$. The unique $v_H \in V_H$ for which~\eqref{eq:diffusion-MsFEM-Vxy} holds, is the unique solution in~$V_H$ to the problem
\begin{equation}
	\forall \, w_H \in V_H,
	\quad
	\overline{a}^\dif{\left( v_H,w_H \right)}
	=
	a^{\varepsilon,\dif}{\left( v_H^\varepsilon,w_H \right)}.
	\label{eq:diffusion-MsFEM-Vxy-var-char}
\end{equation}
In addition, we have, with the constants~$m$ and $M$ from~\eqref{ass:bounds}, the estimate
\begin{equation*}
	\lVert \nabla v_H \rVert_{L^2(\mesh)} 
	\leq
	\frac{M}{m} \lVert \nabla v_H^\varepsilon \lVert_{L^2(\mesh)}.
\end{equation*}
\end{lemma}

\begin{proof}
	Let~$v_H \in V_H$ be the unique element of~$V_H$ such that $v_H^\varepsilon$ and $v_H$ satisfy~\eqref{eq:diffusion-MsFEM-Vxy}. Take any $w_H \in V_H$. Using that~$\nabla v_H$ and~$\nabla w_H$ are piecewise constant, we compute
	\begin{equation*}
		a^{\varepsilon,\dif}(v_H^\varepsilon,w_H)
		=
		\sum_{K\in\mesh} \sum_{\alpha,\beta=1}^d
			(\partial_\beta w_H)\vert_K \,
			a^{\varepsilon,\dif}_K{ \left( x^\alpha + \VK{\alpha}, x^\beta \right)}
			(\partial_\alpha v_H)\vert_K.
	\end{equation*}
	For the MsFEM without oversampling, the numerical correctors belong to the sampling test space~$V_{K,0}$. We can thus use~\eqref{eq:diffusion-gen-correctors} to obtain
	\begin{equation*}
	\forall \, 1 \leq \alpha, \, \beta \leq d,
	\quad	
	a^{\varepsilon,\dif}_K{ \left( x^\alpha + \VK{\alpha}, x^\beta \right)}
	=
	a^{\varepsilon,\dif}_K{ \left( x^\alpha + \VK{\alpha}, x^\beta + \VK{\beta} \right)}.
	\end{equation*}
	Using the definitions of~$\overline{A}$ in~\eqref{eq:diffusion-eff-msfem-gal} and of~$\overline{a}^{\dif}$ in~\eqref{eq:diffusion-form-effective} (we recall that these expressions hold true here upon replacing the numerical correctors by those under consideration), we conclude that
	\begin{equation*}
		a^{\varepsilon,\dif}{\left( v_H^\varepsilon,w_H \right)}
		=
		\sum_{K\in\mesh} \sum_{\alpha,\beta=1}^d \int_K
			\partial_\beta w_H \,
			\overline{A}_{\beta,\alpha} \,
			\partial_\alpha v_H
		=
		\overline{a}^{\dif}(v_H,w_H).
	\end{equation*}
It follows that~$v_H$ satisfies \eqref{eq:diffusion-MsFEM-Vxy-var-char}. In addition, in view of the coercivity of~$\overline{A}$ established in Lemma~\ref{lem:diffusion-eff-coer} and by the Lax-Milgram Theorem, problem \eqref{eq:diffusion-MsFEM-Vxy-var-char} uniquely characterizes~$v_H$. 

The estimate on~$v_H$ follows by testing the characterization~\eqref{eq:diffusion-MsFEM-Vxy-var-char} against $w_H=v_H$. This yields
\begin{align*}
	m \left\lVert \nabla v_H \right\rVert_{L^2(\mesh)}^2
	\leq 
	\overline{a}^{\dif}(v_H,v_H)
	=
	a^{\varepsilon,\dif}{ \left( v_H^\varepsilon,v_H \right)}
	&=	
	\sum_{K\in\mesh} \int_K \nabla v_H \cdot A^\varepsilon \, \nabla v_H^\varepsilon
	\\
	&\leq 
	M \sum_{K\in\mesh} \left\lVert \nabla v_H \right\rVert_{L^2(K)} \, \left\lVert \nabla v_H^\varepsilon \right\rVert_{L^2(K)}.
\end{align*}	
The first inequality follows from coercivity of~$\overline{A}$ and the second inequality from the upper bound on~$A^\varepsilon$ in~\eqref{ass:bounds} and the Cauchy-Schwarz inequality. 
With a discrete Cauchy-Schwarz inequality, we obtain
\begin{equation*}
	m \left\lVert \nabla v_H \right\rVert_{L^2(\mesh)}^2
	\leq
	M \sum_{K\in\mesh} \left\lVert \nabla v_H \right\rVert_{L^2(K)} \, \left\lVert \nabla v_H^\varepsilon \right\rVert_{L^2(K)}
	\leq 
	M \, \left\lVert \nabla v_H \right\rVert_{L^2(\mesh)} \, \left\lVert \nabla v_H^\varepsilon \right\rVert_{L^2(\mesh)}.
\end{equation*}
The proof is completed upon simplifying by~$\left\lVert \nabla v_H \right\rVert_{L^2(\mesh)}$.
\end{proof}

For the remainder of this section, we consider MsFEMs without oversampling. Let~$u^{\varepsilon,\mathsf{G}}_H$ denote the solution to the MsFEM approximation~\eqref{eq:gen-MsFEM-OS} (we use the superscript~$\mathsf{G}$ to stress that this is a \emph{Galerkin} approximation) and let~$u^{\varepsilon,\mathsf{PG}}_H$ denote the solution to the non-intrusive MsFEM~\eqref{eq:diffusion-gen-MsFEM-noni} (which is equivalent to the \emph{Petrov-Galerkin} MsFEM~\eqref{eq:diffusion-gen-MsFEM-testP1}, since we do not apply the oversampling technique).

We first study the error $u^{\varepsilon,\mathsf{G}}_H - u^{\varepsilon,\mathsf{PG}}_H$ when $\varepsilon\to0$. In this case, we do not need a rate of convergence in~$H$ and we shall relax the condition~$f \in L^2(\Omega)$ to the condition $f \in H^{-1}(\Omega)$. Then the definition of the linear form~$F$ in~\eqref{eq:diffusion-bilin-form} has to be adapted. Given~$f \in H^{-1}(\Omega)$, there exist $f_0,f_1,\dots,f_d \in L^2(\Omega)$ such that 
$
	\displaystyle
	F(v)
	=
	\sum_{K\in\mesh} \left( \int_K f_0 \, v +
		\sum_{\beta=1}^d \int_K f_\beta \, \partial_\beta v
	\right),
$
which is in fact well-defined for any $v \in H^1(\mesh)$ and thus in particular on~$V_H$, the underlying affine space for the MsFEM, and the multiscale space~$V_H^\varepsilon$.

We consider in Lemma~\ref{lem:estim-Gal-PG-eps} a sequence of diffusion tensors~$A^\varepsilon$ that $H$-converges to a constant diffusion tensor. This means that~$u^\varepsilon$ converges weakly in~$H^1(\Omega)$ as $\varepsilon\to0$ towards a function~$u^\star \in H^1_0(\Omega)$, solution to the homogenized problem~\eqref{eq:diffusion-hom-pde}, and $A^\varepsilon\nabla u^\varepsilon \rightharpoonup A^\star \nabla u^\star$ weakly in~$L^2(\Omega)$.

\begin{lemma}
	\label{lem:estim-Gal-PG-eps}
	Consider the MsFEM-lin or the MsFEM-CR, both without oversampling. Suppose that $(A^\varepsilon)_{\varepsilon > 0}$ is a sequence of matrices satisfying~\eqref{ass:bounds} that $H$-converges to a constant matrix. Let~$f \in H^{-1}(\Omega)$. Then $\left\lVert u^{\varepsilon,\mathsf{G}}_H - u^{\varepsilon,\mathsf{PG}}_H \right\rVert_{H^1(\mesh)} \to 0$ as $\varepsilon\to0$. 
\end{lemma}

\begin{remark}
	\label{rem:estim-eps-rate}
	A rate of convergence can be obtained under some additional structural assumptions on~$A^\varepsilon$; see Lemma~\ref{lem:estim-Gal-PG-per}.
\end{remark}

We need a few auxiliary results to establish Lemma~\ref{lem:estim-Gal-PG-eps}. The first result below concerns the convergence of the numerical correctors as~$\varepsilon\to0$.

\begin{lemma}
	Suppose that~$A^\varepsilon$ $H$-converges to a constant homogenized matrix~$A^\star$. Consider the MsFEM-lin or the MsFEM-CR, both without oversampling. Then, for all $K\in\mesh$ and all $1\leq\alpha\leq d$, we have $\VK{\alpha} \rightharpoonup 0$ weakly in~$H^1(K)$ as~$\varepsilon\to0$.
	\label{lem:corr-to-0}
\end{lemma}

\begin{proof}
	We introduce for each~$\alpha=1,\dots,d$ the function~$\tau^{\varepsilon,\alpha} = x^\alpha + \VK{\alpha}$. Then~\eqref{eq:diffusion-gen-correctors} implies the equation
	$
		-{\operatorname{div}(A^\varepsilon \nabla \tau^{\varepsilon,\alpha})}
		=
		0
	$ in~$K$.
	For the MsFEM-lin, the boundary conditions of the local problems for~$\VK{\alpha}$ (see~\eqref{eq:msfem-lin-dofe} with $S_K=K$) lead to~$\tau^{\varepsilon,\alpha}=x^\alpha$ on~$\partial K$. The boundary conditions associated to the MsFEM-CR follow from~\eqref{eq:msfem-cr-dofe} and are as follows: the flux $\vec{n}\cdot A^\varepsilon \nabla \tau^{\varepsilon,\alpha}$ is constant on each face of~$K$ (but may depend on~$\varepsilon$) and $\displaystyle \int_h \tau^{\varepsilon,\alpha} = \int_h x^\alpha$ for all faces~$h$ of~$K$.

	It follows that the homogenized limit~$\tau^{\star,\alpha}$ of~$\tau^{\varepsilon,\alpha}$ satisfies the equation
	$
		-{\operatorname{div}(A^\star \nabla \tau^{\star,\alpha})}
		=
		0
	$ in~$K$.
	For the MsFEM-lin, the boundary condition for the homogenized problem is~$\tau^{\star,\alpha}=x^\alpha$ on~$\partial K$. The boundary conditions associated to the MsFEM-CR are a constant flux $\vec{n}\cdot A^\star \nabla \tau^{\alpha,\star}$ on each face of~$K$ and $\displaystyle \int_h \tau^{\star,\alpha} = \int_h x^\alpha$ for all faces~$h$ of~$K$.
	
	Both for the MsFEM-lin and the MsFEM-CR, the homogenized equation has a unique solution, which is easily seen to be~$\tau^{\star,\alpha} = x^\alpha$, because~$A^\star$ is constant. Therefore, $\tau^{\star,\alpha} \rightharpoonup x^\alpha$ weakly in~$H^1(K)$. Subtracting the function~$x^\alpha$, we deduce the desired convergence.
\end{proof}

We will also use the following result, which is a straightforward generalization of the extended Poincaré inequality in~\cite[Lemma 3.31]{ern_theory_2004}. 

\begin{lemma}
\label{lem:poincare-extended}
Let~$W$ be the subspace of~$H^1(\mesh)$ defined by 
\begin{equation*}
	W = \left\{
		v \in H^1(\mesh) \, \left\vert \, 
		\int_h \llbracket v \rrbracket = 0 \right. \text{ for each face } h \text{ of } \fmesh, \,
		\int_h v = 0 \text{ for each face } h \subset \partial \Omega
	\right\}.
\end{equation*}
There exists a constant $C>0$ depending only on~$\Omega$ but not on~$H$ such that
\begin{equation*}
	\forall \, v \in W,
	\qquad
	\lVert v \rVert_{L^2(\Omega)} \leq C \, \lVert \nabla v \rVert_{L^2(\mesh)}.
\end{equation*}
\end{lemma}

Note that the multiscale space~$V_{H,0}^\varepsilon$ is contained in~$W$ for both the MsFEM-lin and the MsFEM-CR, without oversampling. Finally, we provide a number of useful bounds for the difference between~$u^{\varepsilon,\mathsf{G}}_H$ and~$u^{\varepsilon,\mathsf{PG}}_H$.

\begin{lemma}
	Let $f \in H^{-1}(\Omega)$ and consider the MsFEM-lin or the MsFEM-CR, both without oversampling. Let $e^\varepsilon_H = u^{\varepsilon,\mathsf{G}}_H - u^{\varepsilon,\mathsf{PG}}_H$. There exists a unique $e_H^\Pone\in V_H$ and a linear combination of the numerical correctors, that we denote by $e_H^{\mathsf{osc}}$, such that $e^\varepsilon_H = e^\Pone_H + e^\mathsf{osc}_H$, and it holds, with the constants $m,M$ from~\eqref{ass:bounds} and the constant~$C$ from Lemma~\ref{lem:poincare-extended},
	\begin{align}
		a^{\varepsilon,\dif}{\left(e^\varepsilon_H, \, e^\varepsilon_H \right)}
		&=
		F(e^\mathsf{osc}_H),
		\label{eq:estim-Gal-PG-Cea}\\
		\left\lVert \nabla e^\mathsf{osc}_H \right\rVert_{L^2(K)}
		&\leq
		\frac{M}{m} \left\lVert \nabla e^\Pone_H \right\rVert_{L^2(K)}
		\quad 
		\text{for all } K \in \mesh,
		\label{eq:estim-gradient-error3}\\
		\left\lVert \nabla e^\mathsf{\Pone}_H \right\rVert_{L^2(\mesh)}
		&\leq
		\frac{M}{m} \left\lVert \nabla e^\varepsilon_H \right\rVert_{L^2(\mesh)},
		\label{eq:estim-gradient-error1}\\
		\left\lVert \nabla e^\varepsilon_H \right\rVert_{L^2(\mesh)}
		&\leq
		\sqrt{1+C^2} \frac{M^2}{m^3} \lVert F \rVert_{\mathcal{L}(H^1(\mesh))},
		\label{eq:estim-gradient-error2}
	\end{align}
	where~$\lVert \cdot \rVert_{\mathcal{L}(H^1(\mesh))}$ is the operator norm on~$\mathcal{L}(H^1(\mesh))$.
	\label{lem:estim-gradient-error}
\end{lemma}

\begin{proof}
	Since the numerical approximations~$u^{\varepsilon,\mathsf{G}}_H$ and~$u^{\varepsilon,\mathsf{PG}}_H$ both belong to the multiscale approximation space $V_H^\varepsilon$, it follows that $e^\varepsilon_H \in V_H^\varepsilon$, and we are in a position to use~\eqref{eq:diffusion-MsFEM-Vxy}: there exists a unique $e^\Pone_H \in V_H$ such that
	\begin{equation}
		e^\varepsilon_H = e^\Pone_H + e^\mathsf{osc}_H, \quad 
		e^\mathsf{osc}_H = \sum_{K\in\mesh} \sum_{\alpha=1}^d \left.\left(\partial_\alpha e^\Pone_H\right)\right\vert_K \, \VK{\alpha}.
	\label{eq:estim-Gal-PG-error}
	\end{equation}
	Applying Lemma~\ref{lem:MsFEM-Vxy-no-zero-char} to~$v_H^\varepsilon = e^\varepsilon_H$,we immediately obtain~\eqref{eq:estim-gradient-error1}.

	Now recall that the numerical correctors are defined by~\eqref{eq:diffusion-gen-correctors}. Using the fact that~$\nabla e_H^\Pone$ is piecewise constant, this implies that~$e^{\mathsf{osc}}_H$ satisfies the following variational problem in each~$K\in\mesh$:
	\begin{equation*}
		\forall \, w \in V_{K,0}, \quad
		a^{\varepsilon,\dif}_K \left( e^\mathsf{osc}_H, \, w \right) 
		= 
		-a^{\varepsilon,\dif}_K{ \left(
			e^\Pone_H, \, w
		\right)}.
	\end{equation*}
	Without oversampling, it holds $\VK{\alpha} \in V_{K,0}$ for each $1\leq\alpha\leq d$, so~$e^\mathsf{osc}_H$ can be used as a test function here. With the bounds in~\eqref{ass:bounds}, implying continuity and coercivity of~$a_K^{\varepsilon,\dif}$, we obtain~\eqref{eq:estim-gradient-error3}.

	Next using~\eqref{eq:estim-Gal-PG-error}, we can write
	\begin{equation*}
		a^{\varepsilon,\dif}{\left( e^\varepsilon_H, \, e^\varepsilon_H \right)}
		=
		a^{\varepsilon,\dif}{\left( u^{\varepsilon,\mathsf{G}}_H, \, e^\varepsilon_H \right)} -
		a^{\varepsilon,\dif}\left( u^{\varepsilon,\mathsf{PG}}_H, \, e^\Pone_H \right) -
		a^{\varepsilon,\dif}{\left( u^{\varepsilon,\mathsf{PG}}_H, \, e^\mathsf{osc}_H \right)}.
	\end{equation*}
	We deduce from~\eqref{eq:IBP-bubbles} that
	$
		a^{\varepsilon,\dif}{ \left(u^{\varepsilon,\mathsf{PG}}_H, \, e^\mathsf{osc}_H \right)} = 0
	$.
	Since~$e^\varepsilon_H$ can be used as a test function in the discrete problem~\eqref{eq:gen-MsFEM-OS} and~$e^\Pone_H$ in~\eqref{eq:diffusion-gen-MsFEM-testP1}, we have
	$
	a^{\varepsilon,\dif}{\left( u^{\varepsilon,\mathsf{G}}_H, \, e^\varepsilon_H \right)} -
	a^{\varepsilon,\dif}{\left( u^{\varepsilon,\mathsf{PG}}_H, \, e^\Pone_H \right)}
	=
	F\left( e_H^\mathsf{osc} \right)
	$,
	which shows~\eqref{eq:estim-Gal-PG-Cea}. It follows that
	\begin{equation*}
		a^{\varepsilon,\dif}{ \left( e^\varepsilon_H, \, e^\varepsilon_H \right)}
		\leq 
		\lVert F \rVert_{\mathcal{L}(H^1(\mesh))} \, \left\lVert e^\mathsf{osc}_H \right\rVert_{H^1(\mesh)}
		\leq 
		\lVert F \rVert_{\mathcal{L}(H^1(\mesh))} \, \sqrt{1+C^2} \left\lVert \nabla e^\mathsf{osc}_H \right\rVert_{L^2(\mesh)},
	\end{equation*}
	where $C$ is the Poincaré constant from Lemma~\ref{lem:poincare-extended}.
	Now applying~\eqref{eq:estim-gradient-error3} and~\eqref{eq:estim-gradient-error1} on the right, and using coercivity of~$a^{\varepsilon,\dif}$ on the left, we find
	\begin{equation*}
		m \left\lVert \nabla e^\varepsilon_H \right\rVert_{L^2(\mesh)}^2
		\leq 
		\sqrt{1+C^2} \left( \frac{M}{m} \right)^2 \,
		\lVert F \rVert_{\mathcal{L}(H^1(\mesh))} \, \left\lVert \nabla e^\varepsilon_H \right\rVert_{L^2(\mesh)},
	\end{equation*}
	from which we deduce~\eqref{eq:estim-gradient-error2}.
\end{proof}

\begin{proof}[Proof of Lemma~\ref{lem:estim-Gal-PG-eps}]
Let $e^\varepsilon_H = u^{\varepsilon,\mathsf{G}}_H - u^{\varepsilon,\mathsf{PG}}_H$. We will use~\eqref{eq:estim-Gal-PG-error}. By Lemma~\ref{lem:estim-gradient-error}, we have~\eqref{eq:estim-Gal-PG-Cea}. Combined with~\eqref{ass:bounds} and Lemma~\ref{lem:poincare-extended}, this implies
\begin{equation*}
    C \left\lVert u^{\varepsilon,\mathsf{G}}_H - u^{\varepsilon,\mathsf{PG}}_H \right\rVert_{H^1(\mesh)}^2
    \leq 
    a^{\varepsilon,\dif}{\left(e^\varepsilon_H, \, e^\varepsilon_H\right)}
	=
	F\left(e^\mathsf{osc}_H\right)
    =
	\sum_{K\in\mesh} \sum_{\alpha=1}^d \left.\left(\partial_\alpha e^\Pone_H\right)\right\vert_K \, F\left(\VK{\alpha}\right).
\end{equation*}
By Lemma~\ref{lem:corr-to-0}, we know that $\VK{\alpha}\rightharpoonup0$ as $\varepsilon\to0$ weakly in~$H^1(K)$ for each~$K$ and for each~$\alpha$. Therefore, $F\left(\VK{\alpha}\right) \to0$ as $\varepsilon\to0$. In view of~\eqref{eq:estim-gradient-error1} and~\eqref{eq:estim-gradient-error2}, every derivative~$\left.\left(\partial_\alpha e^\Pone_H\right)\right\vert_K$ is bounded independently of~$\varepsilon$. It follows that $F\left(e^\mathsf{osc}_H\right)\to0$ as $\varepsilon\to0$. The conclusion now follows from the above inequality.
\end{proof}

We next study the convergence of~$u^{\varepsilon,\mathsf{G}}_H - u^{\varepsilon,\mathsf{PG}}_H$ as~$H\to0$. To this end, we return to the original hypotheses of Sec.~\ref{sec:msfem-diffusion}, i.e., $f\in L^2(\Omega)$. Note that for the next result, the additional convergence hypothesis of Lemma~\ref{lem:estim-Gal-PG-eps} for~$A^\varepsilon$ is not needed.

\begin{lemma}
	\label{lem:estim-Gal-PG-H}
	Consider the MsFEM-lin or the MsFEM-CR, both without oversampling.
	Assume that~$f \in L^2(\Omega)$. Then there exists a constant~$C$ independent of~$\varepsilon$, $H$ and~$f$ such that
	\begin{equation*}
		\left\lVert u^{\varepsilon,\mathsf{G}}_H - u^{\varepsilon,\mathsf{PG}}_H \right\rVert_{H^1(\mesh)}
		\leq
		C H \lVert f \rVert_{L^2(\Omega)}.
	\end{equation*}
\end{lemma}

To prove this lemma, we will use some Poincaré-Friedrichs inequalities, for which we refer e.g.~to~\cite[Lemma 4.3]{le_bris_msfem_2013},~\cite[Lemma B.66]{ern_theory_2004}. 

\begin{proof}
Let $e^\varepsilon_H = u^{\varepsilon,\mathsf{G}}_H - u^{\varepsilon,\mathsf{PG}}_H$ and recall the results of Lemma~\ref{lem:estim-gradient-error}. We have $e^\varepsilon_H = e^\Pone_H + e^\mathsf{osc}_H$ (see~\eqref{eq:estim-Gal-PG-error}), and~\eqref{eq:estim-Gal-PG-Cea} provides, for $f\in L^2(\Omega)$, the equality
$
    a^{\varepsilon,\dif}{ \left(e^\varepsilon_H, \, e^\varepsilon_H \right)}
    =
    \left(f,e^\mathsf{osc}_H \right)_{L^2(\Omega)}.
$
Hence, by the Cauchy-Schwarz inequality,
\begin{equation}
    a^{\varepsilon,\dif}{\left( e^\varepsilon_H, \, e^\varepsilon_H \right)}
    \leq
    \lVert f \rVert_{L^2(\Omega)} \left\lVert e^\mathsf{osc}_H \right\rVert_{L^2(\Omega)}.
	\label{eq:estim-Gal-PG-estimate0}
\end{equation}

For the MsFEM-lin (without oversampling), it holds that $\VK{\alpha}=0$ on~$\partial K$ for all mesh elements~$K$ and all $1 \leq \alpha \leq d$, and it follows that $e^{\mathsf{osc}}_H=0$ on the boundaries of all mesh elements. In the case of the MsFEM-CR (without oversampling), it holds that $\displaystyle \int_h \VK{\alpha} =0$ for all faces~$h$ of the mesh and all~$1 \leq \alpha \leq d$. (Note that the average of~$\VK{\alpha}$ over any face~$h$ is well-defined even if~$\VK{\alpha}$ is in general discontinuous along faces.) Since $\partial_\alpha e_H^\Pone$ is constant on each mesh element~$K$, we also have $\displaystyle \int_h e_H^{\mathsf{osc}} =0$. Hence, both for the MsFEM-lin and for the MsFEM-CR, an appropriate variant of the Poincaré-Friedrichs inequality yields a constant~$C$ independent of~$K$ but dependent on the regularity of the mesh, such that
\begin{equation}
    \left\lVert e^\mathsf{osc}_H \right\rVert_{L^2(K)} \leq 
    CH \left\lVert \nabla e^\mathsf{osc}_H \right\rVert_{L^2(K)}.
    \label{eq:estim-Gal-PG-estimate1}
\end{equation}
Upon inserting the inequalities~\eqref{eq:estim-Gal-PG-estimate1}, \eqref{eq:estim-gradient-error3} and~\eqref{eq:estim-gradient-error1} into~\eqref{eq:estim-Gal-PG-estimate0}, it follows that
\begin{equation*}
    a^{\varepsilon,\dif}{ \left( e^\varepsilon_H,e^\varepsilon_H \right)}
    \leq 
    CH \left(\frac{M}{m}\right)^2 \left\lVert \nabla e^\varepsilon_H \right\rVert_{L^2(\mesh)} \,
    \lVert f \rVert_{L^2(\Omega)}.
\end{equation*}
One more time using the lower bound in~\eqref{ass:bounds}, we find
\begin{equation*}
    \left\lVert \nabla e^\varepsilon_H \right\rVert_{L^2(\mesh)}
    \leq 
    CH \frac{M^2}{m^3} \lVert f \rVert_{L^2(\Omega)}.
\end{equation*}
The proof is concluded by application of Lemma~\ref{lem:poincare-extended} to~$e_H^\varepsilon$.
\end{proof}

\subsection{Convergence results in the periodic setting}
\label{sec:homogenization-convergence}

We now study the MsFEM-lin applied to the periodic setting introduced in Sec.~\ref{sec:homogenization} in some more detail. To the best of our knowledge, all convergence results known for the MsFEM are obtained in this periodic setting (see e.g.~\cite{efendiev_convergence_2000, hou_removing_2004, efendiev_multiscale_2009, allaire_multiscale_2005,hesthaven_high-order_2014, le_bris_msfem_2013,le_bris_msfem_2014,le_bris_multiscale_2014,le_bris_multiscale_2019}). The analysis in these works relies on the explicit description of the microstructure that we summarized in Sec.~\ref{sec:homogenization}. In particular, recall the existence of a homogenized diffusion coefficient given by~\eqref{eq:diffusion-hom-coef} and the first-order two-scale expansion~\eqref{eq:expansion-1-hom}. We emphasize, however, that the \emph{application} of the MsFEM does not require the periodic setting, nor does it even suppose the PDE under consideration to be embedded in a sequence of PDEs for a family of parameters $\varepsilon$ that tend~0. We refer to Sec.~\ref{sec:gen-num} for examples of such numerical experiments.

Applying the MsFEM to a sequence of matrices~$A^\varepsilon = A^\mathsf{per}(\cdot/\varepsilon)$, we obtain a sequence of effective tensors~$\overline{A}(\varepsilon)$. Each~$\overline{A}(\varepsilon)$ is defined by~\eqref{eq:diffusion-eff-msfem-gal} for a fixed value of~$\varepsilon$. We have the following convergence result.

\begin{lemma}
\label{lem:eff-coef-conv}
Let~$\overline{A}(\varepsilon)$ be the sequence of effective tensors obtained in~\eqref{eq:diffusion-eff-msfem-gal} by applying the MsFEM-lin without oversampling to $A^\varepsilon = A^{\mathsf{per}}(\cdot/\varepsilon)$. We have~$\overline{A}(\varepsilon) \to A^\star$ as~$\varepsilon\to0$.
\end{lemma}

\begin{proof}
We fix a mesh element~$K \in \mesh$.
First observe that~$\overline{A}(\varepsilon)$ and~$A^\star$ satisfy
\begin{equation}
	\left. \overline{A}_{\beta,\alpha}(\varepsilon) \right\vert_K 
    =
    \frac{1}{|K|} a^{\varepsilon,\dif}_K{ \left(x^\alpha + \VK{\alpha},\, x^\beta \right)},
    \qquad 
	A^\star_{\beta,\alpha} 
    =
    \int_Q e_\beta \cdot A^\mathsf{per}(e_\alpha + \nabla w_\alpha),
\label{eq:diffusion-eff-msfem-pg}
\end{equation}
for each $1 \leq \alpha,\beta \leq d$, in view of the variational formulations satisfied by~$\VK{\alpha}$ (solution to the PDE~\eqref{eq:diffusion-MsFEM-correctors}) and~$w_\alpha$ (solution to the PDE~\eqref{eq:diffusion-correctors}).
We recall that~$Q$ is the unit cube of~$\bbR^d$.

Now let~$\tau^{\varepsilon,\alpha} = x^\alpha + \VK{\alpha}$. In view of Lemma~\ref{lem:corr-to-0}, $\tau^{\varepsilon,\alpha} \rightharpoonup \tau^{\star,\alpha}$ as~$\varepsilon\to0$ weakly in~$H^1(K)$, with $\tau^{\star,\alpha}(x)=x^\alpha$. Writing the two-scale expansion~\eqref{eq:expansion-1-hom} of~$\tau^{\varepsilon,\alpha}$, we thus have, when~$\varepsilon$ is small,
\begin{equation*}
	\tau^{\varepsilon,\alpha}(x)
	\approx
	\tau^{\star,\alpha}(x) + \varepsilon \sum_{\gamma=1}^d w_\gamma \left( \frac{x}{\varepsilon} \right) \partial_\gamma \tau^{\star,\alpha}(x)
	=
	x^\alpha + \varepsilon \, w_\alpha \left( \frac{x}{\varepsilon} \right),
\end{equation*}
and the difference tends to zero in~$H^1(K)$ as $\varepsilon\to0$. Inserting this convergence in~\eqref{eq:diffusion-eff-msfem-pg}, we deduce that
\begin{equation*}
	\lim_{\varepsilon\to0} \left. \overline{A}_{\beta,\alpha}(\varepsilon) \right\vert_K
	=
	\lim_{\varepsilon\to0} \frac{1}{|K|} \int_K
		e_\beta \cdot A^\mathsf{per} \left(\frac{x}{\varepsilon}\right) \left(
			e_ \alpha + \nabla w_\alpha \left(\frac{x}{\varepsilon}\right)
		\right) \dd x
	=
	A^\star_{\beta,\alpha}.
\end{equation*}
The convergence to the mean on the unit cube in the last equality follows from the $Q$-periodicity of the function \mbox{
$
e_\beta \cdot A^\mathsf{per} \, (e_ \alpha + \nabla w_\alpha).
$
}
\end{proof}

The following lemma studies the convergence of~$u^{\varepsilon,\mathsf{G}}_H - u^{\varepsilon,\mathsf{PG}}_H$ towards $0$ as $\varepsilon\to0$ for the MsFEM-lin without oversampling. As was stated in Rem.~\ref{rem:estim-eps-rate}, thanks to the periodic setting, we now obtain a rate for the convergence stated in Lemma~\ref{lem:estim-Gal-PG-eps}.

\begin{lemma}
\label{lem:estim-Gal-PG-per}
Let~$f \in L^2(\Omega)$. Suppose that the family of meshes~$(\mesh)_{H>0}$ is quasi-uniform. Consider the MsFEM-lin without oversampling. For~$A^\varepsilon = A^\mathsf{per}(\cdot/\varepsilon)$ sufficiently regular, we have 
\begin{equation*}
	\left\lVert u^{\varepsilon,\mathsf{G}}_H - u^{\varepsilon,\mathsf{PG}}_H \right\rVert_{H^1(\Omega)} 
	\leq
	C\varepsilon \, \lVert f \rVert_{L^2(\Omega)},
\end{equation*}
where the constant~$C$ depends on the dimension~$d$ and the constants~$m, M$ in~\eqref{ass:bounds}, but not on $\varepsilon$, $H$ or~$f$.
\end{lemma}

\begin{proof}
Let~$e^\varepsilon_H  = u^{\varepsilon,\mathsf{G}}_H - u^{\varepsilon,\mathsf{PG}}_H$. Lemma~\ref{lem:estim-gradient-error} applies, so we can use~\eqref{eq:estim-Gal-PG-Cea} and a Cauchy-Schwarz inequality to find
\begin{equation}
	a^{\varepsilon,\dif}{ \left( e^\varepsilon_H, \, e^\varepsilon_H \right)}
	\leq
	\lVert f \rVert_{L^2(\Omega)} \, \left\lVert e^{\mathsf{osc}}_H \right\rVert_{L^2(\Omega)}
	\leq
	\lVert f \rVert_{L^2(\Omega)} \,
	\left\lVert	
		\sum_{K\in\mesh} \sum_{\alpha=1}^d 
		\left. \left(\partial_\alpha e^\Pone_H\right) \right\vert_K \VK{\alpha}
	\right\rVert_{L^2(\Omega)}.
\label{eq:estim-Gal-PG-per0}
\end{equation}
Next we seek a bound on~$\VK{\alpha}$ in~$L^2(K)$. 
Using~\eqref{eq:diffusion-MsFEM-correctors} and~\eqref{eq:diffusion-correctors}, we have
\begin{equation*}
	\operatorname{div} \left( A^\mathsf{per}\left(\frac{\cdot}{\varepsilon}\right) \nabla
		\left[ \VK{\alpha} - \varepsilon \, w_\alpha \left(\frac{\cdot}{\varepsilon}\right) \right]
	\right)
	=0
	\quad 
	\text{in } K.
\end{equation*}
Since~$\VK{\alpha}$ vanishes on~$\partial K$ (recall that we consider the MsFEM-lin without oversampling), the maximum principle~\cite[Theorem 8.1]{gilbarg_elliptic_2001} yields
\begin{equation*}
	\left\lVert 
		\VK{\alpha} - \varepsilon \, w_\alpha \left(\frac{\cdot}{\varepsilon}\right)
	\right\rVert_{L^2(K)}
	\leq
	\sup_{\partial K} \left\lvert \VK{\alpha} - \varepsilon \, w_\alpha \left(\frac{\cdot}{\varepsilon}\right) \right\rvert \sqrt{\int_K 1}
	=
	\varepsilon \, |K|^{1/2} 
	\sup_{\partial K} 
	\left\lvert 
		w_\alpha \left(\frac{\cdot}{\varepsilon}\right)
	\right\rvert.
\end{equation*}
When~$A^\mathsf{per}$ is sufficiently regular, the corrector functions~$w_\alpha$ are uniformly bounded. Then the mesh regularity provides a constant~$C$ such that for each~$K\in\mesh$ and each~$1 \leq \alpha \leq d$, we have 
\begin{equation*}
	\left\lVert 
		\VK{\alpha}
	\right\rVert_{L^2(K)}
	\leq
	\left\lVert 
		\VK{\alpha} - \varepsilon \, w_\alpha \left(\frac{\cdot}{\varepsilon}\right)
	\right\rVert_{L^2(K)}
	+
	\varepsilon \,  \left\lVert 
		w_\alpha \left(\frac{\cdot}{\varepsilon}\right)
	\right\rVert_{L^2(K)}
	\leq
	C \varepsilon H^{d/2}.
\end{equation*}
Since all~$\VK{\alpha}$ have disjoint supports, we can use the latter estimate to bound
\begin{align}
	\left\lVert	
		\sum_{K\in\mesh} \sum_{\alpha=1}^d 
		\left. \left(\partial_\alpha e^\Pone_H\right) \right\vert_K\VK{\alpha}
	\right\rVert_{L^2(\Omega)}^2
	&=
	\sum_{K\in\mesh} \left \lVert
		\sum_{\alpha=1}^d \left. \left( \partial_\alpha e^\Pone_H \right) \right\vert_K\VK{\alpha}
	\right\rVert_{L^2(K)}^2
	\nonumber\\
	&\leq 
	C \varepsilon^2 
	\sum_{K\in\mesh} 
		\sum_{\alpha=1}^d \left( H^{d/2} \left. \left(\partial_\alpha e^\Pone_H\right) \right\vert_K \right)^2
	\nonumber\\
	&\leq
	C \varepsilon^2
	\sum_{K\in\mesh} 
		\sum_{\alpha=1}^d \left\lVert \partial_\alpha e^\Pone_H \right\rVert_{L^2(K)}^2
	\nonumber\\
	&=
	C\varepsilon^2 \, \left\lVert \nabla e^\Pone_H \right\rVert_{L^2(\Omega)}^2.
	\label{eq:estim-Gal-PG-per1}
\end{align}
The last inequality relies on the quasi-uniformity of the mesh.

We insert~\eqref{eq:estim-Gal-PG-per1} combined with~\eqref{eq:estim-gradient-error1} into~\eqref{eq:estim-Gal-PG-per0} to find
\begin{equation*}
	a^{\varepsilon,\dif}{ \left( e^\varepsilon_H, \, e^\varepsilon_H \right)}
	\leq
	C\varepsilon \,
	\lVert f \rVert_{L^2(\Omega)} \,
	\left\lVert \nabla e^\varepsilon_H \right\rVert_{L^2(\Omega)}.
\end{equation*}
Applying the coercivity property in~\eqref{ass:bounds} on the left-hand side and a Poincaré inequality on~$\Omega$, we obtain the desired result.
\end{proof}

The classical error estimate for the Galerkin MsFEM approach~\eqref{eq:diffusion-MsFEM} is obtained in the periodic setting and under some regularity assumption on~$A^\mathsf{per}$ and on the homogenized limit~$u^\star$. The bound obtained in~\cite[Theorem 6.5]{efendiev_multiscale_2009} reads
\begin{equation}
	\left\lVert u^\varepsilon - u^{\varepsilon,\mathsf{G}}_H \right\rVert_{H^1(\Omega)}
	\leq 	
	C\left(H+\varepsilon+\sqrt{\varepsilon/H}\right),
	\label{eq:msfem-error-estimate}
\end{equation}
for some~$C$ independent of~$\varepsilon$ and~$H$.
Lemma~\ref{lem:estim-Gal-PG-per} shows that the same estimate holds true for~$u^{\varepsilon,\mathsf{PG}}_H$, the Petrov-Galerkin MsFEM approximation, under the correct regularity assumptions. We note that the bound for~$u^{\varepsilon,\mathsf{PG}}_H$ can also be inferred from Lemma~\ref{lem:estim-Gal-PG-H}. However, since the MsFEM is applied in the regime where~$\varepsilon < H$, the result of Lemma~\ref{lem:estim-Gal-PG-per} is more precise, thanks to the extra structural assumptions made on the diffusion tensor~$A^\varepsilon$.
\section{Numerical comparison}
  \label{sec:gen-num}

%%
%%%%%%%%%%%%%%%%%%%%%%%%%%%%%%%%%%
%% Numerical experiments for the comparison of intrusive and non-intrusive MsFEMs for diffusion problems (periodic and non-periodic, MsFEM-lin(-OS), MsFEM-CR(-OS), always DOF-continuous oversampling)
%%%%%%%%%%%%%%%%%%%%%%%%%%%%%%%%%%
%%

We now compare the Galerkin MsFEM~\eqref{eq:gen-MsFEM-OS}, its non-intrusive approximation~\eqref{eq:gen-MsFEM-noni} and the Petrov-Galerkin MsFEM~\eqref{eq:gen-MsFEM-OS-testP1} on a concrete numerical example in 2D ($d=2$). The numerical approximations obtained for these various MsFEMs shall be denoted $u_H^{\varepsilon,\mathsf{G}}$, $u_H^{\varepsilon,\text{\sf G-ni}}$ and $u_H^{\varepsilon,\mathsf{PG}}$, respectively. 

\subsection{Description of the numerical experiments}
\label{sec:gen-num-setting}
We consider the pure diffusion equation~\eqref{eq:diffusion-pde} on the domain $\Omega = (0,1) \times (0,1)$. Thus, the local bilinear forms are $a_K^\varepsilon = a_K^{\varepsilon,\dif}$ defined in Example~\ref{ex:diffusion-vf-gen}, where we consider the three diffusion tensors
\begin{subequations}
	\label{eq:diffusion-test}
	\begin{align}
		A^{\varepsilon,\mathsf{per}}(x) &= \, \nu^\varepsilon(x) \operatorname{Id}, 
		\quad 
		\nu^\varepsilon(x) = 
		1 + 100 \, \cos^2{(\pi \, x_1 / \varepsilon)} \sin^2{(\pi \, x_2 / \varepsilon)},
		\label{eq:diffusion-test-per} \\
		A^{\varepsilon,\mathsf{lp}}(x) &= \left( 1+\cos^2{(2\pi x_1)} \right) \, A^\varepsilon(x),
		\label{eq:diffusion-test-locper}
		\\
		A^{\varepsilon,\mathsf{np}}(x) &= 1 + 
		\left( 
			1 + 100 \, \cos^2{(\pi \, x_1 / \varepsilon)} \sin^2{(\pi \, x_2 / \varepsilon)}
		\right)
		\cos^2{\left( 
			\frac{x_1^2 + x_2^2}{\varepsilon}
		\right)}.
		\label{eq:diffusion-test-nonper}
	\end{align}
\end{subequations}
We fix $f(x) = \sin{(x_1)}\sin{(x_2)}$. 

The coefficient~$A^{\varepsilon,\mathsf{per}}$ is $\varepsilon$-periodic with period~$\varepsilon=\pi/150 \approx 0.02$. The coefficient~$A^{\varepsilon,\mathsf{lp}}$ is locally periodic and, although a homogenized coefficient exists (see~\cite{bensoussan_asymptotic_1978}), it is not constant. Consequently, a certain number of lemmas established in Sec.~\ref{sec:compare-gal-pg} are not known to hold true. Finally, we include the coefficient~$A^{\varepsilon,\mathsf{np}}$ as an example of a multiscale problem for which we are not aware of any explicit homogenization results.
We will see nevertheless that the non-intrusive MsFEMs that we introduced above provide good approximations compared to their intrusive G-MsFEM counterparts for all test cases.

A reference solution~$u_h^\varepsilon$ is computed on a uniform $1024\times1024$ mesh~$\fmesh$ by means of a standard~$\Pone$ finite element method using \textsc{FreeFEM++}~\cite{hecht_new_2012}. The mesh~$\fmesh$ (as well as the coarse mesh introduced below) consists of squares cut in two along a diagonal that is in the same direction for all squares, i.e., such as the meshes in Fig.~\ref{fig:os}. The \textsc{FreeFEM++} scripts to perform all different MsFEMs can be found at~\cite{rutger_biezemans_2023_7525059}.

We compare the reference solution~$u_h^\varepsilon$ to MsFEM solutions obtained on a coarse mesh~$\mesh$ for varying~$H$. The mesh~$\mesh$ is a uniform $1/H \times 1/H$ triangulation of~$\Omega$. We test the MsFEM-lin and the MsFEM-CR using the sampling operator $s_K^\varepsilon = a^{\varepsilon,\dif}_K$.  All oversampling methods in this section use a homothety ratio of 3 for the construction of the oversampling patches in Def.~\ref{def:ospatch}. 
%, and \DOF-continuous basis functions, which ensure certain continuity properties on the boundary of the mesh elements. 
A precise definition of the associated basis functions can be found in Examples~\ref{ex:msfem-lin-basis} and~\ref{ex:msfem-cr-basis}. The mesh~$\fmesh$ is a refinement of~$\mesh$ for all values of~$H$. Therefore, for each $K \in \mesh$, we use the corresponding submesh of~$\fmesh$ (consisting of all triangles included in~$K$) for the numerical approximation of the numerical correctors in~\eqref{eq:MsFEM-gen-correctors-dofe} by $\Pone$ Lagrange finite elements.

\begin{remark}
	\label{rem:error-computation}
	We provide a few remarks on the computation of the error, which takes place in the post-processing step of the MsFEM. Evidently, these computations have to be carried out by integration on the fine scale and one may try to perform these computations on the global mesh~$\fmesh$. However, the legacy code does not, in general, operate on the global fine mesh. 
	Moreover, we stress that the approximation~$u^\varepsilon_H$ is in general discontinuous across element edges (for the MsFEM-CR, and for all MsFEMs with oversampling), and can therefore not be represented globally by e.g.~a piecewise $\Pone$ function on the fine mesh~$\fmesh$ (even if one supposes that~$\fmesh$ is conformal).
	Thus, one has to compute the error element by element, using the code for the microscale, according to the sum 
	\begin{equation*}
		\left\lVert u^\varepsilon - u^\varepsilon_H \right\rVert_{H^1(\mesh)}^2
		=
		\sum_{K\in\mesh} \left\lVert u^\varepsilon - u^\varepsilon_H \right\rVert_{H^1(K)}^2.
	\end{equation*}
	To do so, Equation~\eqref{eq:gen-MsFEM-post} can be used on each element~$K$ to find the correct values of~$u^\varepsilon_H$, and the global fine mesh~$\mathcal{T}_h$ is never used.
\end{remark}

\subsection{Results}
\label{sec:gen-num-results}

We first compare the approximations~$u_H^{\varepsilon,\mathsf{G}}$ and~$u_H^{\varepsilon,\text{\sf G-ni}}$ for varying~$H$ in Fig.~\ref{fig:Gal-vs-Gni} for MsFEMs without oversampling and MsFEMs with \DOF-continuous oversampling. Without oversampling (OS), the approximation~$u_H^{\varepsilon,\text{\sf G-ni}}$ equals~$u_H^{\varepsilon,\mathsf{PG}}$ due to Lemma~\ref{lem:IBP-bubbles-gen}. We also report the error committed by the G-MsFEM. We observe that, without oversampling, the difference $u_H^{\varepsilon,\mathsf{G}} - u_H^{\varepsilon,\text{\sf G-ni}}$ is much smaller than this error. As a result, the errors obtained with the G-MsFEM and its non-intrusive approximation are of the same size. Indeed, the error of the non-intrusive G-MsFEM-lin deviates from the error of the G-MsFEM-lin by at most 0.05\% for all tests that we report here. For the MsFEM-CR, this is at most 1.2\%. In both cases, the two MsFEM variants thus have practically the same accuracy. This is in agreement with the theoretical result of Lemma~\ref{lem:estim-Gal-PG-H}.

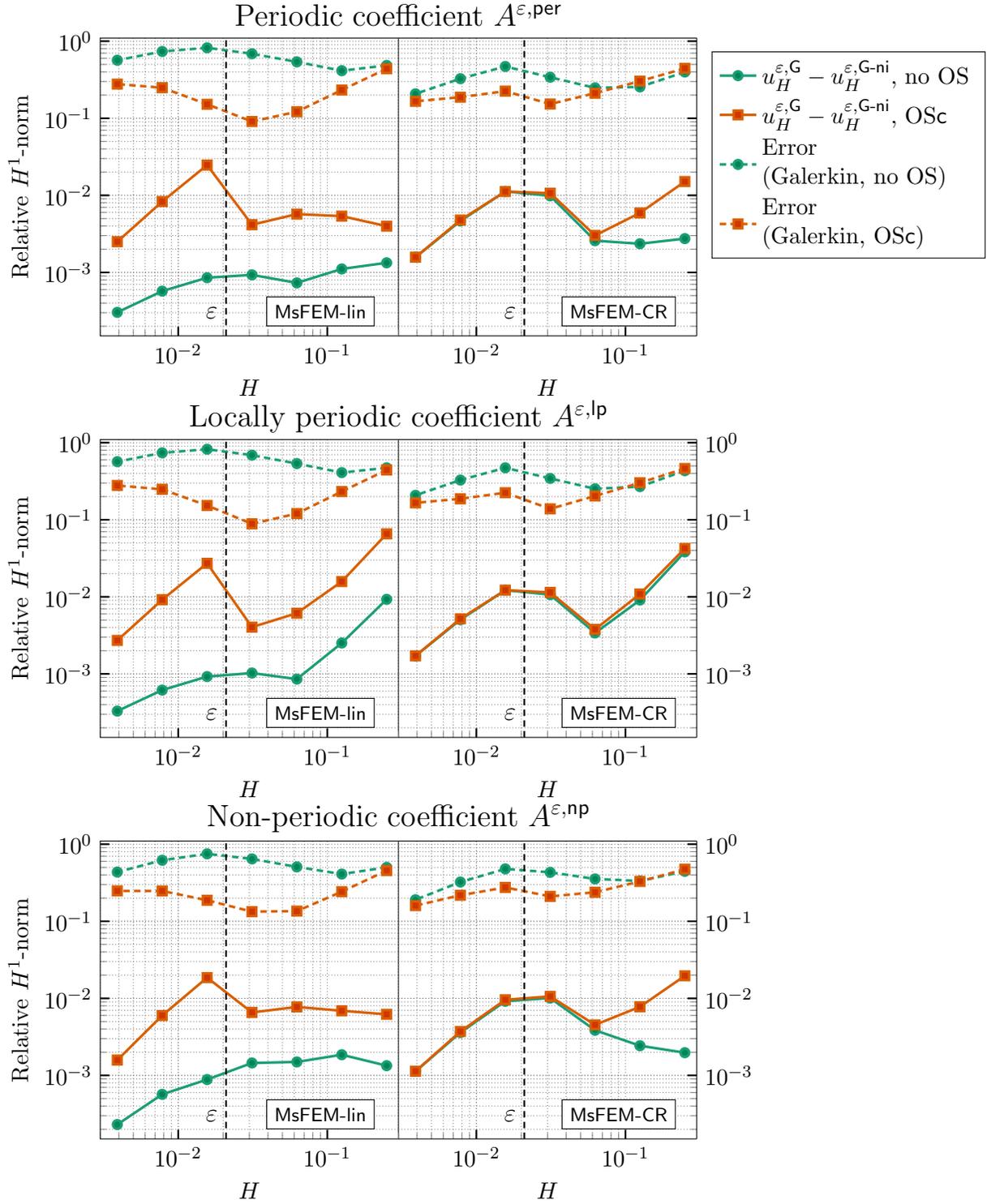
\begin{figure}
	\centering
    %%
%%%%%%%%%%%%%%%%%%%%%%%%%%%%%%%%%%
%% Results of numerical experiments to compare the difference between Galerkin and non-intrusive Galerkin MsFEM approximations to the error of the Galerkin approximation (periodic and non-periodic, MsFEM-lin(-OS), MsFEM-CR(-OS), always DOF-continuous oversampling)
%%%%%%%%%%%%%%%%%%%%%%%%%%%%%%%%%%
%%

\begin{tikzpicture}

\tikzmath{
    \Hmax = 0.3; % used as maximum value on x-axis
    \Hmin = 0.003; % used as minimum value on x-axis
    \emax = 1.1; % used as maximum value on y-axis
    \emin = 0.00015; % used as minimum value on y-axis
    \eps = 0.020947; % vertical line at the value of epsilon in the numerical experiment
    \epsheight = 0.0002; % y-value for the symbol epsilon next to the vertical line
    \Htype = 0.09; % horizontal position for MsFEM type
	\refnorm = 0.0116239; % norms of the reference solution for the rescaling of \MsFEMcompareVar and \MsFEMcompare and \MsFEMcompareNP
	\refnormVar = 0.007729; 
	\refnormNP = 0.0173429;  
}

  \begin{groupplot}[
    group style={
      group name=my plots,
      group size=2 by 3,
      ylabels at=edge left,
      xlabels at=edge bottom,
      horizontal sep=0cm,
	  vertical sep=1.7cm,
    },width= 0.39\textwidth,
     height = 0.39\textwidth,
     every axis plot/.append style={very thick},
	 legend style={
		cells={align=left},
		nodes={text width=.21\textwidth,text depth=}
		},
	 major tick style={thick, black} 
    ]

	%%
	%% Periodic test case 
	%% 

	\nextgroupplot[ % Plot layout
	    %% Axis settings
	    xmax = \Hmax, xmin = \Hmin,
	    ymin = \emin, ymax = \emax,
	    xmode = log,
	    ymode = log,
	    xlabel = {$H$},
	    ylabel = Relative $H^1$-norm,
    ]
    
        %% Load data from macros defined in load_data.sty
	    \addplot[Dark2-A, mark options={solid, fill=Dark2-A!125}, mark = *] table [x = {H}, y expr = \thisrow{Lin-Gal-Galni}/\refnorm] {\MsFEMcompare};
		\addplot[Dark2-B, mark options={solid, fill=Dark2-B!125}, mark = square*] table [x = {H}, y expr = \thisrow{LinOSc3-Gal-Galni}/\refnorm] {\MsFEMcompare};
		\addplot[Dark2-A, densely dashed, mark options={solid, fill=Dark2-A!125}, mark = *] table [x = {H}, y = {Lin-Gal}] {\MsFEMtestsRel};
		\addplot[Dark2-B, densely dashed, mark options={solid, fill=Dark2-B!125}, mark = square*] table [x = {H}, y = {LinOSc3-Gal}] {\MsFEMtestsRel};
	    %% Vertical line at x=epsilon
	    \draw [thick, black, densely dashed] ({axis cs:\eps,0}|-{rel axis cs:0,0}) -- ({axis cs:\eps,0}|-{rel axis cs:0,1});
	    \node[above left,black] at (axis cs: \eps,\epsheight) {\large $\varepsilon$};
	    \node[above,black,fill=white,draw=black] at (axis cs: \Htype,\epsheight) {\sf \footnotesize MsFEM-lin};

    \nextgroupplot[% Plot layout
	    %% Axis settings
	    xmax = \Hmax, xmin = \Hmin,
	    ymin = \emin, ymax = \emax,
	    xmode = log,
	    ymode = log,
	    xlabel = {$H$},
	    axis y line*=right, %line* to avoid an arrow on top
		yticklabels={,,},
	    legend style={at={(1.05,0.6075)},anchor=west}
	    ]
	    
	    %% Load data from macros defined in load_data.sty
	    \addplot[Dark2-A, mark options={solid, fill=Dark2-A!125}, mark = *] table [x = {H}, y expr = \thisrow{CR-Gal-Galni}/\refnorm] {\MsFEMcompare};
		\addplot[Dark2-B, mark options={solid, fill=Dark2-B!125}, mark = square*] table [x = {H}, y expr = \thisrow{CROSc3-Gal-Galni}/\refnorm] {\MsFEMcompare};
		\addplot[Dark2-A, densely dashed, mark options={solid, fill=Dark2-A!125}, mark = *] table [x = {H}, y = {CR-Gal}] {\MsFEMtestsRel};
		\addplot[Dark2-B, densely dashed, mark options={solid, fill=Dark2-B!125}, mark = square*] table [x = {H}, y = {CROSc3-Gal}] {\MsFEMtestsRel};
	    %% Vertical line at x=epsilon
	    \draw [thick, black, densely dashed] ({axis cs:\eps,0}|-{rel axis cs:0,0}) -- ({axis cs:\eps,0}|-{rel axis cs:0,1});
	    \node[above left,black] at (axis cs: \eps,\epsheight) {\large $\varepsilon$};
	    \node[above,black,fill=white,draw=black] at (axis cs: \Htype,\epsheight) {\sf \footnotesize MsFEM-CR};
	    
	    %% Specify labels for legend
	    \legend{
	        {$u^{\varepsilon,\mathsf{G}}_H - u^{\varepsilon,\mathsf{G\text{-}ni}}_H$, no OS},
			{$u^{\varepsilon,\mathsf{G}}_H - u^{\varepsilon,\mathsf{G\text{-}ni}}_H$,  OS\textsf{c}},
			{Error\\ (Galerkin, no OS)},
			{Error\\ (Galerkin, OS\textsf{c})}
	    }

	%%
	%% Locally periodic test case
	%%

    \nextgroupplot[ % Plot layout
	    %% Axis settings
	    xmax = \Hmax, xmin = \Hmin,
	    ymin = \emin, ymax = \emax,
	    xmode = log,
	    ymode = log,
	    xlabel = {$H$},
	    ylabel = Relative $H^1$-norm,
    ]
    
        %% Load data from macros defined in load_data.sty
	    \addplot[Dark2-A, mark options={solid, fill=Dark2-A!125}, mark = *] table [x = {H}, y expr = \thisrow{Lin-Gal-Galni}/\refnormVar] {\MsFEMcompareVar};
		\addplot[Dark2-B, mark options={solid, fill=Dark2-B!125}, mark = square*] table [x = {H}, y expr = \thisrow{LinOSc3-Gal-Galni}/\refnormVar] {\MsFEMcompareVar};
		\addplot[Dark2-A, densely dashed, mark options={solid, fill=Dark2-A!125}, mark = *] table [x = {H}, y = {Lin-Gal}] {\MsFEMtestsVarRel};
		\addplot[Dark2-B, densely dashed, mark options={solid, fill=Dark2-B!125}, mark = square*] table [x = {H}, y = {LinOSc3-Gal}] {\MsFEMtestsVarRel};
	    %% Vertical line at x=epsilon
	    \draw [thick, black, densely dashed] ({axis cs:\eps,0}|-{rel axis cs:0,0}) -- ({axis cs:\eps,0}|-{rel axis cs:0,1});
	    \node[above left,black] at (axis cs: \eps,\epsheight) {\large $\varepsilon$};
	    \node[above,black,fill=white,draw=black] at (axis cs: \Htype,\epsheight) {\sf \footnotesize MsFEM-lin};

    \nextgroupplot[% Plot layout
	    %% Axis settings
	    xmax = \Hmax, xmin = \Hmin,
	    ymin = \emin, ymax = \emax,
	    xmode = log,
	    ymode = log,
	    xlabel = {$H$},
	    axis y line*=right, %line* to avoid an arrow on top
	    ]
	    
	    %% Load data from macros defined in load_data.sty
	    \addplot[Dark2-A, mark options={solid, fill=Dark2-A!125}, mark = *] table [x = {H}, y expr = \thisrow{CR-Gal-Galni}/\refnormVar] {\MsFEMcompareVar};
		\addplot[Dark2-B, mark options={solid, fill=Dark2-B!125}, mark = square*] table [x = {H}, y expr = \thisrow{CROSc3-Gal-Galni}/\refnormVar] {\MsFEMcompareVar};
		\addplot[Dark2-A, densely dashed, mark options={solid, fill=Dark2-A!125}, mark = *] table [x = {H}, y = {CR-Gal}] {\MsFEMtestsVarRel};
		\addplot[Dark2-B, densely dashed, mark options={solid, fill=Dark2-B!125}, mark = square*] table [x = {H}, y = {CROSc3-Gal}] {\MsFEMtestsVarRel};
	    %% Vertical line at x=epsilon
	    \draw [thick, black, densely dashed] ({axis cs:\eps,0}|-{rel axis cs:0,0}) -- ({axis cs:\eps,0}|-{rel axis cs:0,1});
	    \node[above left,black] at (axis cs: \eps,\epsheight) {\large $\varepsilon$};
	    \node[above,black,fill=white,draw=black] at (axis cs: \Htype,\epsheight) {\sf \footnotesize MsFEM-CR};
	
	%%
	%% Non-periodic test case
	%%

    \nextgroupplot[ % Plot layout
	    %% Axis settings
	    xmax = \Hmax, xmin = \Hmin,
	    ymin = \emin, ymax = \emax,
	    xmode = log,
	    ymode = log,
	    xlabel = {$H$},
	    ylabel = Relative $H^1$-norm,
    ]
    
        %% Load data from macros defined in load_data.sty
	    \addplot[Dark2-A, mark options={solid, fill=Dark2-A!125}, mark = *] table [x = {H}, y expr = \thisrow{Lin-Gal-Galni}/\refnormNP] {\MsFEMcompareNP};
		\addplot[Dark2-B, mark options={solid, fill=Dark2-B!125}, mark = square*] table [x = {H}, y expr = \thisrow{LinOSc3-Gal-Galni}/\refnormNP] {\MsFEMcompareNP};
		\addplot[Dark2-A, densely dashed, mark options={solid, fill=Dark2-A!125}, mark = *] table [x = {H}, y = {Lin-Gal}] {\MsFEMtestsNPRel};
		\addplot[Dark2-B, densely dashed, mark options={solid, fill=Dark2-B!125}, mark = square*] table [x = {H}, y = {LinOSc3-Gal}] {\MsFEMtestsNPRel};
	    %% Vertical line at x=epsilon
	    \draw [thick, black, densely dashed] ({axis cs:\eps,0}|-{rel axis cs:0,0}) -- ({axis cs:\eps,0}|-{rel axis cs:0,1});
	    \node[above left,black] at (axis cs: \eps,\epsheight) {\large $\varepsilon$};
	    \node[above,black,fill=white,draw=black] at (axis cs: \Htype,\epsheight) {\sf \footnotesize MsFEM-lin};

    \nextgroupplot[% Plot layout
	    %% Axis settings
	    xmax = \Hmax, xmin = \Hmin,
	    ymin = \emin, ymax = \emax,
	    xmode = log,
	    ymode = log,
	    xlabel = {$H$},
	    axis y line*=right, %line* to avoid an arrow on top
	    ]
	    
	    %% Load data from macros defined in load_data.sty
	    \addplot[Dark2-A, mark options={solid, fill=Dark2-A!125}, mark = *] table [x = {H}, y expr = \thisrow{CR-Gal-Galni}/\refnormNP] {\MsFEMcompareNP};
		\addplot[Dark2-B, mark options={solid, fill=Dark2-B!125}, mark = square*] table [x = {H}, y expr = \thisrow{CROSc3-Gal-Galni}/\refnormNP] {\MsFEMcompareNP};
		\addplot[Dark2-A, densely dashed, mark options={solid, fill=Dark2-A!125}, mark = *] table [x = {H}, y = {CR-Gal}] {\MsFEMtestsNPRel};
		\addplot[Dark2-B, densely dashed, mark options={solid, fill=Dark2-B!125}, mark = square*] table [x = {H}, y = {CROSc3-Gal}] {\MsFEMtestsNPRel};
	    %% Vertical line at x=epsilon
	    \draw [thick, black, densely dashed] ({axis cs:\eps,0}|-{rel axis cs:0,0}) -- ({axis cs:\eps,0}|-{rel axis cs:0,1});
	    \node[above left,black] at (axis cs: \eps,\epsheight) {\large $\varepsilon$};
	    \node[above,black,fill=white,draw=black] at (axis cs: \Htype,\epsheight) {\sf \footnotesize MsFEM-CR};
	    
    \end{groupplot} 

	\node (per) at ($(my plots c1r1.center)!0.5!(my plots c2r1.center)+(0,.17\textwidth)$) {\Large Periodic coefficient~$A^{\varepsilon,\mathsf{per}}$};
	\node (nonper) at ($(my plots c1r2.center)!0.5!(my plots c2r2.center)+(0,.17\textwidth)$) {\Large Locally periodic coefficient~$A^{\varepsilon,\mathsf{lp}}$};
	\node (nonper) at ($(my plots c1r3.center)!0.5!(my plots c2r3.center)+(0,.17\textwidth)$) {\Large Non-periodic coefficient~$A^{\varepsilon,\mathsf{np}}$};

\end{tikzpicture}
    \caption{
		Solid lines: difference between the Galerkin MsFEM approximation ($u^{\varepsilon,\mathsf{G}}_H$ defined by~\eqref{eq:gen-MsFEM-OS}) and the non-intrusive Galerkin MsFEM approximation ($u^{\varepsilon,\mathsf{G\text{-}ni}}_H$ defined by~\eqref{eq:gen-MsFEM-noni}), without oversampling (no OS) and with \DOF-continuous oversampling (OS\textsf{c}), for the diffusion coefficients in~\eqref{eq:diffusion-test} as the mesh size~$H$ varies.
		Dashed lines: error of the Galerkin MsFEM with respect to the reference solution. All values are normalized with respect to the $H^1$ norm of the reference solution.
		}
    \label{fig:Gal-vs-Gni}
\end{figure} 

The estimates obtained in Sec.~\ref{sec:compare-gal-pg} do not apply to MsFEMs with oversampling. From Fig.~\ref{fig:Gal-vs-Gni}, we can see that the difference $u_H^{\varepsilon,\mathsf{G}} - u_H^{\varepsilon,\text{\sf G-ni}}$ is still small with respect to the error committed by the G-MsFEM when \DOF-continuous oversampling is applied. The approximation errors for the non-intrusive G-MsFEMs with \DOF-continuous oversampling differ by at most 1.3\% from the error of the G-MsFEM. Similar conclusions hold for the MsFEM-lin with \DOF-extended oversampling. The difference between the G-MsFEM and the non-intrusive G-MsFEM is larger for the MsFEM-CR with \DOF-extended oversampling. We do not include these results in the comparison of Fig.~\ref{fig:Gal-vs-Gni} because both methods perform particularly badly when compared to the G-MsFEM without oversampling.  

Let us also point out the qualitative and quantitative similarities between the performance of the MsFEM for the periodic and the non-periodic diffusion coefficients. Although the study of the homogenized limit of~$u^\varepsilon$ becomes increasingly difficult for the various coefficients~\eqref{eq:diffusion-test-per} to~\eqref{eq:diffusion-test-nonper}, the non-intrusive approximation does not deteriorate the accuracy of the MsFEM in these numerical tests.

Before moving on to a comparison with the Petrov-Galerkin MsFEMs with oversampling, let us discuss a phenomenon in Fig.~\ref{fig:Gal-vs-Gni} and~\ref{fig:MsFEM-tests} known as the `resonance effect' in the literature, preventing convergence of the MsFEM if the coarse scale~$H$ is close to~$\varepsilon$. Upon further decreasing~$H$, convergence is found only when~$H$ is sufficiently small with respect to the microscale~$\varepsilon$, in which case we are in the regime of classical FEMs. From a theoretical point of view, this is explained by the term~$\sqrt{\varepsilon/H}$ in the error estimate~\eqref{eq:msfem-error-estimate} (or~$\varepsilon/H$ for the MsFEM-lin with oversampling; see~\cite{efendiev_convergence_2000}). We note that the same error estimate was obtained in~\cite{le_bris_msfem_2013} for the MsFEM-CR (without oversampling). Figure~\ref{fig:MsFEM-tests} shows that the resonance effect is more pronounced for the MsFEM-lin with oversampling than for the MsFEM-CR with oversampling. 

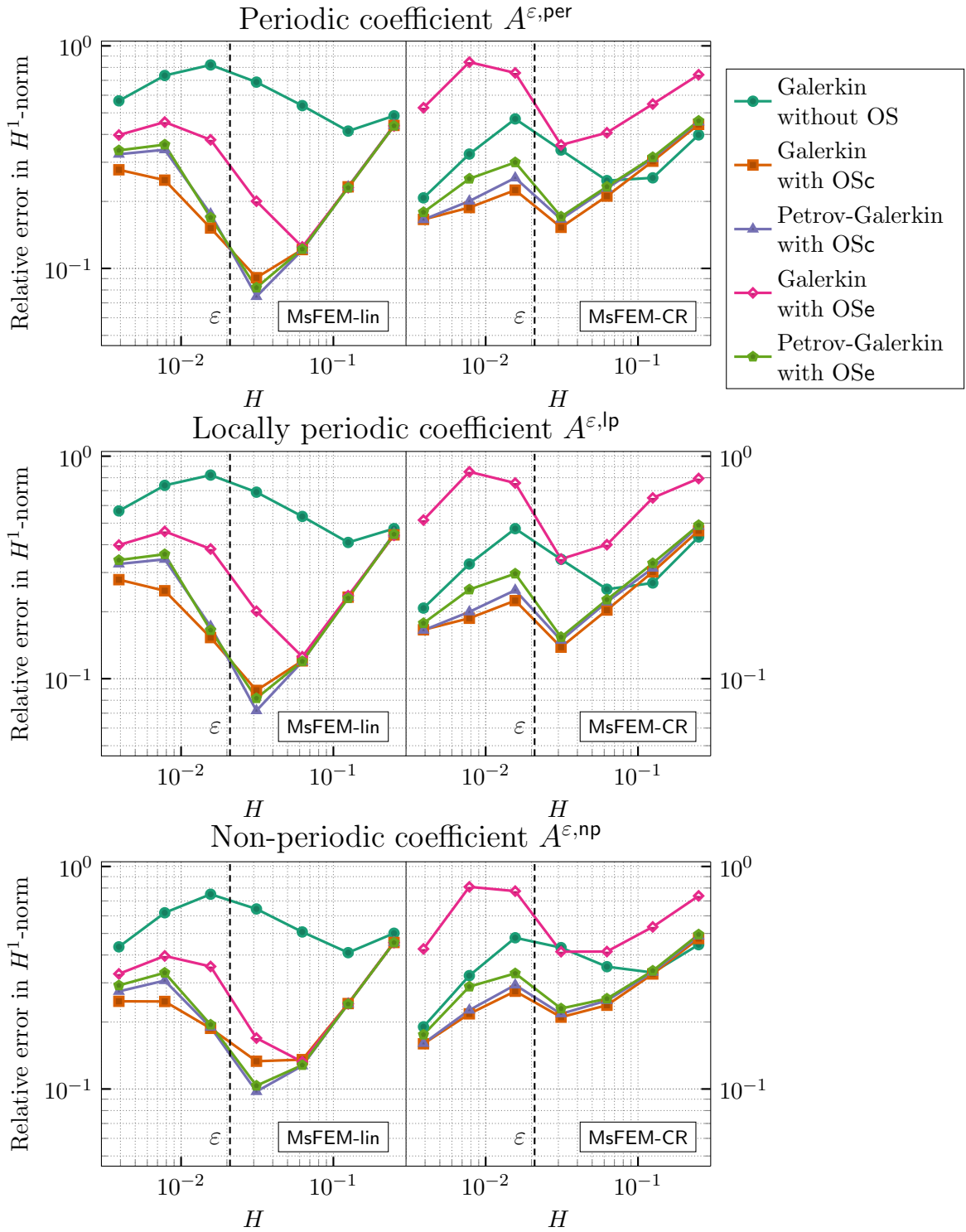
\begin{figure}
	\centering
	%%
%%%%%%%%%%%%%%%%%%%%%%%%%%%%%%%%%%
%% Results of numerical experiments to compare the difference between Petrov-Galerkin and Galerkin MsFEM approximations (diffusion problems, periodic and non-periodic, MsFEM-lin(-OS), MsFEM-CR(-OS), always DOF-continuous oversampling)
%%%%%%%%%%%%%%%%%%%%%%%%%%%%%%%%%%
%%

\begin{tikzpicture}

\tikzmath{
    \Hmax = 0.3; % used as maximum value on x-axis
    \Hmin = 0.003; % used as minimum value on x-axis
    \emax = 1.05; % used as maximum value on y-axis
    \emin = 0.045; % used as minimum value on y-axis
    \eps = 0.020947; % vertical line at the value of epsilon in the numerical experiment
    \epsheight = 0.052; % y-value for the symbol epsilon next to the vertical line
    \Htype = 0.1; % horizontal position for MsFEM type
}

  \begin{groupplot}[
    group style={
      group name=my plots,
      group size=2 by 3,
      ylabels at=edge left,
      xlabels at=edge bottom,
      horizontal sep=0cm,
	  vertical sep=1.7cm,
    },width= 0.39\textwidth,
     height = 0.39\textwidth,
     every axis plot/.append style={very thick},
	 legend style={
		cells={align=left},
		nodes={text width=.17\textwidth,text depth=}
		}, 
	 major tick style={thick, black} 
    ]

	%%
	%% Periodic test case
	%%

	\nextgroupplot[ % Plot layout
	    %% Axis settings
	    xmax = \Hmax, xmin = \Hmin,
	    ymin = \emin, ymax = \emax,
	    xmode = log,
	    ymode = log,
	    xlabel = {$H$},
	    ylabel = Relative error in $H^1$-norm,
    ]
    
        %% Load data from macros defined in load_data.sty
	    \addplot table [x = {H}, y = {Lin-Gal}] {\MsFEMtestsRel};
	    \addplot table [x = {H}, y = {LinOSc3-Gal}] {\MsFEMtestsRel};
	    \addplot table [x = {H}, y = {LinOSc3-PG}] {\MsFEMtestsRel};
		\addplot table [x = {H}, y = {LinOSr3-Gal}] {\MsFEMtestsRel};
		\addplot table [x = {H}, y = {LinOSr3-PG}] {\MsFEMtestsRel};
	    %% Vertical line at x=epsilon
	    \draw [thick, black, densely dashed] ({axis cs:\eps,0}|-{rel axis cs:0,0}) -- ({axis cs:\eps,0}|-{rel axis cs:0,1});
	    \node[above left,black] at (axis cs: \eps,\epsheight) {\large $\varepsilon$};
	    \node[above,black,fill=white,draw=black] at (axis cs: \Htype,\epsheight) {\sf \footnotesize MsFEM-lin};

    \nextgroupplot[% Plot layout
	    %% Axis settings
	    xmax = \Hmax, xmin = \Hmin,
	    ymin = \emin, ymax = \emax,
	    xmode = log,
	    ymode = log,
	    xlabel = {$H$},
	    axis y line*=right, %line* to avoid an arrow on top
	    yticklabels={,,},
		legend style={
			at={(1.05,0.3835)},
			anchor=west,
			row sep=4pt
		}
	    ]
	    
	    %% Load data from macros defined in load_data.sty
	    \addplot table [x = {H}, y = {CR-Gal}] {\MsFEMtestsRel};
	    \addplot table [x = {H}, y = {CROSc3-Gal}] {\MsFEMtestsRel};
	    \addplot table [x = {H}, y = {CROSc3-PG}] {\MsFEMtestsRel};
		\addplot table [x = {H}, y = {CROSr3-Gal}] {\MsFEMtestsRel};
		\addplot table [x = {H}, y = {CROSr3-PG}] {\MsFEMtestsRel};
	    %% Vertical line at x=epsilon
	    \draw [thick, black, densely dashed] ({axis cs:\eps,0}|-{rel axis cs:0,0}) -- ({axis cs:\eps,0}|-{rel axis cs:0,1});
	    \node[above left,black] at (axis cs: \eps,\epsheight) {\large $\varepsilon$};
	    \node[above,black,fill=white,draw=black] at (axis cs: \Htype,\epsheight) {\sf \footnotesize MsFEM-CR};
	    
	    % Specify labels for legend
	    \legend{
	        Galerkin\\ without OS,
	        Galerkin\\ with OS\textsf{c},
	        Petrov-Galerkin\\ with OS\textsf{c},
			Galerkin\\ with OS\textsf{e},
			Petrov-Galerkin\\ with OS\textsf{e}
	    }

    %%
	%% Locally periodic test case
	%%	
	
	\nextgroupplot[ % Plot layout
	    %% Axis settings
	    xmax = \Hmax, xmin = \Hmin,
	    ymin = \emin, ymax = \emax,
	    xmode = log,
	    ymode = log,
	    xlabel = {$H$},
	    ylabel = Relative error in $H^1$-norm,
    ]
    
        %% Load data from macros defined in load_data.sty
	    \addplot table [x = {H}, y = {Lin-Gal}] {\MsFEMtestsVarRel};
	    \addplot table [x = {H}, y = {LinOSc3-Gal}] {\MsFEMtestsVarRel};
	    \addplot table [x = {H}, y = {LinOSc3-PG}] {\MsFEMtestsVarRel};
		\addplot table [x = {H}, y = {LinOSr3-Gal}] {\MsFEMtestsVarRel};
		\addplot table [x = {H}, y = {LinOSr3-PG}] {\MsFEMtestsVarRel};
	    %% Vertical line at x=epsilon
	    \draw [thick, black, densely dashed] ({axis cs:\eps,0}|-{rel axis cs:0,0}) -- ({axis cs:\eps,0}|-{rel axis cs:0,1});
	    \node[above left,black] at (axis cs: \eps,\epsheight) {\large $\varepsilon$};
	    \node[above,black,fill=white,draw=black] at (axis cs: \Htype,\epsheight) {\sf \footnotesize MsFEM-lin};

    \nextgroupplot[% Plot layout
	    %% Axis settings
	    xmax = \Hmax, xmin = \Hmin,
	    ymin = \emin, ymax = \emax,
	    xmode = log,
	    ymode = log,
	    xlabel = {$H$},
	    axis y line*=right, %line* to avoid an arrow on top
	    legend style={at={(1.05,1.3)},anchor=west}
	    ]
	    
	    %% Load data from macros defined in load_data.sty
	    \addplot table [x = {H}, y = {CR-Gal}] {\MsFEMtestsVarRel};
	    \addplot table [x = {H}, y = {CROSc3-Gal}] {\MsFEMtestsVarRel};
	    \addplot table [x = {H}, y = {CROSc3-PG}] {\MsFEMtestsVarRel};
		\addplot table [x = {H}, y = {CROSr3-Gal}] {\MsFEMtestsVarRel};
		\addplot table [x = {H}, y = {CROSr3-PG}] {\MsFEMtestsVarRel};
	    %% Vertical line at x=epsilon
	    \draw [thick, black, densely dashed] ({axis cs:\eps,0}|-{rel axis cs:0,0}) -- ({axis cs:\eps,0}|-{rel axis cs:0,1});
	    \node[above left,black] at (axis cs: \eps,\epsheight) {\large $\varepsilon$};
	    \node[above,black,fill=white,draw=black] at (axis cs: \Htype,\epsheight) {\sf \footnotesize MsFEM-CR};

	%%
	%% Non-periodic test case
	%%	
	
	\nextgroupplot[ % Plot layout
	    %% Axis settings
	    xmax = \Hmax, xmin = \Hmin,
	    ymin = \emin, ymax = \emax,
	    xmode = log,
	    ymode = log,
	    xlabel = {$H$},
	    ylabel = Relative error in $H^1$-norm,
    ]
    
        %% Load data from macros defined in load_data.sty
	    \addplot table [x = {H}, y = {Lin-Gal}] {\MsFEMtestsNPRel};
	    \addplot table [x = {H}, y = {LinOSc3-Gal}] {\MsFEMtestsNPRel};
	    \addplot table [x = {H}, y = {LinOSc3-PG}] {\MsFEMtestsNPRel};
		\addplot table [x = {H}, y = {LinOSr3-Gal}] {\MsFEMtestsNPRel};
		\addplot table [x = {H}, y = {LinOSr3-PG}] {\MsFEMtestsNPRel};
	    %% Vertical line at x=epsilon
	    \draw [thick, black, densely dashed] ({axis cs:\eps,0}|-{rel axis cs:0,0}) -- ({axis cs:\eps,0}|-{rel axis cs:0,1});
	    \node[above left,black] at (axis cs: \eps,\epsheight) {\large $\varepsilon$};
	    \node[above,black,fill=white,draw=black] at (axis cs: \Htype,\epsheight) {\sf \footnotesize MsFEM-lin};

    \nextgroupplot[% Plot layout
	    %% Axis settings
	    xmax = \Hmax, xmin = \Hmin,
	    ymin = \emin, ymax = \emax,
	    xmode = log,
	    ymode = log,
	    xlabel = {$H$},
	    axis y line*=right, %line* to avoid an arrow on top
	    legend style={at={(1.05,1.3)},anchor=west}
	    ]
	    
	    %% Load data from macros defined in load_data.sty
	    \addplot table [x = {H}, y = {CR-Gal}] {\MsFEMtestsNPRel};
	    \addplot table [x = {H}, y = {CROSc3-Gal}] {\MsFEMtestsNPRel};
	    \addplot table [x = {H}, y = {CROSc3-PG}] {\MsFEMtestsNPRel};
		\addplot table [x = {H}, y = {CROSr3-Gal}] {\MsFEMtestsNPRel};
		\addplot table [x = {H}, y = {CROSr3-PG}] {\MsFEMtestsNPRel};
	    %% Vertical line at x=epsilon
	    \draw [thick, black, densely dashed] ({axis cs:\eps,0}|-{rel axis cs:0,0}) -- ({axis cs:\eps,0}|-{rel axis cs:0,1});
	    \node[above left,black] at (axis cs: \eps,\epsheight) {\large $\varepsilon$};
	    \node[above,black,fill=white,draw=black] at (axis cs: \Htype,\epsheight) {\sf \footnotesize MsFEM-CR};
	    
    \end{groupplot} 

	\node (per) at ($(my plots c1r1.center)!0.5!(my plots c2r1.center)+(0,.17\textwidth)$) {\Large Periodic coefficient~$A^{\varepsilon,\mathsf{per}}$};
	\node (locper) at ($(my plots c1r2.center)!0.5!(my plots c2r2.center)+(0,.17\textwidth)$) {\Large Locally periodic coefficient~$A^{\varepsilon,\mathsf{lp}}$};
	\node (nonper) at ($(my plots c1r3.center)!0.5!(my plots c2r3.center)+(0,.17\textwidth)$) {\Large Non-periodic coefficient~$A^{\varepsilon,\mathsf{np}}$};

\end{tikzpicture}
	\caption{
		Comparison of the errors of the (intrusive) Galerkin MsFEM~\eqref{eq:gen-MsFEM-OS} and the (non-intrusive) Petrov-Galerkin MsFEM~\eqref{eq:gen-MsFEM-OS-testP1} for the diffusion coefficients in~\eqref{eq:diffusion-test} as the mesh size~$H$ varies. Different oversampling strategies are applied: \DOF-continuous (OS\textsf{c}, Def.~\ref{def:corr-dofc}) and \DOF-extended (OS\textsf{e}, Def.~\ref{def:corr-dofe}). The Galerkin MsFEM without OS is included to illustrate the effect of the OS strategies.
	}
	\label{fig:MsFEM-tests}
\end{figure}

We consider next in Fig.~\ref{fig:MsFEM-tests} MsFEMs with the two different oversampling strategies of Sec.~\ref{sec:gen-framework}: \DOF-continuous and \DOF-extended oversampling. The PG-MsFEM variant, with or without oversampling, is completely equivalent to its non-intrusive implementation by virtue of Lemma~\ref{lem:gen-MsFEM-PG-eff}. With oversampling, however, it does not coincide with the (intrusive or non-intrusive) G-MsFEM. 

With oversampling, the matrices of the linear systems for the G-MsFEM and PG-MsFEM are different; Lemma~\ref{lem:IBP-bubbles-gen} does not apply. The result is that the differences $u_H^{\varepsilon,\mathsf{G}} - u_H^{\varepsilon,\mathsf{PG}}$ are larger than the differences $u_H^{\varepsilon,\mathsf{G}} - u_H^{\varepsilon,\text{\sf G-ni}}$. This is reflected in the numerical errors of the methods. We show the errors of the PG-MsFEM and the G-MsFEM with respect to the reference solution~$u^\varepsilon_h$ in Fig.~\ref{fig:MsFEM-tests}. (The non-intrusive G-MsFEM is too close to the G-MsFEM to be distinguishable on the scale of Fig.~\ref{fig:MsFEM-tests} for all MsFEMs except the MsFEM with \DOF-extended oversampling.) The G-MsFEM without oversampling is also shown to highlight the effect of oversampling. 

Let us first consider the two different oversampling strategies. For all Galerkin MsFEMs, it is clear that the \DOF-continuous variant performs (much) better than the \DOF-extended variant. For the Petrov-Galerkin MsFEMs, the difference between the two oversampling strategies is smaller, but the \DOF-continuous version of oversampling continues to perform better over all. 

Although clear differences in the performance of the Galerkin and Petrov-Galerkin MsFEMs with \DOF-continuous oversampling can be observed, these differences are small and both MsFEM approaches have a comparable accuracy. There is no systematic disadvantage in choosing the non-intrusive PG-MsFEM over the (intrusive or non-intrusive) G-MsFEM.
Moreover, the non-periodic test cases again show the robustness of all MsFEM variants when going beyond the setting of periodic homogenization. In particular, this demonstrates the robustness of the non-intrusive approaches for the MsFEM developed in this article.

%% Final declarations
\section*{Acknowledgments}
The first author acknowledges the support of DIM Math INNOV. The work of the second and third authors is partially supported by ONR under grant N00014-20-1-2691 and by EOARD under grant FA8655-20-1-7043. These two authors acknowledge the continuous support from these two agencies. The fourth author thanks Inria for the financial support enabling his two-year partial leave (2020-2022) that has significantly facilitated the
collaboration on this project.

%\todo[inline]{Fix ref \cite{brezzi_bint_1997, bastian_generic_2008,rutger_biezemans_2023_7525059}}
\printbibliography

\end{document}